\documentclass[a4paper,11pt,nohead,oneside]{amsart}

\usepackage{graphicx}
\usepackage{perpage}
\usepackage{url}
\usepackage{color}
\usepackage[T1]{fontenc}
\usepackage[a4paper,scale={0.72,0.78},marginratio={1:1},footskip=7mm,headsep=10mm]{geometry}

\usepackage[colorlinks=true, pdfstartview=FitV, linkcolor=blue, citecolor=blue, urlcolor=blue,pagebackref=false]{hyperref}

\usepackage{charter}

\usepackage{mathabx, stmaryrd, dsfont}

\usepackage[showlabels,sections,floats,textmath,displaymath]{}

\numberwithin{equation}{section}
\newtheorem{theorem}{Theorem}[section]

\newtheorem{proposition}[theorem]{Proposition}

\newtheorem{claim}[theorem]{Claim}

\newtheorem{hyp}[theorem]{Assumption}

\newtheorem{example}{Example}[section]

\renewenvironment{proof}[1][Proof]{\begin{trivlist}
\item[\hskip \labelsep {\bfseries #1}]}{\qed\end{trivlist}}

\DeclareMathOperator{\sign}{\mathrm{sign}}

\newcommand{\ind}{\mathds{1}}
\renewcommand{\ge}{\geq}
\renewcommand{\le}{\leq}

\DeclareMathSymbol{\leqslant}{\mathalpha}{AMSa}{"36} 
\DeclareMathSymbol{\geqslant}{\mathalpha}{AMSa}{"3E} 
\DeclareMathSymbol{\eset}{\mathalpha}{AMSb}{"3F}     
\renewcommand{\leq}{\;\leqslant\;}                   
\renewcommand{\geq}{\;\geqslant\;}                   
\newcommand{\dd}{\,\text{\rm d}}             

\newcommand{\sumtwo}[2]{\sum_{\substack{#1 \\ #2}}} 

\newcommand{\bP}{{\ensuremath{\mathbf P}} }
\newcommand{\bE}{{\ensuremath{\mathbf E}} }

\newcommand{\bX}{{\ensuremath{\mathbf X}} }

\newcommand{\bb}{{\ensuremath{\mathbf b}} }
\newcommand{\bd}{{\ensuremath{\mathbf d}} }
\newcommand{\bdelta}{{\ensuremath{\boldsymbol \delta}} }

\newcommand{\bbN}{{\ensuremath{\mathbb N}} }

\newcommand{\bbP}{{\ensuremath{\mathbb P}} }

\newcommand{\bbR}{{\ensuremath{\mathbb R}} }

\newcommand{\bbZ}{{\ensuremath{\mathbb Z}} }

\newcommand{\ga}{\alpha}
\newcommand{\gb}{\beta}
\newcommand{\gd}{\delta}
\newcommand{\gep}{\varepsilon}       
\newcommand{\gp}{\varphi}

\makeatletter
\def\captionfont@{\footnotesize}
\def\captionheadfont@{\scshape}

\long\def\@makecaption#1#2{%
  \vspace{2mm}
  \setbox\@tempboxa\vbox{\color@setgroup
    \advance\hsize-6pc\noindent
    \captionfont@\captionheadfont@#1\@xp\@ifnotempty\@xp
        {\@cdr#2\@nil}{.\captionfont@\upshape\enspace#2}%
    \unskip\kern-6pc\par
    \global\setbox\@ne\lastbox\color@endgroup}%
  \ifhbox\@ne 
    \setbox\@ne\hbox{\unhbox\@ne\unskip\unskip\unpenalty\unkern}%
  \fi
  \ifdim\wd\@tempboxa=\z@ 
    \setbox\@ne\hbox to\columnwidth{\hss\kern-6pc\box\@ne\hss}%
  \else 
    \setbox\@ne\vbox{\unvbox\@tempboxa\parskip\z@skip
        \noindent\unhbox\@ne\advance\hsize-6pc\par}%
\fi
  \ifnum\@tempcnta<68 
    \addvspace\abovecaptionskip
    \moveright 3pc\box\@ne
  \else 
    \moveright 3pc\box\@ne
    \nobreak
    \vskip\belowcaptionskip
  \fi
\relax
}
\makeatother
\def\writefig#1 #2 #3 {\rlap{\kern #1 truecm
\raise #2 truecm \hbox{#3}}}

\newcommand{\tif}{\text{ if }}

\newcommand{\x}{\mathbf{x}}
\newcommand{\y}{\mathbf{y}}

\newcommand{\z}{\mathbf{z}}
\newcommand{\bga}{\boldsymbol{\alpha}}
\newcommand{\bS}{\mathbf{S}}
\newcommand{\bZ}{\mathbf{Z}}

\begin{document}

\title[Multivariate strong renewal theorems and local large deviations]{Strong renewal theorems and local large deviations\\
 for multivariate random walks and renewals}
\author{Quentin Berger}

\address{Q. Berger, Sorbonne Universit\'e, LPSM\\
Campus Pierre et Marie Curie, case courrier 158\\
4 place Jussieu, 75252 Paris Cedex 5, France}
\email{quentin.berger@sorbonne-universite.fr}
\date{}

\begin{abstract}
We study a random walk $\mathbf{S}_n$ on $\mathbb{Z}^d$ ($d\geq 1$), in the domain of attraction of an operator-stable distribution with index $\boldsymbol{\alpha}=(\alpha_1,\ldots,\alpha_d) \in (0,2]^d$: in particular, we allow the scalings to be different along the different coordinates.  We prove a strong renewal theorem, \textit{i.e.}\ a sharp asymptotic of the Green function $G(\mathbf{0},\mathbf{x})$ as $\|\mathbf{x}\|\to +\infty$, along the ``favorite direction or scaling'': (i) if $\sum_{i=1}^d \alpha_i^{-1} < 2$ (reminiscent of Garsia-Lamperti's condition when $d=1$ \cite{cf:GL}); (ii) if a certain \emph{local} condition holds (reminiscent of Doney's \cite[Eq. (1.9)]{cf:Don97} when $d=1$).
We also provide uniform bounds on the Green function $G(\mathbf{0},\mathbf{x})$, sharpening estimates when $\mathbf{x}$ is away from this favorite direction or scaling.
These results improve significantly the existing literature, which was mostly concerned with the case $\ga_i\equiv \ga$, in the favorite scaling, and has even left aside the case $\ga\in[1,2)$ with non-zero mean.
Most of our estimates rely on new general (multivariate) local large deviations results, that were missing in the literature and that are of interest on their own.
\end{abstract}



\thanks{The author acknowledges the support of grant ANR-17-CE40- 0032-02 of the French National Research Agency.}

\maketitle

\tableofcontents

\section{Setting of the paper}

\subsection{Multivariate random walks, domains of attraction}
We consider a \emph{$d$-dimensional} random walk $\bS = (\bS_n)_{n\geq 0}$: $\bS_0 = \mathbf{0}$, and  $\bS_n:=\sum_{j=1}^n \bX_j$, where $(\bX_j)_{j\geq 0}$ is an i.i.d.\ sequence of $\bbZ^d$-valued random variables (we treat only the case of a lattice distribution for the simplicity of exposition, but non-lattice counterparts should hold).
We assume that $\bX_1$ is non-defective, \textit{i.e.}\ $\bP(\|X_1\| <+\infty) = 1$ (let $\|\cdot\|$ denote the $L^1$ norm).
If $\bX_1\in \bbN^d$, we then call $\bS_n$ a \emph{multivariate renewal process}, and $\bS=\{\bS_0, \bS_1, \bS_2, \ldots\}$ is interpreted as a random subset of $\bbN^d$ (with a slight abuse of notations).

We assume that $\bS$ is aperiodic and in the domain of attraction of a non-degenerate multivariate stable distribution with index $\bga:=(\ga_1,\ldots,\ga_d) \in(0,2]^d$: there is a \emph{recentering sequence} $\mathbf{b}_n =(b_n^{(1)}, \ldots, b_n^{(d)})$ and \emph{scaling sequences} $a_n^{(1)}, \ldots, a_n^{(d)}$ such that, setting $A_n$ the diagonal matrix with $A_n(i,i)= a_n^{(i)}$, we have as $n\to +\infty$
\begin{equation}
\label{cf:stableconv}
A_n^{-1} (\bS_n -\mathbf{b}_n)= \Big(\frac{S_n^{(1)}- b_n^{(1)}}{a_n^{(1)}} , \ldots, \frac{S^{(d)}_n - b_n^{(d)}}{a_n^{(d)}} \Big)  \Rightarrow  \bZ \quad \text{ in distribution}.
\end{equation}
Here, $\bZ$ is a multivariate stable law, whose non-degenerate density is denoted $g_{\bga}(\x)$.
As in \cite{cf:RG79,cf:Don91, cf:Meer91}, we allow the scaling sequences to be different along different coordinates. 
The case where $a_n^{(i)} \equiv a_n$ for all $1\leq i\leq d$ (that is $A_n= a_n {\rm I}_d$) was considered by L\'evy \cite{cf:Levy} and Rvaceva \cite{cf:Rva62}, and will be referred to as the \emph{balanced} case.
We refer to Appendix~\ref{appA} for further discussion on \emph{generalized} domains of attractions (here we only consider the case where  $A_n$  is diagonal), and for a brief description of multivariate regular variation.



\subsection{First notations}

For every $i\in\{1,\ldots,d\}$, $S_n^{(i)}$ has to be in the domain of attraction of a $\ga_i$-stable distribution.
Let us set $F_i (x) := \bP(X_1^{(i)} \le x) $ and $\bar F_i (x) := \bP(X_1^{(i)} >x)$.

When $\ga_i \in(0,2)$, there exist some slowly varying function $L_i(\cdot)$, and constants $p_i,q_i \geq 0$ (with $p_i+q_i=1$) such that
\begin{equation}
\label{hyp:XY}
\bar F_i (x) \sim   p_i L_i(x) x^{-\ga_i}  \quad \text{and} \quad  F_i (-x) \sim   q_i L_i(x) x^{-\ga_i} \quad \text{ as } x\to+\infty\, ,
\end{equation}
 and when $p_i=0$ or $q_i=0$, we interpret this as $o(L_i(x)x^{-\ga_i})$. Note that \eqref{hyp:XY} is equivalent to $S_n^{(i)}$ being in the domain of attraction of an $\ga_i$-stable law, $\ga_i\in(0,2)$, see \cite[IX.8, Eq. (8.14)]{cf:Feller}.
When $\ga_i =2$, then we set
\begin{equation}
\label{def:sigmax}
\sigma_i(x):=\bE \big[ \big(X^{(i)}_1\big)^2 \ind_{\{| X_1^{(i)} | \leq x\}} \big] \, .
\end{equation}
By \cite[IX.8, Thm.~1]{cf:Feller}, having $\sigma_i(x)$ slowly varying is equivalent to $S_n^{(i)}$ being in the domain of attraction of the normal distribution.

The scaling sequence $a_n^{(i)}$ is then characterized by the following relation
\begin{equation}
\label{def:an}
\begin{split}
L_i(a_n^{(i)}) (a_n^{(i)})^{-\ga_i} &\sim  1/n \quad \text{as } n\to +\infty, \ \tif \ga_i\in(0,2) ;\\
\sigma_i(a_n^{(i)}) (a_n^{(i)})^{-2} &\sim 1/n  \quad \text{as } n\to +\infty, \  \tif \ga_i=2.
\end{split}
\end{equation}
Note that in any case, $a_n^{(i)}$ is regularly varying with exponent $1/\ga_i$.

Regarding the recentering sequences $b_n^{(i)}$,  we set (see \cite[IX.8, Eq.~(8.15)]{cf:Feller}): 
\begin{equation}
\label{def:bn}
b_n^{(i)} \equiv 0 \ \ \tif \ga_i\in(0,1) ; \quad b_n^{(i)} := n\mu_i \ \ \tif \ga_i>1 ; \quad b_n^{(i)}=n \mu_i(a_n^{(i)}) \ \ \tif \ga_i=1.
\end{equation}
We defined $\mu_i:=\bE[X^{(i)}_1]$ when $X_1^{(i)}$ is integrable, and $\mu_i(x) := \bE[ X_1^{(i)} \ind_{\{|X_1^{(i)}|\leq x\}}]$.

\subsection{Overview of the literature and of our results}

The main focus of our paper is the behavior of the Green's function $G(\mathbf{0},\x)=G(\x) := \sum_{n=1}^{+\infty} \bP(\bS_n = \x)$, as $\|\x\| \to+\infty$. 
The literature is vast in the case of dimension $d=1$, see e.g.\ \cite{cf:GL,cf:Erik,cf:Don97} or \cite{cf:CD16} for some landmarks. It has also been studied in a variety of papers in the case of dimension $d\geq 2$, but only in the balanced case ($\ga_i\equiv \ga$), and in some specific cases. Let us now present an overview of the conditions under which the asymptotic behavior of $G(\x)$ is known ($d\geq 2$):

\smallskip
$\ast$  In the case $\ga=2$ (Normal domain of attraction), with non-zero mean: with some moment conditions and along the correct angle $\x =(t, \lfloor t \mu_2/\mu_1\rfloor)$, see \cite{cf:NS} (this has been improved in \cite{cf:Don66} and \cite{cf:Stam69}),
with an exponential tail condition,  in a small cone around the mean vector, see~\cite{cf:CW84}. 
Some estimates away from the favorite directions are provided in \cite[Lem.~5]{cf:Nag80}, under a zero mean, finite variance condition.

$\ast$ For $\ga\in(0,2)$, in the \emph{centered} case (\textit{i.e.}\ $b_n^{(i)} \equiv 0$): if $d/2<\ga<2$ and along a given angle, see \cite{cf:Will68}; if $\ga\in(0,1)$ and along a given angle, with an additional \emph{local} condition, see \cite[Cor.~3.B]{cf:Will68}.
This has also been proven more recently in \cite{cf:C14} under an integro-local condition.
We also mention \cite[Prop.~26.1]{cf:Spitzer} and \cite{cf:U98} for simple moment conditions to obtain the asymptotic behavior of $G(\x)$, in the case $\ga=2$.

 \smallskip
 The contribution of the present paper is threefold:
 (i) we give the sharp behavior of $G(\x)$ in the case $\ga\in [1,2)$ with non-zero mean, in a cone around the mean vector (we call it favorite direction): this was missing in the literature---we also treat the case $\ga=1$ with infinite mean;
 (ii) we give uniform bounds on $G(\x)$, giving improved estimates when $\x$ is outside the favorite direction;
 (iii) we extend the results to the case of random walks in the domain of attraction of an operator stable distribution, allowing for different scalings along the different components (and we weaken Williamson's condition~\cite[Eq.~(3.10)]{cf:Will68} in the case $\ga\in(0,1)$).

As a central tool, we prove some multivariate local large deviations estimates, \textit{i.e.}\ we go beyond the local limit theorem in a large deviation regime. This is of its own interest since such estimates were missing in the literature, and appear central in controlling the small-$n$ contribution to $G(\x)$. We prove a local large deviation in the general setting, see Theorem~\ref{thm:locallimit1}. Then we propose a new (and natural) multivariate Assumption~\ref{hyp:2}, which extends Doney's condition \cite[Eq.~(1.9)]{cf:Don97} to the multivarate settind, and generalizes Williamson's condition \cite[Eq.~(3.10)]{cf:Will68}: we obtain a better local large deviation result under this assumption. 

\smallskip
Let us now give a brief overview on how the rest of the paper is organized. First, we present our local large deviations estimates and our Assumption~\ref{hyp:2} (that gives a sharper result), in Section~\ref{sec:locallimit}. In Section~\ref{sec:diag}, we state our strong renewal theorems (along the favorite direction or scaling), that we divide into three parts: the centered case, \textit{i.e.}\ when $\bb_n\equiv 0$; the non-zero mean case with $\ga_i>1$; the case $\ga_i=1$, that we set aside because it needs additional care.
In Section~\ref{sec:away}, we present the uniform bounds on $G(\x)$ (in dimension $d=2$ for simplicity).
The rest of the paper, Sections~\ref{sec:prooflocallimit} to \ref{sec:away}, is devoted to the proofs: Section~\ref{sec:prooflocallimit} for the local large deviations, Sections~\ref{sec:casI}-\ref{sec:casII}-\ref{sec:casIII} for the strong renewal theorems, and Section~\ref{sec:away} for the estimates when $\x$ is away from  the favorite direction or scaling.
Finally, we collect in the appendix some useful comments: in Appendix~\ref{appA}, we recall some definitions and results about multivariate regular variation and generalized domains of attraction;
in Appendix~\ref{appB}, we discuss further on our Assumption~\ref{hyp:2}.

\subsection{A general working assumption}

We assume in the rest of the paper, mostly for simplicity of notations, that the left and right tail distributions of $\bX_1$, $F_i(-x)$ and $\bar F_i(x)$, are dominated by subexponential distributions.

\begin{hyp}
\label{hyp:1}
There exists some slowly varying functions $(\gp_i)_{i\leq d}$, and some $\gamma_i\ge \alpha_i$ such that for all $x\in\bbN$ and  $i\in\{1,\ldots, d\}$
\begin{equation}
\label{hyp:XY2}
 F_i(-x) + \bar F_i(x) := \bP(X_1^{(i)} \le -x)  + \bP(X_1^{(i)}>x) \leq  \gp_i(x) x^{-\gamma_i} .
\end{equation}
When $\bE[(X^{(i)}_1)^2]=+\infty$, we may take $\gamma_i  =\ga_i$, and $\gp_i(\cdot)$ a constant multiplicative of $L_i(\cdot)$.\\
When $\bE[(X_1^{(i)})^2]<+\infty$, we may take $\gp_i(\cdot)$ and $\gamma_i$ such that $\sum_{n\geq 1} \gp_i(n) n^{1-\gamma_i} <+\infty$.
\end{hyp}

This assumption is essentially used to generalize \eqref{hyp:XY} to the case $\ga_i=2$: the exponent $\gamma_i$ gives further information on the left and right tail distribution. It does not appear to be a real restriction (components are allowed to have a much stronger tail, having formally $\gamma_i=\infty$), but is easier for presenting the results.
Also, we used the same exponent for the left and right tail distribution for simplicity, but all results can be adapted to the case of different tail behaviors.
A typical example we have in mind is when the distribution of $\bX_1$ is regularly varying in $\bbR^{d}$ with exponent $-(\gamma_1, \ldots, \gamma_d)$. We refer to Appendix~\ref{appA} for a definition of multivariate regular variation, see in particular~\eqref{def:regvarmeas}---we also present two examples (Examples~\ref{ex:indep}-\ref{ex:depend}) of distribution of $\bX_1$ we keep in mind.

\section{Local large deviations}
\label{sec:locallimit}

Let us start by stating the local limit theorem obtained by Griffin in \cite{cf:Grif} in our setting, and disentangled by Doney \cite{cf:Don91} (it is proven in dimension $2$, but as stressed by Doney its proof is valid in any dimension): uniformly for $\x =(x_1,\ldots,x_d)\in\bbZ^d$,
\begin{equation}
\label{LLT}
 a_n^{(1)} \cdots a_n^{(d)} \bP\big( \bS_n=\x \big) -  g_{\bga}\big( \x_n \big)  \to 0  \quad \text{as } n\to +\infty \, ,
\end{equation}
with $\x_n:= A_n^{-1} (\x - \mathbf{b}_n)=\big(\frac{x_1 -b_n^{(1)}}{a_n^{(1)}} ,\ldots, \frac{x_d -b_n^{(d)}}{a_n^{(d)}}\big) .$

Our first set of results  concerns local large deviation estimates, which improve~\eqref{LLT} in the case $\| \x_n\| \to +\infty$. But let us start by reviewing some of the existing literature. A great part of it focuses on the balanced case ($A_n = a_n {\rm I}_d$):
in \cite{cf:HLMS05}, large deviations are proven, and in \cite{cf:Zai99,cf:NZ05}, some sufficient conditions (that we do not detail here) are given to obtain a local limit theorem of the type $\bP(\bS_n\in A) \sim n \bP(\bX_1 \in A)$ ---the case $\ga=1$ is left aside. 
As far as the ``non-balanced'' case is concerned, we refer to  \cite[Ch. 9]{cf:MSbook} for large deviations estimates, see for example Theorem~9.1.3, where it is shown that $\bP \big( \langle \bS_n, \theta \rangle > x_n \big) $ is of the order of $n \bP \big( \langle \bX_1 ,\theta \rangle >  x_n \big)$ when in the domain of attraction of an operator stable distribution with no normal component.

To summarize, there exists no general result that would treat ``mixed'' Normal and stable cases, and that would give a good (and general) local large deviation, under a weak assumption. Our aim is therefore to provide simple local large deviation estimates, that will be a crucial tool  for our renewal results of Sections~\ref{sec:diag}-\ref{sec:resultsaway}.
 We also give an improved result below, under some more local assumption on the distribution of $\bX_1$. The proof of the local large deviation results are presented in Section~\ref{sec:prooflocallimit}.

\subsection{A first local limit theorem}

Let us denote $\hat \bS_n := \bS_n - \lfloor \mathbf{b}_n \rfloor$ the \emph{recentered} walk (we take the integer part of $\mathbf{b}_n$ simply so that $\hat \bS_n$ is still $\bbZ^d$ valued).
As far as a large deviation estimate is concerned, univariate large deviation estimates already give (we recall these results in Section~\ref{sec:univariate} below) that there is a constant $C_0$ such that for any $\x\geq 0$, 
\begin{equation}
\label{eq:largedevmulti}
\bP\big( \hat  \bS_n \ge \x \big)  \leq C_0 \min_{i\in\{1,\ldots d\}} \Big\{ n \gp\big(x_i \big) \, \big( x_i \big)^{-\gamma_i}   +  \exp\big( -   \tfrac{c x_i^2 }{n \sigma(x_i)} \big) \ind_{\{\ga_i= 2\}}\Big\}\, ,
\end{equation}
where the inequality $\hat \bS_n \ge \x $ is componentwise.
We now give a local version of it.

\begin{theorem}
\label{thm:locallimit1}
Assume that Assumption \ref{hyp:1} holds.
There exist constants $c_1,C_1$ such that for any fixed $i\in\{1,\ldots, d\}$ and $\x \in \bbZ$ with $|x_i|\ge a_n^{(i)}$, we have
\begin{equation}
\label{eq:localCD}
a_n^{(1)} \cdots a_n^{(d)} \times\bP\big( \hat \bS_n = \x \big) \leq  
C_1  n \gp_i ( |x_i| )\,  |x_i|^{-\gamma_i} +  C_1 \exp\big( -  \tfrac{c_1 |x_i|^2}{n \sigma_i(|x_i|)} \big)  \ind_{\{\ga_i=2\}}.
\end{equation}
\end{theorem}

\noindent
The idea of this result is similar to that of \cite[Theorem 1.1]{cf:CD16} for the univariate case (where only the case $\ga\in(0,1)\cup(1,2)$ is treated), and we give the details  in Section~\ref{sec:simplelocal}.

\subsection{A \emph{local} multivariate assumption for an improved local limit theorem}

In dimension $d=1$, better local large deviations can be  obtained under a \emph{local} assumption on the distribution of $\bX_1$, see \cite{cf:Don97} in the case $\ga\in(0,1)$ and \cite[Thm~2.7]{cf:B17} in the case $\ga\in (0,2)$.
We present here an assumption which can be thought as the analogous of Doney's condition \cite[Eq.~(1.9)]{cf:Don97} to the multivariate setting, and generalizes Williamson's condition \cite[Eq.~(3.10)]{cf:Will68}. We comment on that Assumption below.


\begin{hyp}
\label{hyp:2}
There exist a constant $\mathrm{C}_d$, slowly varying functions $(\varphi_i)_{1\leq i \leq d}$ and exponents $(\gamma_i)_{1\le i\le d}$ (the same as in Assumption \ref{hyp:1})  such that for any fixed~$i\in\{1,\ldots, d\}$
\begin{align}
\label{eq:hyp2}
\bP(\bX_1 = \x) &\leq  \frac{ {\rm C}_d \gp_i(|x_i|) \big( 1+|x_i| \big)^{-\gamma_i}}{ \prod_{j=1}^d (1+|x_j|)} \times \prod_{j\ne i} h^{(i)}_{|x_{i}|} (|x_j|) \, ,
\end{align}
where the functions $h^{(i)}_u(v)$  ($u,v\in \bbN$) for $i\in\{1,\ldots, d\}$ verify:
\begin{equation}
\label{cond:h}
\text{(i)} \   h^{(i)}_u(v) \le 1  \quad ; \quad \text{(ii)} \  \sup_{u\ge 0}  \sum_{v\ge 0} \frac{h^{(i)}_u(v)}{1+|v|} <+\infty \quad  ; \quad 
\text{(iii)} \  \sup_{u ,v \ge 0 \, ;\,  u'\in[u,2u]}  \frac{ h^{(i)}_{u'}(v)}{h^{(i)}_u(v)} <+\infty \, .
\end{equation}
\end{hyp}
First of all, we present two important examples that verify Assumption \ref{hyp:2}: they are local versions of Example~\ref{ex:indep} (independent case) and Example~\ref{ex:depend} (dependent case). 

\begin{example}
\label{ex1}
There are positive exponents $\gamma_i$ and slowly varying functions $\gp_i(\cdot)$ ($i\in\{1,\ldots,d\}$), such that
$\bP(\bX_1 = \x) = \prod_{i=1}^d \gp_i(x_i) \,  x_i^{-(1+\gamma_i)}$, for $\x\in \bbN^d$.
\end{example}

\begin{example}
\label{ex2}
There are positive exponents $\beta,(\beta_i)_{1\le i\le d}$ with $\gb> \sum_{i=1}^d \gb_i^{-1}$, and $\psi(\cdot)$ a slowly varying function, such that
$\bP(\bX_1= \x) =  \psi\big( \sum_{i=1}^d  x_i^{\beta_i}\big) \times  \big( \sum_{i=1}^d  x_i^{\beta_i}\big)^{-\beta}$, for $ \x\in\bbN^d$.\\
Assumption~\ref{hyp:2} is verified with $\gamma_i :=\gb_i \big( \gb - \sum_{i=1}^d \gb_i^{-1} \big)$, see Appendix~\ref{appB}.
\end{example}
\noindent
We mention that a two-dimensional, \emph{balanced}, version of Example~\ref{ex2} is used in~\cite{cf:GK} (it comes from the biophysics literature, see~\cite{cf:GO}):  the dimension is $d=2$, $\beta_i\equiv 1$, and $\beta=2+\alpha$, $\ga>0$.

\smallskip
Let us now give a general idea behind the choice of Assumption~\ref{hyp:2}---assume for simplicity that all $x_i$'s are positive.
We start with writing
\begin{align*}
\bP(\bX_1 = \x)  = \bP&\big( \bX_1 = \x \ \, \big|\,  X_1^{(i)} \in[x_i,2x_i] \ \forall i\in \{1,\ldots, d\} \big) \\
& \quad \times \bP\big(X_1^{(i)} \in [x_i, 2x_i] \big)  \times \bP\big( X_1^{(j)} \in [x_j, 2x_j]\  \forall j\neq i  \, \big|\, X_1^{(i)} \in [x_i, 2x_i]  \big)\, .
\end{align*}
First, conditioned on the event that $\bX_1$ is in the rectangle $ [x_1,2x_1] \times \cdots \times [x_d,2x_d]$, a natural assumption is that the probability of being at one particular site is bounded by $c \big( \prod_{i=1}^d x_i \big)^{-1}$ (\textit{i.e.}\ uniform on the rectangle): this gives the first denominator of \eqref{eq:hyp2}.
 Then, $\bP(X_1^{(i)} \in [x_i,2x_i] )$ is  bounded by a constant times $\gp(x_i) x_i^{-\gamma_i}$ by Assumption~\ref{hyp:1}: it gives the first numerator in \eqref{eq:hyp2}.
 The last term is, by H\"older's inequality, bounded by
\[ \prod_{j\neq i}\bP\big( X_1^{(j)} \in [x_j, 2x_j]  \, \big|\, X_1^{(i)} \in [x_i, 2x_i]  \big)^{1/(d-1)} \, ,  \]
which accounts for the product of the $h_{x_i}^{(i)}(x_j)$. 
We keep in mind two cases: (i) when the coordinates are independent (see Example~\ref{ex1}), we recover $h_{x_i}^{(i)}(x_j) \le  x_j^{- a}$ for some $a>0$; (ii) when the coordinates are dependent (see Example~\ref{ex2}), there is some threshold $t(x_i)$ such that $h^{(i)}_{x_i} (x_j) \le  \big( \frac{x_j}{t(x_i)} \vee \frac{t(x_i)}{x_j} \big)^{-a}$ for some $a>0$, and this  satisfies the conditions~\eqref{cond:h} (we refer to Appendix~\ref{appB} for more details, see~\eqref{rewriteex2}-\eqref{realh} and below).

\smallskip
We stress that the term $h^{(i)}_{|x_i|}(|x_j|)$ in \eqref{eq:hyp2} is central: in particular, item (ii) in~\eqref{cond:h} insures that there is a constant $\mathrm{C}$ such that for any $i$,
\begin{equation}
\label{hyp:Doney1}
\bP(X_1^{(i)} = x_i) \le \mathrm{C}\, \gp_i(|x_i|) (1+|x_i|)^{-(1+\gamma_i)}\, ,
\end{equation}
which is Doney's condition \cite[Eq.~(1.9)]{cf:Don97} for each component (generalized to the case $\ga_i\ge 1$).
Also, we point out that Assumption~\ref{hyp:2} is similar in spirit but \emph{weaker} than Williamson's condition \cite[Eq.~(3.10)]{cf:Will68}, which considers the balanced case $\ga_i \equiv \ga < \min( d,2)$, and says that there is a constant $K_0<+\infty$ such that for any $\x \in \bbZ^d$,
\begin{equation}
\label{hyp:Williamson}
\bP(\bX_1 = \x)  \leq K_0 \, ( 1+\| x\|)^{-d}  \bP\big( \|\bX_1\|> \| \x \| \big) \, .
\end{equation}
(\eqref{hyp:Williamson} does not include the case of independent $X^{(i)}$'s, whereas our Assumption~\ref{hyp:2} does.)


Under Assumption~\ref{hyp:2}, we are able to improve Theorem~\ref{thm:locallimit1}.

\begin{theorem}
\label{thm:locallimit2}
Suppose that Assumption \ref{hyp:2} holds.
Then there are constants $c_2,C_2$ such that for any $\x \in \bbZ^d$
\begin{align*}
\bP\big( \hat  \bS_n = \x \big) \leq  \frac{C_2}{\prod_{i=1}^d  \max\{|x_i|, a_n^{(i)}\}} \times
\min_{i\in\{1,\ldots d\}} \Big\{ n \gp_i\big( |x_i| \big) \,  |x_i|^{-\gamma_i}   +  e^{ - c_2 ( |x_i|/ a_n^{(i)} )^2  }\ind_{\{\ga_i= 2\}} \Big \}  \, .
\end{align*}
\end{theorem}
The case of dimension $d=1$ with $\ga_1\in (0,2)$ is proven in \cite[Theorem~2.7]{cf:B17}: Theorem~\ref{thm:locallimit2} therefore generalizes it to the case $\ga_1=2$, and to the multivariate, non-balanced case.
It is a significant improvement of Theorem~\ref{thm:locallimit1}, in particular when (several)  $x_i$'s are much larger than $a_n^{(i)}$.

\subsection{About the balanced case, and Williamson's condition}
We may obtain another bound if we consider the balanced case, and assume that there is a positive exponent~$\gamma$, and some slowly varying $\gp(\cdot)$ such that
\begin{equation}
\label{eq:Will}
\bP(\bX_1=\x) \leq \gp(\|\x\|) \| \x \|^{-(d+\gamma)}\, .
\end{equation}
This is a natural extension of Williamson's condition~\eqref{hyp:Williamson} to the case $\ga=2$, and as seen in Appendix~\ref{appB} (when treating Example~\ref{ex2}), it implies Assumption~\ref{hyp:2}. 

\begin{theorem}
\label{thm:locallimit3}
Suppose that $a_n^{(i)}\equiv a_n$ (balanced case) and that \eqref{eq:Will} holds.
Then there are constants $c_3,C_3$, such that  for $\|\x\| \geq a_n$ we have
\[
\bP\big( \hat \bS_n =\x \big) \leq  C_3  n \gp(\|\x\|)\, \|\x\|^{-(d+\gamma)} + \frac{1}{(a_n)^d}\, e^{ - c_3 ( \|\x\|/a_n)^2} \ind_{\{\ga =2\}}\, .
\]
\end{theorem}
In practice, we will not use assumption~\eqref{eq:Will} in the rest of the paper: it requires to work in the balanced case, and would not improve our renewal results. We however include Theorem~\ref{thm:locallimit3} since it is an important improvement of Theorem~\ref{thm:locallimit2}, and may reveal useful (in particular in the setting of \cite{cf:GK} and \cite{cf:BGKiid} where \eqref{eq:Will} is verified).

\subsection{Some conventions for the rest of the paper}
\label{sec:convention}
First of all, all regularly varying quantities ($a_n^{(i)},b_n^{(i)}, L_i(\cdot),  \mu_i(\cdot), \gp_i(\cdot)$...) will be interpreted as  functions of positive real numbers, which may be taken infinitely differentiable (see~\cite[Th. 1.8.2]{cf:BGT}).

As we may work along subsequences and exchange the role of the $X^{(i)}$'s, we  assume that $a_n^{(1)} \le  \cdots \le  a_n^{(d)}$ (insuring in particular that $\ga_1\ge \cdots \ge \ga_d$)---the first coordinate is the one with the less fluctuations. 
Finally, assume that $a_n^{(j)}/a_n^{(i)} \to a_{i,j}\in \{0,1\}$ for $j\leq i$ (if $a_n^{(i)}/a_n^{(j)} \to a\in(0,1)$ then rescale the limiting law by~$a$). Having $a_{i,1}=1$ for all $i$ corresponds to the \emph{balanced} case.
We will also assume that: either $b_n^{(i)} \equiv 0$ (as it is the case when $\ga_i<1$; $\ga_i>1$ with $\mu_i=0$;  in the symmetric case for $\ga_i=1$), or that $b_n^{(i)}/a_n^{(i)} \to \pm\infty$ (as it is the case when $\ga_i\geq 1$ with $\mu_i \in \bbR^*$ or $\ga_i=1$ with $p_i\neq q_i$)---the only case where subtleties may arise is when $\ga_i=1$ with $|\mu_i|=0$ or $+\infty$ and $p_i=q_i$. (If $b_n^{(i)}/a_n^{(i)} \to b_i \in \bbR$, then we can reduce to the case $b_n^{(i)}\equiv 0$, at the expense of a translation of the limiting law.)

In the rest of the paper, we denote $u\vee v = \max(u,v)$ and $u\wedge v = \min (u,v)$. For two sequences $(u_n)_{n\ge 0}$, $(v_n)_{n\geq 0}$, we write $u_n \sim v_n$ is $u_n/v_n \to 1$ as $n\to+\infty$,  $u_n= O(v_n)$ if $u_n/v_n$ stays bounded, and $u_n \asymp v_n$ if $u_n= O(v_n)$ and $v_n = O(u_n)$.

\section{Strong renewal theorems}
\label{sec:diag}

We now consider the Green function
$G(\x):=\sum_{n =1}^{\infty} \bP(\bS_n = \x)$,
and we study its behavior as $\| \x\|\to+\infty$.
If $(\bS_n)_{n\geq 0}$ is a (multivariate) renewal process, 
we interpret $G(\x)$ as the renewal mass function, $\bP(\x\in\bS)$.

\subsection{About the favorite direction or scaling}
In the sum $\sum_{n =1}^{\infty} \bP(\bS_n = \x)$, the main contribution comes from some typical number of jumps: identifying that number allows us to determine a favorite direction or scaling along which we will get sharp asymptotics of $G(\x)$.
Let us define $n_i := n_i(\x)$ for $i \in\{ 1, \ldots, d\}$ by the relation
\begin{equation}
\label{def:ni}
 \begin{split}
b_{n_i}^{(i)} &= x_i  \quad \ \ \tif  |b_{n_i}^{(i)}|/a_n^{(i)} \to +\infty \quad (\text{$b_n$ and $x_i$ need to have the same sign})\, , \\
a_{n_i}^{(i)} &= |x_i|  \quad \tif   b_n^{(i)} \equiv 0 \, .
 \end{split}
\end{equation}
Then $n_i$ is the typical number of steps for the $i$-th coordinate to reach $x_i$.
This definition might not give a unique $n_i$, but any choice will work, and $n_i$ is unique up to asymptotic equivalence. If $\ga_i >1$ with $\mu_i \neq 0$, then we have $n_i = |x_i | /|\mu_i|$; if $\ga_i =1$ and $\mu_{i}\in \bbR^*$ or $\ga_i=1$ and $p_i\neq q_i$ then we have $n_i \sim |x_i|/|\mu_i(|x_i|)|$ (see details below, in Section~\ref{sec:casIII-prelim}); and if  $b_n^{(i)}\equiv 0$ then $n_{i}\sim x_{i}^{-\ga_{i}} \phi_{i}(x_{i})^{-1}$  with $\phi_{i}= L_{i}$ if $\ga_{i}\in (0,2)$ and $\phi_{i}=\sigma_{i}$ if $\ga_{i}=2$, thanks to the definition~\eqref{def:an} of $a_n^{(i)}$.

There are mainly three regimes that we consider, 
\begin{enumerate}
\item[I.] Centered case: $\bb_n \equiv \mathbf{0}$. The typical number of steps to reach $\x$ is $n_{i_0} = \min_i n_i$; the \textit{favorite scaling} are the points $\x$ with $x_i \asymp a_{n_{i_0}}^{(i)}$ for all~$i$, see \eqref{def:favdir2} below.
\item[II.] Non-zero mean case: $\mu_i\in \bbR^*$ for some $i$, with $\ga_i>0$. Let  $i_0 = \min\{i , \mu_i\neq 0\}$: the typical number of steps to reach $\x$ is  $n_{i_0} + O(a_{n_{i_0}}^{(i_0)})$;  the \emph{favorite direction} are the points $\x$ with $x_i = b^{(i)}_{n_{i_0}} + O(a_{n_{i_0}}^{(i)})$ for all $i$, see \eqref{def:favdir} below.
\item[III.] Case $\ga_{i_0}=1$, where $i_0 = \min\{i, b_n^{(i)}\not\equiv 0\}$. Assume that either $\mu_{i_0}\in\bbR^*$ or $p_{i_0}\neq q_{i_0}$. The typical number of steps to reach $\x$ is $n_{i_0} + O(m_{i_0})$ with $m_{i_0}:=a_{n_{i_0}}^{(i_0)} / |\mu_{i_0} (a_{n_{i_0}}^{(i_0)}) |$ (see Section~\ref{sec:casIII-prelim}); the \textit{favorite direction} are the points $\x$ with $x_i = b^{(i)}_{n_{i_0}} + O(a_{n_{i_0}}^{(i)})$ for all~$i$, see \eqref{def:favdir3} below. Some more subtleties arise in that case.
\end{enumerate} 

We now present strong renewal theorems, \textit{i.e.}\ sharp asymptotics of $G(\x)$, in cases I-II-III, along the favorite direction or scaling (the proofs are presented in Sections~\ref{sec:casI}-\ref{sec:casII}-\ref{sec:casIII}).
Recall that $g_{\bga}(\cdot)$ is the density of the limiting multivariate stable law.

\subsection{Case I (\emph{centered}): $\mathbf{b}_n \equiv \mathbf{0}$}

We assume here that $\mathbf{b}_n\equiv \mathbf{0}$, and that $\sum_{i=1}^d \ga_i^{-1} >1$, so that $\sum_{n\geq 1} (a_n^{(1)} \cdots a_n^{(d)})^{-1} <+\infty$, and $\bS_n$ is transient.
We leave aside for the moment the case $d=1$, $\ga_1=1$ (considered in \cite{cf:B17}), and the case $d=2$, $\bga=(2,2)$, which are marginal cases---the transience of the random walk depends on the slowly varying functions $L_i(\cdot)$.

\begin{theorem}
\label{thm:ga<1}
Suppose $\bb_n \equiv \mathbf{0}$ and $\sum_{i=1}^d \ga_i^{-1} >1$, and that (i) $\sum_{i=1}^{d} \ga_i^{-1}<2$ or (ii) Assumption~\ref{hyp:2} holds. Recall the definition~\eqref{def:ni} of $n_i$.
 If $\| \x \|\to +\infty$ such that for all~$1\leq i\leq d$
\begin{equation}
\label{def:favdir2}
x_{i} / a_{n_{1}}^{(i)} \to t_i\in \bbR^* \qquad \text{ as } |x_1| \to +\infty \quad  (t_1=\sign(x_1))\, ,
\end{equation}
then we have that, 
\begin{equation}
\label{casIIfavorite}
G(\x) \sim  \frac{\mathtt{C}_{\bga} n_{1}}{a_{n_{1}}^{(1)} \cdots a_{n_{1}}^{(d)}}\, ,\quad \text{ with } \mathtt{C}_{\bga} = \int_{0}^{\infty} u^{-2+\sum\ga_i^{-1}} g_{\bga } \big(t_1 u^{1/\ga_1}, \ldots, t_d u^{1/\ga_d} \big) \dd u \, .
\end{equation}
Recall $n_1 \sim |x_1|^{\ga_1} \phi_1(|x_1|)^{-1}$ with $\phi_1=L_1$ if $\ga_1\in(0,2)$ and $\phi_1=\sigma_1$ if $\ga_1=2$.
\end{theorem}

We refer to \eqref{def:favdir2} as $\x$ going to infinity along the \emph{favorite scaling}.
Note that under \eqref{def:favdir2} we have $n_i \sim |t_i|^{\ga_i} n_1$, so we can exchange the role of the coordinates if needed.

\subsubsection*{Comments on the balanced case}
In the balanced case, $a_{n_{1}}^{(i)} \equiv  |x_1|$ and $\ga_{i}\equiv \ga$: we obtain that if  either $\ga>2/d$ or  Assumption~\ref{hyp:2} holds and if $x_i/x_1 \to t_i\in \bbR^*$ for all $i\in \{1,\ldots, d\}$,
\begin{equation}
\label{renewal:Will}
G(\x) \sim \mathtt{C}_{\bga}  |x_1|^{\ga - d } \phi(|x_1|)^{-1},\ \ \ \text{ with } \mathtt{C}_{\bga} = \ga \int_{0}^{\infty} v^{d-1-\ga}  g_{\bga } \big(t_1 v, \ldots, t_d v \big) \dd u \, , 
\end{equation}
with $\phi=L$ if $\ga\in (0,2)$ and $\phi =\sigma$ if $\ga=2$ ($d\neq 2$). 
This recovers Williamson's result \cite{cf:Will68} (we used a change of variable for the integral), under weaker conditions if $\ga\le d/2$.

\subsubsection*{The marginal case $d=2$, $\bga=(2,2)$}
\label{sec:marginal}

In the same spirit as for the case $d=1,\ga_1=1,b_n^{(1)}\equiv 0$ (studied in \cite[Sect.~3.2]{cf:B17}), we treat here the case  $d=2$ with $\bga=(2,2)$ and $\mathbf{b}_n \equiv \mathbf{0}$. We give here a renewal theorem (along the favorite scaling) in the case where $\bS_n$ is transient, \textit{i.e.}\ if  $\sum_{n=1}^{+\infty} (a_n^{(1)} a_n^{(2)})^{-1} <+\infty$.

\begin{theorem}
\label{thm:marginald=2}
Suppose that $d=2$ with $\bga=(2,2)$ and $\mu_1=\mu_2=0$ ($\mathbf{b}_n \equiv \mathbf{0}$),  and assume also $\sum_{n=1}^{+\infty}(a_n^{(1)} a_n^{(2)})^{-1} <+\infty $. Recall the definition~\eqref{def:ni} of $n_1,n_2$. If $\| \x \| \to + \infty$  such that $x_2/a_{n_1}^{(2)}$ (equivalently $n_1/n_2$) stays bounded away from $0$ and $+\infty$, we have that
\[G(\x) \sim g_{\bga}(0,0) \sum_{n \ge n_1 }  \frac{1}{a_n^{(1)} a_n^{(2)}}  .\]
Note that $n_1 \mapsto \sum_{n \ge n_1 }  (a_n^{(1)} a_n^{(2)})^{-1} $ vanishes as a slowly varying function.
\end{theorem}
In the balanced case ($a_n^{(i)}\equiv a_n$), then $\bS_n$ is transient if and only if $\int_1^{+\infty} \frac{du}{ u \sigma(u)} <+\infty$ (recall the definition \eqref{def:sigmax} of $\sigma(\cdot)$), and we can rewrite the above as: if $\|\x\| \to+ \infty$ such that $|x_1|/|x_2|$  stays bounded away from $0$ and $+\infty$, then
\[G(\x) \sim 2 g_{\bga}(0,0) \int_{|x_1|}^{+\infty}  \frac{du}{ u \sigma(u)}, \quad \text{ as } |x_1|\to+\infty \, .\]

\subsection{Case II (\emph{non-zero mean}):   $\mu_i\neq 0$ for some $i$ with $\ga_i>1$}

We let $i_0$ be the first $i$ such that $\mu_i \neq 0$, and assume that $x_{i_0}$ and $\mu_{i_0}$ have the same sign. 

\begin{theorem}
\label{thm:ren1}
Assume that $\ga_{i_0}>1$, $\mu_{i_0}\neq 0$, and that $\mu_i=0$ for $i<i_0$. Assume that one among the following three conditions holds: 
\[(i) \ \sum_{i=1}^{d} \ga_i^{-1} <2 \ ; \qquad (ii)\ \gamma_{i_0} > \sum_{i\neq i_0} \ga_i^{-1} \ ;\qquad (iii)\ \text{Assumption~\ref{hyp:2}} \, .\]
Recall that $n_{i_0}= |x_{i_0}|/|\mu_{i_0}|$, see \eqref{def:ni}. If $\| \x \| \to + \infty$  such that for all $1\leq i \leq d$
\begin{equation}
\label{def:favdir}
(x_{i}- b_{n_{i_0}}^{(i)}) /a_{n_{i_0}}^{(i)} \to t_i \in \bbR \qquad \text{ as } |x_{i_0}| \to +\infty \quad  (t_{i_0}=0) ,
\end{equation}
(if (i) or (ii) does not hold, assume that $t_i\neq 0$ for $i$'s with $b_n^{(i)}\equiv 0$), then we have that
\begin{equation}
\label{casIfavorite}
G(\x) \sim  \frac{  \mathtt{C}'_{\bga} a_{n_{i_0}}^{(i_0)}}{a_{n_{i_0}}^{(1)} \cdots a_{n_{i_0}}^{(d)}}\, ,\quad \text{ with } \mathtt{C}'_{\bga} = \int_{-\infty}^{\infty} g_{\bga } \big(t_1+  \kappa_1 u, \ldots , t_d + \kappa_d u  \big) \dd u \, .
\end{equation}
where we set $\kappa_i= \mu_i a_{i,{i_0}} \ind_{\{i\ge i_0\}}$ (recall Section~\ref{sec:convention}, $a_{i,i_0}=0$ if ${\ga_i<\ga_{i_0}}$). 
\end{theorem}

\noindent
As for Theorem~\ref{thm:ga<1}, we refer to \eqref{def:favdir} as $\x$ going to infinity along the \emph{favorite direction}.

\subsubsection*{Comments on the balanced case}
If $a_n^{(i)} \equiv a_n$ and $\ga_i\equiv \ga$, case II corresponds to having  $\ga>1$ and one $\boldsymbol{\mu}:=(\mu_1,\dots, \mu_d) \neq \mathbf{0}$.
If $\ga >d/2$, if $\gamma_{i_0} > (d-1)/2$ (in the case $\gamma_{i_0}>\ga=2$) or if Assumption~\ref{hyp:2} holds (put otherwise if (i),(ii) or (iii) in Theorem~\ref{thm:ren1} holds), we therefore obtain  that for $\mathbf{t} = (t_1,\ldots, t_d)$ with $t_i\neq 0$ if $\mu_i=0$, 
\begin{equation}
\label{balancedga>1}
G\big( \lfloor  r  \boldsymbol{\mu} +   a_{r} \mathbf{t} \rfloor \big) \sim  \frac{ \mathtt{C}'_{\mathbf{t}} }{ (a_r)^{d-1}} \qquad \text{ with }  \mathtt{C}'_{\mathbf{t}} =  \int_{-\infty}^{+\infty} g_{\bga} \big( \mathbf{t} + u  \boldsymbol{\mu}  \big) du \, , \quad \text{ as } r\to+\infty.
\end{equation}

In the symmetric case where we have $\mu_i \equiv \mu \neq 0$, the result simplifies: let us state it along the diagonal $\mathbf{1} = (1,\ldots, 1)$ for simplicity,
\[G ( r \mathbf{1} ) \sim \frac{|\mu|^{\frac{d-1}{\ga} -1 } }{(a_r)^{d-1}}\int_{-\infty}^{+\infty} g( v \mathbf{1} ) dv   \quad \text{ as } r\to+\infty\, .  \]
Indeed, we used that $a_{r/|\mu|} \sim |\mu|^{-1/\ga} a_r$, and a change of variable for the integral.

\subsection{Case III: $\ga_{i_0}=1$}
\label{sec:alpha=1}
Let us define $i_0 = \min\{ i , b_n^{(i)} \not \equiv  0\}$, and assume that $\ga_{i_0}=1$ with either $\mu_{i_0} \in \bbR^*$ or  $p_{i_0} \neq q_{i_0}$. 
For an overview of results and estimates on (univariate) random walks of Cauchy type, we refer to \cite{cf:B17}---many of the estimates we use below come from there.
Having $\mu_{i_0}\in \bbR^*$ or $p_{i_0}\neq q_{i_0}$ ensures in particular that $|b_n^{(i_0)}|/a_{n}^{(i_0)} \to+\infty$:

$\ast$  If $\mu_{i_0} \in\bbR^*$ then $b_n^{(i_0)} \sim \mu_{i_0} n$ and $a_n^{(i_0)} = o(n)$ ($|\mu_{i_0}|<+\infty$ implies that $L_{i_0}(x) = o(1)$).

$\ast$ If $|\mu_{i_0}| = +\infty$ then $b_n^{(i_0)} \sim (p_{i_0}-q_{i_0}) n \ell_{i_0}(a_n^{(i_0)})$ with $\ell_{i_0}(x) := \int_1^x L_{i_0} (u) u^{-1} du $ which verifies $\ell_{i_0}(x)/L_{i_0}(x) \to +\infty$ as $x\to+\infty$, see \cite[Prop. 1.5.9.a]{cf:BGT}. Since on the other hand $a_{n}^{(i_0)} \sim n L_{i_0}(a_n^{(i_0)})$, we get that $a_n^{(i_0)}=o(|b_n^{(i_0)}|)$.

$\ast$ If $\mu_{i_0} =0$, then similarly, $b_n^{(i_0)} \sim -(p_{i_0}-q_{i_0}) n \ell^\star_{i_0}(a_n^{(i_0)})$ with $\ell^{\star}_{i_0}(x) := \int_x^{\infty} L_{i_0} (u) u^{-1}$, which also verifies $\ell_{i_0}(x)/L_{i_0}(x) \to +\infty$ as $x\to+\infty$. We also get that $a_n^{(i_0)}=o(|b_n^{(i_0)}|)$.

Analogously to Section~\ref{sec:convention}, if $\ga_i=1$, we work along a subsequence such that the following limit exists  
\begin{equation}
\label{tildea}
\tilde a_{i,i_0} :=\lim_{n\to\infty} \frac{a_n^{({i_0})}}{\mu_{i_0}(a_n^{(i_0)})}  \frac{\mu_i(a_n^{(i)})}{a_n^{(i)}}  \in \bbR \quad \text{ for } i\ge i_0 \quad (\tilde a_{i_0,i_0}=1).
\end{equation}
If $\ga_{i}<1$ we set $\tilde a_{i,i_0} =0$.
We stress that it is possible to have $\tilde a_{i,i_0}>0$ even if $a_{i,i_0}=0$.  For instance, take $L_{i_0}(x)=1$ and $L_i(x)=\log x$: we get that $a_n^{(i_0)} \sim n$, and $a_n^{(i)}\sim n \log n$ so $a_{i,i_0}=0$; but we have that $\mu_{i_0}(n)\sim \log n$ and $\mu_i(n) \sim \frac12 (\log n)^2$, so $\tilde a_{i,i_0}=1/2$.

 \begin{theorem}
\label{thm:ga=1}
Assume that $\ga_{i_0}=1$ with $\mu_{i_0}\in \bbR^*$ or $p_{i_0}\neq q_{i_0}$, and that $b_n^{(i)} \equiv 0$ for $i<i_0$. Suppose that Assumption~\ref{hyp:2} holds.
Define $\tilde \kappa_i = \tilde a_{i,i_0} \ind_{\{i \ge i_0\}}$, and recall the definition~\eqref{def:ni} of~$n_{i_0}$. If $\| \x \|\to +\infty$ such that  for all $1\leq i \leq d$,
\begin{equation}
\label{def:favdir3}
(x_{i}- b_{n_{i_0}}^{(i)}) / a_{n_{i_0}}^{(i)} \to t_i\in \bbR \qquad \text{ as } |x_{i_0}| \to +\infty \quad (t_{i_0}=0),
\end{equation}
(with $t_i\neq 0$ when $b_n^{(i)}\equiv 0$)
then we have 
\begin{equation}
\label{casIIIfavorite}
G(\x)  \sim \frac{1}{|\mu_{i_0}(a_{n_{i_0}}^{(i_0)})| } \cdot \frac{ \mathtt{C}''_{\bga}  a_{n_{i_0}}^{(i_0)}}{ a_{n_{i_0}}^{(1)} \cdots a_{n_{i_0}}^{(d)}}\, ,\quad \text{ with } \mathtt{C}''_{\bga} = \int_{-\infty}^{\infty} g_{\bga } \big(t_1+  \tilde \kappa_1 u, \ldots , t_d + \tilde \kappa_d u  \big) \dd u .
\end{equation}
Note that $n_{i_0} \sim |x_{i_0}|/|\mu_{i_0}(|x_{i_0}|)|$,  and  $\mu_{i_0}(a_{n_{i_0}}^{(i_0)})\sim \mu_{i_0}(|x_{i_0}|) $ as $|x_{i_0}|\to+\infty$, see Section~\ref{sec:casIII-prelim}. 
\end{theorem}
Again, we refer to \eqref{def:favdir3} as $\x$ going to infinity along the favorite direction.

\subsubsection*{Comments on the balanced case}
In the balanced and \emph{symmetric} case, we have $a_n^{(i)} \equiv a_n$ and $b_n^{(i)} \equiv b_n$ ($\mu_i(x) \equiv \mu(x)$, $\tilde \kappa_i \equiv 1$). The favorite direction is the diagonal $\mathbf{1} =(1,\ldots,1)$, and we can write, for $\mathbf{t} = (t_1,\ldots, t_d)$,
\begin{equation}
G( r \mathbf{1} + \lfloor a_{r/|\mu(r)|} \mathbf{t}\rfloor) \stackrel{r\to+\infty}{\sim}  \frac{\mathtt{C}''_{\mathbf{t}}}{ |\mu( r) |\times (a_{r/|\mu(r)|})^{d-1}} \quad \text{ with } \mathtt{C}''_{\mathbf{t}} =\int_{-\infty}^{\infty} g_{\bga } \big(\mathbf{t} +  u \mathbf{1} \big) \dd u
\, .
\label{balancedga=1}
\end{equation}
Indeed, $n_{i_0}\sim r/|\mu(r)|$, and we also used that $\mu(a_{n_{i_0}}) \sim \mu(|b_{n_{i_0}}|) =\mu(r)$, see \cite[Lemma~4.3]{cf:B17}.

As a simple example, take Example~\ref{ex2} with $\gb_i\equiv 1$, $\gb=1+d$:
$\bP(\bX_1 = \x) = \mathrm{c}_d\, \|\x\|^{-(1+d)}$ for $x\in \mathbb{N}^d$, and $\bP(X_1^{(i)}>n) \sim \mathrm{c}_1/n$.
We have $\ga_i\equiv 1$, and $a_n \sim  n/\mathrm{c}_1$, $\mu(n) \sim \mathrm{c}_1 \log n$: we therefore get that
$ G(r\mathbf{1} + ( r/\log r)\mathbf{t})  \sim \mathrm{c}_{\mathbf{t}} (\log r)^{d-2} r^{-(d-1)} $ as $r\to+\infty$.

\section{Renewal estimates away from the favorite direction or scaling}
\label{sec:resultsaway}

In this section, we provide bounds on $G(\x)$ that hold uniformly on $\x$: in particular, this sharpens our estimates when $\x$ goes away from the favorite direction or scaling (one would have $\mathbf{C}_{\bga} , \mathbf{C}'_{\bga}$ or $\mathbf{C}''_{\bga}\to 0$ in Theorems~\ref{thm:ga<1}, \ref{thm:ren1} or \ref{thm:ga=1}). 
We do not obtain sharp asymptotics for $G(\x)$, mostly because the local large deviation estimates of Section~\ref{sec:locallimit} are not sharp---first of all because our Assumption~\ref{hyp:2} does not give the precise asymptotic of $\bP(\bX_1=\x)$.
Let us stress that in \cite{cf:BGKcramer}, the authors manage to obtain the sharp asymptotic of $G(\x)$ in a specific setting (with application to a DNA model): $\bX_1 \in \bbN^2$, and the local probabilities $\bP(\bX_1=\x)$ are known asymptotically, one coordinate having a heavy-tail, the second one having an exponential tail.
One should also be able to obtain the sharp asymptotics of $G(\x)$ for instance in Example~\ref{ex2}, but we do not pursue it here to avoid 
additional lengthy and technical calculations.

We also stress that having uniform bounds on $G(\x)$ turn out to be useful, for instance when studying the intersection of two independent (multivariate) renewal processes $\bS=\{\bS_n\}_{n\geq 0}$, $\bS'=\{\bS'_n\}_{n\geq 0}$ with same distribution.
Indeed, $\bE[| \bS \cap \bS'|] = \sum_{\x \in\bbZ^d} \bP(\x \in \bS)^2$, and to known whether $\bS \cap \bS'$ is finite, good bounds on $G(\x)=\bP(\x\in\bS)$ are essential. The main contribution to $\bE[| \bS \cap \bS'|]$ will come from points along the favorite direction, and one needs to know how fast $G(\x)$ decreases when $\x$ moves away from it.
We refer to \cite[App.~A.2]{cf:BGKiid} for some results on the intersection of two independent renewal processes.

For the simplicity of the exposition, we only present the case of  dimension $d=2$.
Also, we will work under Assumption~\ref{hyp:2}.
Often, results will be sharper in the case of renewal processes, as will be outlined in our theorems. We divide our statements into three parts: $\mathbf{b}_n \equiv 0$ (centered);  $b_n^{(1)}, b_n^{(2)} \not \equiv 0$ (non-zero mean for both coordinates); $b_n^{(i_0)}\not \equiv 0$ and $b_n^{(i_1)} \equiv 0$ (mixed case).
The proofs are presented in Section~\ref{sec:away}.

\subsection{Case I (\emph{centered case}), $\mathbf{b}_n\equiv \mathbf{0}$}

Let us leave aside the marginal case $d=2$ $\bga=(2,2)$: we have $\ga_1^{-1}+\ga_2^{-1}>1$.
Recall the definition \eqref{def:ni} of $n_i$ ($n_i \sim |x_i|^{\ga_i} \phi_i(|x_i|)^{-1}$ with $\phi_i=L_i$ if $\ga_i\in(0,2)$ and $\phi_i= \sigma_i$ if $\ga_i=2$), and let $i_0,i_1$ be  such that $n_{i_0} = \min\{ n_1,n_2\}$ and  $n_{i_1} = \max\{n_1,n_2\}$.

\begin{theorem}
\label{thm:awayI}
Assume that $\bb_n \equiv 0$, and that Assumption~\ref{hyp:2} holds. 
Then for any $\gd>0$, we have a constant $C_{\gd}$ such that for any $\x\in \bbZ^2$, 
\begin{equation}
\label{casIIaway}
G(\x) \le  \frac{C_{\gd} n_{i_0}}{a_{n_{i_0}}^{(1)}  a_{n_{i_0}}^{(2)}}     \Big( \frac{n_{i_1}}{n_{i_0}} \Big)^{- \nu+ \gd} \, , \qquad \text{with } \nu= (1+\ga_{i_1}^{-1}) \frac{\ga_{1}^{-1} +\ga_{2}^{-1} -1}{\ga_{1}^{-1} +\ga_{2}^{-1} +1} \, .
\end{equation}
If $(\bS_n)_{n\ge 0}$ is a renewal process (necessarily $\ga_1,\ga_2<1$), we can replace $\nu$ by $ 1+\ga_{i_1}^{-1}$.
\end{theorem}
Clearly, Theorem~\ref{thm:awayI} improves \eqref{casIIfavorite} in the regime $n_{i_1}/n_{i_0} \to +\infty$.

\subsubsection*{About the balanced case}
If $\ga_i\equiv \ga \in (0,2]$ and $a_n^{(i)}\equiv a_n$, we obtain that under Assumption~\ref{hyp:2}, for any $\gd>0$ there exists a constant $C_{\gd}$ such that for any $\x\in \bbZ^2$, setting $x_{i_0} = \min \{ x_1,x_2\}$ and $x_{i_1} = \max \{  x_1,x_2\}$,
\begin{equation}
G(\x) \leq C_{\gd} |x_{i_0}|^{\ga-2} \phi(|x_{i_0}|)^{-1} \, \Big( \frac{x_{i_1}}{x_{i_0}} \Big)^{- \theta + \gd}  \quad \text{ with } \ \theta := (1+\ga) \frac{2-\ga}{2+\ga} \, ,
\end{equation}
with $\phi=L$ if $\ga\in(0,2)$ and $\phi=\sigma$ if $\ga=2$.
where  (recall $n_i\sim x_{i}^{-\ga} \phi(x_i)^{-1}$). 
If  $(\bS_n)_{n\geq 1}$ is a renewal process (necessarily $\ga\in (0,1)$), then we can replace $\theta$ by $1+\ga$.

\subsection{Case II-III (\emph{non-zero mean}), subcase (a): $b_n^{(1)}, b_n^{(2)} \neq 0$}

Let us consider the case when for both $i=1,2$ we have: either $\ga_i\geq 1$ and $\mu_i\in \bbR^*$, or $\ga_i=1$ and $p_i\neq q_i$. This insures that $b_n^{(i)}\neq 0$ for $i=1,2$, and places us in the setting of cases II and III of Section~\ref{sec:diag}.

Recall the definition \eqref{def:ni} of $n_i$: we have $n_i \sim |x_i|/|\mu_i(|x_i|)|$ (both if $\ga_i>1$ or $\ga_i=1$).
Let us also define $m_i:= a_{n_i}^{(i)}/|\mu_i(a_{n_i}^{(i)})|$:  in Section~\ref{sec:casIII-prelim}, we see that $m_i=o(n_i)$, and that the typical number of steps for the $i^{\rm th}$ coordinate to reach $x_i$ is $n_i +O(m_i)$ (this is trivial if $\ga_i>1$).
Let us stress that the favorite direction ($|x_2- b_{n_1}^{(2)}| =O (a_{n_1}^{(1)})$, $|x_1- b_{n_2}^{(1)}| = O( a_{n_2}^{(2)})$, see \eqref{def:favdir}-\eqref{def:favdir3}) corresponds to having $n_1- n_2=O(m_i)$ for $i=1,2$.
We will state only the case $n_{1} \leq n_{2}$, the other case being symmetric.
\begin{theorem}
\label{thm:casII-IIIb}
Suppose that Assumption~\ref{hyp:2} holds, and that for $i=1,2$: either $\ga_i\geq 1$ and $\mu_i\in\bbR^*$, or $\ga_i=1$ and $p_i\neq q_i$.
Then for every $\gd>0$ there is a constant $C_{\gd}$ such that, for all $\x\in \bbZ^2$ (recalling the definition~\eqref{def:ni} of $n_i$, and of $m_i:= a_{n_i}^{(i)}/|\mu_i(a_{n_i}^{(i)})|$),

\smallskip
{\rm (i)} If $n_1\le n_2 \leq 2 n_1$,
\begin{align}
G(\x) \le  \frac{C_{\gd}}{a_{n_2}^{(2)} |\mu_1(a_{n_1}^{(1)})|} &\times  \Big( \frac{n_2-n_1}{m_2}\Big)^{-1+\gd} \Big(\frac{n_2- n_1}{m_1} \Big)^{\gd} \notag \\
& \times  \bigg\{ \Big( \frac{n_2-n_1}{m_1}\Big)^{-\ga_1}  R^{(1)}(n_2-n_1) +   \Big( \frac{n_2-n_1}{m_2} \Big)^{-\ga_2}  R^{(2)}(n_2-n_1)  \bigg\}  \, ,
\label{casII-IIIawaybi}
\end{align}
with $R^{(i)}(m):= \ind_{\{\ga_i\in (0,2)\}} + (m^{2-\gamma_i} + e^{-c m/m_i}) \ind_{\{\ga_i=2\}}$.

\smallskip
{\rm (ii)} If $n_2\geq 2 n_1$, 
\begin{equation}
\label{casII-IIIawaybii}
G(\x) \le  C_{\gd} \big( n_1 \vee n_2^{1\wedge (\gamma_2/\gamma_1)}  \big) \times n_2^{-(1+\gamma_2)+\gd} \leq C_{\gd} n_2^{-\gamma_2+\gd} \, . 
\end{equation}
If $\bS_n$ is a renewal process, then $G(\x)=0$ as soon as $n_2\geq |x_1|$, in particular if $n_2 \geq n_1^{1+\gd}$.
\end{theorem}
We stress that in the case $\ga_1,\ga_2 >1$, then we can replace $n_i$ by $x_i/\mu_i$ ($x_i$ and $\mu_i$ with the same sign) and $m_i$ by $a_{|x_i|}^{(i)}$.

\subsubsection*{About the balanced case}

In the balanced case ($a_n^{(i)} \equiv a_n$), Theorem~\ref{thm:casII-IIIa} gives the following:

\smallskip
$\ast$ If $\ga>1$,  $n_i =x_i/\mu_i$, and $|n_1-n_2| = |x_1/\mu_1 - x_2/\mu_2|$: the bound
\eqref{casII-IIIawaybi} (together with \eqref{balancedga>1} for the case $|s|\leq a_r$) gives, for any $ |s|\leq r$
\begin{equation}
\label{balancedawayga>1}
G\big( (r,  \lfloor \tfrac{\mu_2}{\mu_1} r \rfloor +s) \big) \leq 
\frac{C }{a_{r}}    \Big( 1\wedge \Big( \frac{|s|}{a_r} \Big)^{-(1+\ga)+\gd} R^{(2)}(|s|)   \Big)
\end{equation}
($R^{(2)}(|s|) = 1$ if $\ga<2$ and $R^{(2)}(|s|)= |s|^{2-\gamma_2} + e^{- c |s|/a_r}$ if $\ga=2$).
For $|s|\geq r$, then \eqref{casII-IIIawaybii} gives that $G\big((r,  \lfloor \tfrac{\mu_2}{\mu_1} r \rfloor +s) \big) \leq C_{\gd} |s|^{-\gamma_2+\gd}$.

\smallskip
$\ast$ If $\ga=1$ and $\mu_1,\mu_2 \in\bbR^*$, then we have $ |\mu_i-\mu_i(a_n)| = O(L(a_n))|$, so  $ |n_1-n_2| = |x_1 / \mu_1(a_{n_1}) -n_2/\mu_2(a_{n_2})| =  |x_1 /\mu_1 - x_2 /\mu_2|  + O(n_1 L(a_{n_1}))$, provided that $ x_1\asymp x_2$ (equivalently $n_1 \asymp n_2$)---note also that $n_1 L(a_{n_1}) =O(a_{n_1})$.
We therefore get the same conclusion as in~\eqref{balancedawayga>1}. The case $|s|\geq r$ is similar to the case $\ga>1$ above.

\smallskip
$\ast$ If $\ga=1$ with $|\mu_i|=+\infty$ or $0$, we assume additionally that the distribution is \emph{symmetric}: we have $\mu_i(n)\equiv \mu(n)$ (we actually only need this for $n$ large).
Then, using Claim~\ref{claim:mu} below, we have $|\mu(a_{n_1}) -\mu(a_{n_2}) | = O(L(a_{n_2}))$ provided that $x_1/x_2$ (hence $n_1/n_2$) is bounded away from $0$ and $+\infty$: we get
$|n_1 \mu(a_{n_1})-  n_2\mu(a_{n_2}) | = |n_1 -n_2|\mu(a_{n_1}) + O(n_2 L(a_{n_2}))$,
with  $n_2 L(a_{n_2}) =O(a_{n_2}) =O(a_{n_1})$. It gives, as long as $x_1\asymp x_2$,  that $|n_1-n_2| \leq \mu(x_1) |x_1-x_2| + O(a_{n_1})$. Using \eqref{casII-IIIawaybi} (and~\eqref{balancedga=1} for the case $s\leq a_{r/\mu(r)}$), we obtain that for any $|s|\leq r$
\begin{equation}
\label{balancedawayga=1}
G\big( (r,r+s)\big) \leq \frac{C}{|\mu(r)| a_{r/|\mu(r)|}} \Big(  \frac{s }{a_{r/|\mu(r)|}} \vee 1 \Big)^{-2+\gd} \, .
\end{equation}
We used that $m_1= a_{n_1}/|\mu(a_{n_1})| \sim a_{r/|\mu(r)|} /|\mu(r)|$ ($m_1=m_2$).
In the case $|s|\geq r$, then applying \eqref{casII-IIIawaybii} gives that $G((r,r+s))\leq C_{\gd} |s|^{-1 +\gd}$, using also that $n_2\geq c_{\gd'}\,  s^{1-\gd'}$.

\smallskip
We mention  that assumption \eqref{eq:Will} would not improve much \eqref{balancedawayga>1}-\eqref{balancedawayga=1}: the improvement would be only at the level of the slowly varying function, that are absorbed by the exponent~$\gd$. We refer to the end of Section~\ref{sec:casII-IIIawayb} for a discussion.

\subsection{Case II-III (\emph{non-zero mean, mixed}), subcase (b): $b_n^{(i_0)} \neq 0$, $b_n^{(i_1)} \equiv 0$}

Here, we consider again the setting of cases II and III of Section~\ref{sec:diag}, in the case where the second coordinate is ``centered''.

\begin{theorem}
\label{thm:casII-IIIa}
Suppose that Assumption~\ref{hyp:2} holds, and that there is $\{i_0,i_1\}= \{1,2\}$ such that: $b_n^{(i_1)}\equiv 0$ and, either $\ga_{i_0}\geq 1$ and $\mu_{i_0}\in \bbR^*$, or $\ga_{i_0}=1$ and $p_{i_0}\neq q_{i_0}$. Recall the definition~\eqref{def:ni}: we have $n_{i_0}\sim |x_{i_0}| / |\mu_{i_0}(|x_{i_0}|)|$, and $n_{i_1}\sim |x_{i_1}|^{-\ga_{i_1}} \phi_{i_1}(|x_{i_1}|)^{-1}$  with $\phi_{i_1}= L_{i_1}$ if $\ga_{i_1}\in (0,2)$ and $\phi_{i_1}=\sigma_{i_1}$ if $\ga_{i_1}=2$.
There is a constant $C$ and for any $\gd>0$ there is a constant $C_{\gd}$ such that for any $\x \in \bbZ^2$:

\smallskip
{\rm (i)} If $n_{i_1} \leq n_{i_0}$,
\begin{equation}
\label{casIaway}
G(\x) \le \frac{C}{ |\mu_{i_0}(a_{n_{i_0}}^{(i_0)})| a_{n_{i_0}}^{(i_1)} } + 
\begin{cases}
C_{\gd} \, n_{i_1}^{2-1/\ga_{i_1}}   |x_{i_0}|^{-(1+\gamma_{i_0}) +\gd} & \tif \ga_{i_1} \leq 1/2\, ,\\
0 & \tif \ga_{i_1}>1/2   \, .
\end{cases}
\end{equation}
If $\bS_n$ is a renewal process (necessarily $\ga_{i_1}\in(0,1)$), then there is an exponent $\zeta_{\gd}>0$ such that
$G(\x) \leq  C_{\gd} \,  (n_{i_1})^{2-1/\ga_{i_1}}   |x_{i_0}|^{-(1+\gamma_{i_0}) +\gd} + e^{ - c ( n_{i_0}/n_{i_1} )^{\zeta_{\gd}} }  \, .$

\smallskip
{\rm (ii)} If $n_{i_1} \geq n_{i_0}$, we set $m_\gd:= (|x_{i_1}|^{\gamma_{i_1}/\gamma_{i_0} +\gd} \vee (n_{i_0})^{1+\gd}) \wedge (n_{i_1})^{1-\gd}$ (and $m_\gd=+\infty$ if $(n_{i_0})^{1+\gd} > (n_{i_1})^{1-\gd}$),  and we have
\begin{equation}
\label{casIawayii}
G(\x) \leq \frac{C_{\gd}}{ a_{n_{i_1}}^{(1)} } \, \big(  1\wedge m_{\gd}\,  |x_{i_1}|^{-\gamma_{i_1}}   \big)
\end{equation}
If $\bS_n$ is a renewal process,
$G(\x) \leq  \frac{C_{\gd}  n_{i_0}^{\gd}}{ a_{n_{i_1}}^{(i_1)}} \big( n_{i_0} |x_{i_1}|^{-\ga_{i_1} +\gd} + e^{- c_{\gd} (n_{i_1}/n_{i_0})^{1-\gd} }  \ind_{\{\ga_{i_1}=2\}}\big)  \, . $
\end{theorem}

\noindent
\textit{Notational warning}: In the rest of the paper, we use $c,C,c',C',$... as generic constants, and we will keep the dependence on parameters when necessary, writing for example $c_{\gep},C_{\gep}$ for constants depending on a parameter $\gep$.


\section{Proof of the local large deviations}
\label{sec:prooflocallimit}

In this section, we prove the local limit theorems of Section~\ref{sec:locallimit}: Theorem~\ref{thm:locallimit1} in Section~\ref{sec:simplelocal}, Theorem~\ref{thm:locallimit2} in Section~\ref{sec:localharder}, and Theorem~\ref{thm:locallimit3} in Section~\ref{sec:localWill}.
But first of all, let us  recall some univariate large deviation results.

\subsection{Univariate large deviations: a reminder of Fuk-Nagaev inequalities}
\label{sec:univariate}

We start by giving a brief reminder of useful large deviation results for univariate random walks (\textit{i.e.}\ we focus on $S^{(1)}$) in the domain of attraction of an $\ga_1$-stable distribution---this will be useful throuhout the section. Most of these estimates can be found in \cite{cf:Nag79}, but the case $\ga_1=1$ was improved recently, cf.\ \cite{cf:B17}.
This will enable us to obtain local limit theorems for multivariate random walks in the next section.

In the rest of the section, we denote $M_n^{(i)} := \max_{1\le k\le n} X_k^{(i)}$. We refer to Section~5 in \cite{cf:B17} for an overview on how to derive the following statement from \cite{cf:Nag79}.

\begin{theorem}
\label{thm:fuknagaev}
Suppose that Assumption \ref{hyp:1} holds. There are constants $c,c'$ such that

$\ast$ if $\ga_1 \in(0,1)\cup  (1,2)$,  for any $1\leq y\leq x$
\[\bP \big( S_{n}^{(1)} - b_n^{(1)}  \geq x ; M_n^{(1)} \leq y \big) \leq  \Big( c \frac{y}{x} n  L_1(y) y^{-\ga_1}  \Big)^{ x/y}  \, ; \]

$\ast$ if $\ga_1=1$, for every $\gep>0$, there is some $C_{\gep}>0$ such that, for any $x \ge C_{\gep} a_n^{(1)}$ and $1\leq y\le x$
\[
\bP \big( S_{n}^{(1)} - b_n^{(1)} \geq x ; M_n^{(1)} \leq y \big) \leq 
\Big( c \frac{y}{x} n L_1(y) y^{-1} \Big)^{ (1-\gep) x/y} + e^{ -  ( x/a_n^{(1)}  )^{1/\gep} }  \, ;
\]

$\ast$ if $\ga_1 =2$, for any $y\leq x$
\[\bP \big(  S_{n}^{(1)} - b_n^{(1)}  \geq x ; M_n^{(1)} \leq y \big) \leq 
\Big(c \frac{ y}{ x} n y^{-\gamma_1} \gp_1 (y)   \Big)^{ \frac{x}{2 y} } +   C e^{ - c \frac{x^2}{n \sigma_1(y) } } \, .
\]
\end{theorem}

The case $\ga_1\in (0,1)\cup(1,2)$ is given by Theorems 1.1 and 1.2 in \cite{cf:Nag79} (we also refer to Section 3 of \cite{cf:CD16}, which contains a simpler proof of that fact). The case $\ga_1 = 1$ is given in \cite[Theorem~2.2]{cf:B17}. The case $\ga_1=2$ is given by Corollary 1.7 in \cite{cf:Nag79}.

As a consequence of Theorem~\ref{thm:fuknagaev},  there is a constant $c_0$ such that, whenever $x\ge a_n^{(1)}$,
\begin{equation}
\label{eq:largedev}
\bP \big(  S_n^{(1)} -b_n^{(1)} \ge x \big) \le c_0 n \gp_1(x) x^{-\gamma_1}  +  c_0 e^{ -  \tfrac{ x^2}{c_0 n \sigma_1(x)} }\ind_{\{\ga_1=2\}} \, . 
\end{equation}
Indeed, the left-hand side is bounded by $\bP \big( M_n^{(1)} \ge x/4 \big) + \bP \big( S_{n}^{(1)} - b_n^{(1)} \geq x ; M_n^{(1)} \leq x/4 \big)$.
Using a union bound, and because of Assumption \ref{hyp:1}, the first term is bounded by a constant times $n \gp_1(x) x^{-\gamma_1} $. For the second term, we use Theorem~\ref{thm:fuknagaev}, which gives that 

- if $\ga_1\in (0,1)\cup(1,2)$, it is bounded by a constant times  $\big(  n  x^{-\ga_1} L_1(x) \big)^{4}$;

- if $\ga_1=1$, is it bounded by a constant times   $ \big(  n   L_1(x) x^{-1} \big)^{4(1-\gep)} +  e^{ - c (x/a_n^{(1)})^{1/\gep}}$ ;

- if $\ga_1=2$, it is bounded by $\big(  n  x^{-\gamma_1} \gp_1(x) \big)^{2} + e^{- c x^2/(n\sigma_1(x)) }$.

\smallskip
Another useful consequence of Theorem~\ref{thm:fuknagaev} is the following: let $C,C'$ be two (large) constants, with $C' < C/10$, then there is a constant $c''$ such that for any $x\geq C a_n^{(1)}$, we have
\begin{align}
\label{consequencefuknag}
\bP \big(  S_n^{(1)} -b_n^{(1)} \ge x ,\, M_n^{(1)} \leq C' a_n^{(1)} \big) \leq  (n \gp_1(x) x^{-\gamma_1})^2 e^{ - c'' x_1/a_n^{(1)}} + e^{- c'' (x_1/a_n^{(1)})^2} \ind_{\{\ga_1=2\}} \, .
\end{align}
We used $(n \gp_1(x) x^{-\gamma_1})^2$ for technical purposes (it is needed in the following), but the bound is also valid without the square (or even without this term), bounding $n \gp_1(x) x^{-\gamma_1}$ by $1$ if $x$ is larger than $C a_n^{(1)}$.

Indeed, Theorem~\ref{thm:fuknagaev} gives that the left-hand side is bounded by
\begin{align*}
\Big( c n \gp_1(a_n^{(1)}) &(a_n^{(1)})^{-\gamma_1} \ \frac{a_n^{(1)}}{x_1} \Big)^{c' x_1/a_n^{(1)}  } + e^{-c' x_1/a_n^{(1)}} \ind_{\{\ga_1=1\}}  +  e^{- c' x_1^2 / (n\sigma_1(a_n^{(1)}) )}  \ind_{\{\ga_1=2\}} \, .
\end{align*}
To obtain \eqref{consequencefuknag} from this, we use the following. (1) If $\ga_1\in (0,2)$ then $L_1=\gp_1$, $\gamma_1=\ga_1$ and  $n L_1(a_n^{(1)}) \sim (a_n^{(1)})^{\ga_1} L_1(a_n^{(1)})^{-1}$ so the first and second term are smaller than $\exp(-c' x_1/a_n^{(1)})$ provided that $x_1/a_n^{(1)}\geq C$. Then we use that $  \exp(-c' x_1/a_n^{(1)})$ is bounded by a constant times $(x_1/a_n^{(1)})^{- 4\ga_1 } \exp(-c'' x_1/a_n^{(1)})$ with $c''>c'$ since $x_1\geq C a_n^{(1)}$, and then that $ (x_1/a_n^{(1)})^{-4\ga_1}$ is bounded by a constant times $(n L_1 (x_1) x_1^{-\ga_1})^2$ thanks to Potter's bound \cite[Thm.~1.5.6]{cf:BGT} (recall the definition~\eqref{def:an} of $a_n^{(1)}$). (2) If $\ga_1=2$,  $\gp_1(a_n^{(1)}) (a_n^{(1)})^{-\gamma_1}$ is bounded above by a constant times $\gp(x_1) x_1^{-\gamma_1} (a_n^{(1)}/x_1)^{-1}$ (by Potter's bound, since $\gamma_1>1$).
Therefore, the first term is bounded by  $(n \gp(x_1) x_1^{-\gamma_1})^{ c x_1/a_n^{(1)}}$ times $\exp(-c'' x_1/a_n^{(1)})$ since $x_1\geq C a_n^{(1)}$. We also used that $n \sigma_1(a_n^{(1)}) \sim a_n^{(1)}$ when $\ga=2$.


\subsection{Proof of Theorem \ref{thm:locallimit1}}
\label{sec:simplelocal}

We fix $i\in\{1,\ldots, d\}$, and consider some $\x\in\bbZ^d$ with $x_i\ge a_n^{(i)}$. Recall that $\hat \bS_n = \bS_n - \lfloor \mathbf{b}_n \rfloor$.
We denote $\mathbf{d}_n := \frac{1}{2} \lfloor \mathbf{b}_n \rfloor -\mathbf{b}_{\lfloor n/2 \rfloor}$, so that $\bS_n - \frac{1}{2}\lfloor \mathbf{b}_n \rfloor  = \hat \bS_{\lfloor n/2 \rfloor} -\bd_n$.

We decompose $\bP(\hat\bS_n = \x)$ according to whether $S_{\lfloor  n/2 \rfloor}^{(i)} - \tfrac12 \lfloor b_n^{(i)} \rfloor \ge x_i/2$ or not, so that 
\begin{align}
\bP \big( \hat\bS_n &= \x \big) \notag \\
& \le \bP\big( \hat\bS_n = \x  ; S_{\lfloor  n/2 \rfloor}^{(i)} - \tfrac12 \lfloor b_n^{(i)} \rfloor \ge x_i/2   \big) + \bP \big( \hat\bS_n = \x  ;  S_n^{(i)}- S_{\lfloor  n/2 \rfloor}^{(i)} - \tfrac12 \lfloor b_n^{(i)} \rfloor \ge x_i/2   \big)\, .
\label{splithalf}
\end{align}
The two terms are treated similarly, so we only focus on the first one.
We have
\begin{align}
\label{firstpiece}
&\bP\big( \hat\bS_n = \x \,  ;\,  \hat S_{\lfloor  n/2 \rfloor}^{(i)} \ge x_i/2  +d_n^{(i)} \big)  \\
&= \sumtwo{\z \in \bbZ^d}{z_i \ge \tfrac12 \lfloor b_n^{(i)} \rfloor + x_i/2} \bP\big( \bS_{\lfloor n/2\rfloor} =\z \big) \bP \big( \bS_n - \bS_{\lfloor n/2\rfloor}  =  \lfloor \mathbf{b}_n \rfloor+ \x-\z \big)  \notag \\
& \le \frac{C}{a_{n}^{(1)} \cdots a_n^{(d)}} \!\!\!  \sumtwo{\z \in \bbZ^d}{z_i \ge \tfrac12 \lfloor b_n^{(i)} \rfloor + x_i/2} \!\!\!   \bP\big( \bS_{\lfloor n/2\rfloor} =\z \big)
 =  \frac{C}{a_{n}^{(1)} \cdots a_n^{(d)}}  \bP \big( S_{\lfloor  n/2 \rfloor}^{(i)} - \tfrac12 \lfloor b_n^{(i)} \rfloor \ge x_i/2   \big) \, ,
\notag
\end{align}
where we used the local limit theorem \eqref{LLT} to get that there is a constant $C>0$ such that for any $k\ge 1$ and $\mathbf{y} \in \bbZ^d$, we have $\bP(\bS_k = \mathbf{y}) \le C ( a_k^{(1)} \cdots a_k^{(d)} )^{-1}$.

Then, in order to use \eqref{eq:largedev} for the last probability, we need to control $\tfrac12 \lfloor b_n^{(i)} \rfloor  - b_{\lfloor n/2 \rfloor}^{(i)}$.
\begin{claim}
\label{claim:dn}
There exists a constant $c>0$ such that for all $n$
\[d_n^{(i)} := \frac{1}{2} \lfloor b^{(i)}_n \rfloor - b^{(i)}_{\lfloor n/2 \rfloor} \ge - c a_n^{(i)} \, . \]
\end{claim}
\begin{proof}
 When $\ga_i \in(0,1)$ we have that $b_n^{(i)}\equiv 0$ so this quantity is equal to $0$. When $\ga_i>1$ we have $b_k^{(i)} = k \mu_i$ in which case  $\tfrac12 \lfloor n\mu_i \rfloor  -  \lfloor n/2\rfloor \mu_i \ge - \mu_i $.
When $\ga_i=1$, this is more delicate but not too hard:
\begin{align*}
\frac n2 \mu_i(a_n^{(i)}) - \lfloor n/2 \rfloor \mu_i(a_{\lfloor n/2 \rfloor}^{(i)}) &\ge \frac n2 \big( \mu_i(a_n^{(i)}) - \mu_i(a_{\lfloor n/2\rfloor}^{(i)}) \big) - |\mu_i(a_{\lfloor n/2\rfloor }) | \\
&\ge  - c\, \big( n L_i(a_n^{(i)}) + |\mu_i(a_n^{(i)})| \big) \ge - c' a_n^{(i)}\, .
\end{align*}
For the second inequality we used \cite[Claim 5.3]{cf:B17} that we reproduce below (separate the positive and negative part of $X_1^{(i)}$), using also that $a_n^{(i)}/a_{\lfloor n/2\rfloor}^{(i)}$ is bounded by a constant.
\end{proof}

\begin{claim}[Claim 5.3 in \cite{cf:B17}]
\label{claim:mu}
Assume that $\ga_i=1$.
For every $\gd>0$, there is a constant $c_{\gd}$ such that for every $u\ge v \ge 1$ we have
\[   \frac{1}{L_i(v)} \big| \mu_i(u) - \mu_i(v) \big|  \le c_{\gd} (u/v)^{\gd} \, .  \]
Additionally, if $ c^{-1}\le u/v \leq c$, we have that $\frac{1}{L_i(v)} \big| \mu_i(u) - \mu_i(v) \big| \leq C |\log (u/v)| $.
\end{claim}

Therefore, provided that $x_i \ge C_4 a_n^{(i)}$ with some constant $C_4$ large enough, Claim~\ref{claim:dn} gives that $\tfrac12 \lfloor b_n^{(i)} \rfloor  - b_{\lfloor n/2 \rfloor}^{(i)} \ge - x_i /4$, so that
\begin{equation}
\label{n/2}
\bP\big( S_{\lfloor  n/2 \rfloor}^{(i)} - \tfrac12 \lfloor b_n^{(i)} \rfloor \ge x_i/2   \big)  \le \bP\big( S_{\lfloor  n/2 \rfloor}^{(i)} -  b_{\lfloor n/2 \rfloor}^{(i)}  \ge x_i/4  \big), 
\end{equation}
and then  \eqref{eq:largedev} provides an upper bound. Plugged in \eqref{firstpiece}, this concludes the proof of Theorem~\ref{thm:locallimit1}, possibly by changing the constants to cover the range $x \geq a_n^{(i)}$, $x< C_4 a_n^{(i)}$.\qed

Note that with the same method, using Theorem~\ref{thm:fuknagaev} instead of \eqref{eq:largedev}, one is able to obtain a local version of Theorems~\ref{thm:fuknagaev}.
\begin{proposition}
\label{prop:CD}
There are some $C_4, C_5>0$ such that, for any $\x$ with $x_i \ge C_4 a_n^{(i)}$, and $1\leq y \le x_i$
\begin{equation}
\label{localfuknag}
 \bP \big(\hat\bS_n = \x \, ;\, M_n^{(i)}  \le y\big) \le \frac{C_5}{a_n^{(1)} \cdots a_n^{(d)}  } \bP\big(     \hat S_{\lfloor n/2 \rfloor}^{(i)}\ge x_i/4  ,  \, M_{\lfloor n/2 \rfloor}^{(i)}  \le y\big) 
\end{equation}
\end{proposition}
The proof of this proposition is a straightforward transposition of the proof of Theorem~\ref{thm:locallimit1}, we leave the details to the reader (for the univariate setting, we refer to Proposition~6 in \cite{cf:B17} and its proof).
We also state two other bounds (in dimension $d=2$ for simplicity), that will be useful in the proof of Theorem~\ref{thm:locallimit2}.

\begin{claim}
\label{claim:localfuknag}
There are constants $C_6,C_7$ such that, for any $\x = (x_1,x_2)$ with $x_1 \ge C_6 a_n^{(1)}$, and any $1\leq y\le x_1$
\begin{equation}
\label{claim1}
\bP\big( \hat S_n^{(1)} \ge x_1 , \hat S_n^{(2)} = x_2 , M_n^{(1)} \le y_1\big)
\le \frac{C_7}{a_n^{(2)}} \bP\big( \hat S_{\lfloor n/2 \rfloor}^{(1)} \ge x_1/4 , M_{\lfloor n/2 \rfloor}^{(1)} \le y_1 \big)\, .
\end{equation}
For any $\x = (x_1,x_2)$ with $x_1\ge C_6 a_n^{(1)}$, $x_2\ge C_6 a_n^{(2)}$, and any $1\leq y_1\le x_1$, $1\leq y_2 \le x_2$,
\begin{align}
\label{claim2}
\bP&\big( \hat \bS_n   = \x  , M_n^{(1)} \le y_1, M_n^{(2)} \le y_2 \big) \\
&\leq \frac{C_7}{a_n^{(1)} a_n^{(2)}} 
\bP\big( \hat S_{\lfloor n/4 \rfloor}^{(1)} \ge x_1/16 , M_{\lfloor n/4 \rfloor}^{(1)} \le y_1 \big)^{1/2} \bP\big( \hat S_{\lfloor n/4 \rfloor}^{(2)} \ge x_2/16 , M_{\lfloor n/4 \rfloor}^{(2)} \le y_2 \big)^{1/2}
\, .
\notag
\end{align}
\end{claim}

Then we can use Theorem~\ref{thm:fuknagaev} to control the probabilities in the right-hand sides.

\begin{proof}[Proof of Claim~\ref{claim:localfuknag}]
We prove only \eqref{claim2}, the proof of \eqref{claim1} being identical as that of \eqref{localfuknag}.
We decompose the probability into four parts, according to whether  $S_{\lfloor n/2 \rfloor}^{(i)} - \tfrac12 \lfloor b_n^{(i)} \rfloor \ge x_i/2$ or not, for $i=1,2$: there are two terms we need to control (the other two being symmetric).

(1) The first term we need to control is
\begin{align*}
\bP&\big(\hat \bS_n = \x  , M_n^{(1)} \le y_1, M_n^{(2)} \le y_2 ,  S_{\lfloor n/2 \rfloor}^{(i)} - \tfrac12 \lfloor b_n^{(i)} \rfloor \ge \tfrac12  x_i \ \text{for } i=1,2\big) \\
& \leq  \sumtwo{z_1\ge  \frac12 x_1 + \frac12 \lfloor b_n^{(1)}\rfloor}{z_2 \ge \frac12 x_2 + \frac12 \lfloor b_n^{(2)}\rfloor }
\bP\big( \bS_{\lfloor n/2 \rfloor} = (z_1,z_2) , M_{\lfloor n/2 \rfloor}^{(1)} \le y_1 , M_{\lfloor n/2 \rfloor}^{(2)} \le y_2 \big) \\[-0.9cm]
& \hspace{7cm} \times\bP \big(  \bS_n- \bS_{\lfloor n/2 \rfloor} = (x_1-z_1 ,x_2-z_2) \big) \\[0.3cm]
& \le \frac{C}{a_n^{(1)} a_n^{(2)}} \bP\big( S_{\lfloor n/2 \rfloor}^{(1)} \ge  \tfrac12 x_1 + \tfrac12 \lfloor b_n^{(1)}\rfloor, S_{\lfloor n/2 \rfloor}^{(2)} \ge  \tfrac12 x_2 + \tfrac12 \lfloor b_n^{(2)}\rfloor , M_{\lfloor n/2 \rfloor}^{(1)} \le y_1 , M_{\lfloor n/2 \rfloor}^{(2)} \le y_2 \big)\, .
\end{align*}
For the last inequality, we used the local limit theorem~\eqref{LLT} to bound the last probability by $C/(a_n^{(1)} a_n^{(2)})$ uniformly in $x_1,x_2, z_1, z_2$, and then summed over $z_1,z_2$.
Then, we use Claim~\ref{claim:dn} to get that, provided $x_i \geq C_6 a_n^{(i)}$ with $C_{6}$ large enough, the last probability is bounded by
\begin{align*}
\bP\big( \hat S_{\lfloor n/2 \rfloor}^{(1)} \ge  \tfrac14 x_1, &\hat S_{\lfloor n/2 \rfloor}^{(2)} \ge  \tfrac14 x_2   , M_{\lfloor n/2 \rfloor}^{(1)} \le y_1 , M_{\lfloor n/2 \rfloor}^{(2)} \le y_2 \big)\\
&\le \bP\big( \hat S_{\lfloor n/2 \rfloor}^{(1)} \ge  \tfrac14 x_1, M_{\lfloor n/2 \rfloor}^{(1)} \le y_1 \big)^{1/2} \bP\big( \hat S_{\lfloor n/2 \rfloor}^{(2)} \ge  \tfrac14 x_2   , M_{\lfloor n/2 \rfloor}^{(2)} \le y_2 \big)^{1/2}\,  ,
\end{align*}
where we used Cauchy-Schwarz inequality at last.

(2) The second term we need to control is
\begin{align}
 \label{onesmallonelarge} 
\bP&\big( \hat \bS_n = \x  , M_n^{(1)} \le y_1, M_n^{(2)} \le y_2 ,  S_{\lfloor n/2 \rfloor}^{(1)} - \tfrac12 \lfloor b_n^{(1)} \rfloor \ge \tfrac12  x_1, S_{\lfloor n/2 \rfloor}^{(2)} - \tfrac12 \lfloor b_n^{(2)} \rfloor < \tfrac12 x_2 \big)  \\
& \leq  \sumtwo{z_1\ge  \frac12 x_1 + \frac12 \lfloor b_n^{(1)}\rfloor}{z_2 <  \frac12 x_2 + \frac12 \lfloor b_n^{(2)}\rfloor }
\bP\big(  \bS_{\lfloor n/2 \rfloor} = (z_1,z_2) , M_{\lfloor n/2 \rfloor}^{(1)} \le y_1 \big) \notag\\[-0.6cm]
& \hspace{5cm}\times \bP \big(  \bS_n- \bS_{\lfloor n/2 \rfloor} = (x_1-z_1 ,x_2-z_2), \max_{\lfloor n/2\rfloor \leq i\leq n} X_{i}^{(2)} \le y_2  \big) \, .\notag
\end{align}
Then, we can use Proposition~\ref{prop:CD}, say for the second probability: indeed, we have that uniformly for the range of $z_2$ considered,
\begin{align*}
\bP \big(  \bS_n- \bS_{\lfloor n/2 \rfloor} &= (x_1-z_1 ,x_2-z_2), \max_{\lfloor n/2\rfloor \leq i\leq n} X_{i}^{(2)} \le y_2  \big) \\
& = \bP \big(  \hat \bS_{\lfloor n/2 \rfloor} = (x_1-z_1-\lfloor b_{n/2}^{(1)} \rfloor ,x_2-z_2- \lfloor b_{n/2}^{(2)} \rfloor )  , M_{n-\lfloor n/2 \rfloor}^{(2)} \le y_2  \big)\\
&\le  \frac{C}{a_n^{(1)} a_n^{(2)}} \bP\big(  \hat S_{\lfloor n/4 \rfloor}^{(2)}  \ge x_2/16 , M_{\lfloor n/4 \rfloor }^{(2)} \le y_2\big)
\end{align*}
where we used that $x_2-z_2 - \lfloor b_{n/2}^{(2)} \rfloor \ge x_2/4 \ge C_4 a_n^{(2)}$ (thanks to Claim~\ref{claim:dn}).
Using this in \eqref{onesmallonelarge} and summing over $z_1$ and $z_2$ (and using again Claim~\ref{claim:dn}), we finally get that  \eqref{onesmallonelarge} is bounded by
\[\frac{C}{a_n^{(1)} a_n^{(2)}} \bP\big( \hat S_{\lfloor n/2 \rfloor}^{(1)} \ge x_1/4 , M_{\lfloor n/2 \rfloor}^{(1)} \le y_1 \big)   \bP\big(  \hat S_{\lfloor n/4 \rfloor}^{(2)}  \ge x_2/16 , M_{\lfloor n/4 \rfloor }^{(2)} \le y_2\big)\, . \]

\noindent
 Let us stress that, to obtain the statement of  Claim~\ref{claim:localfuknag}, we additionally use  that
\[\bP\big( \hat S_{\lfloor n/2 \rfloor}^{(1)} \ge   x_1/4 , M_{\lfloor n/2 \rfloor}^{(1)} \le y_1 \big) \le 2 \bP\big( \hat S_{\lfloor n/4 \rfloor}^{(1)} \ge   x_1/16, M_{\lfloor n/4 \rfloor}^{(1)} \le y_1 \big) \, .\]
This comes from splitting the left-hand side according to whether
$S_{\lfloor n/4 \rfloor}^{(1)} -\frac12 \lfloor b_{\lfloor n/2   \rfloor} \rfloor  \ge   x_1/8$ or not, and using again Claim~\ref{claim:dn} to get that $|\lfloor b_{n/4} \rfloor  -\frac12 \lfloor b_{\lfloor n/2 \rfloor} \rfloor |\ge  x_1/16$.
\end{proof}

\subsection{Proof of Theorem \ref{thm:locallimit2}}
\label{sec:localharder}

Let us write the details only in dimension $d=2$ to avoid lengthy notations, the proof works identically when $d\ge 3$.
Also, we only deal with  $\x\geq 0$. We fix a constant $C_8$ (large).
The case $x_1\leq C_8\,  a_n^{(1)}, x_2\leq C_8\, a_n^{(2)}$ falls in the range of the local limit theorem~\eqref{LLT}, so we need to consider only  two cases: $x_1 > C_8 \, a_n^{(1)}, x_2 \le C_8\,  a_n^{(2)}$ (the case $x_1\le C_8\, a_n^{(1)}, x_2> C_8\, a_n^{(2)}$ is symmetric) and $x_1 > C_8\, a_n^{(1)}, x_2 > C_8\, a_n^{(2)}$.

\subsubsection{Case $x_1\ge C_8 \, a_n^{(1)}$, $x_2 \le C_8\, a_n^{(2)}$}

We will treat three different contributions, by writing, for some $C_9>0$
\begin{align}
\label{threeterms}
\bP\big( \hat \bS_n = \x\big) = &\ \bP \big( \hat \bS_n =\x, M_n^{(1)} \ge x_1/ 8 \big) \\
&+ \bP \big( \hat \bS_n =\x, M_n^{(1)} \in (C_9\, a_n^{(1)}, x_1/ 8 ) \big) + \bP \big( \hat \bS_n =\x, M_n^{(1)} \le C a_n^{(1)} \big)\, . \notag
\end{align}

For the last term, we  use Proposition~\ref{prop:CD}, together with Theorem~\ref{thm:fuknagaev} (more precisely~\eqref{consequencefuknag}), to get that it is bounded by a constant times
\begin{align*}
\frac{C}{ a_n^{(1)} a_n^{(2)}} \Big( n \gp_1(x_1) x_1^{-\gamma_1} \  e^{- c'' x_1/a_n^{(1)}}  +  e^{- c''  (x_1 /a_n^{(1)})^2} \ind_{\{\ga_1=2\}} \Big)
\end{align*}

Then, we can use that $e^{- c' x_1/a_n^{(1)}} \leq c(a_n^{(1)}/x_1) e^{- c'' x_1/a_n^{(1)}}$ provided that $x_1/a_n^{(1)}$ is large enough (and similarly for the last term), to get that
\begin{equation}
\label{smallmax}
\bP \big( \hat \bS_n =\x, M_n^{(1)} \le C_9 a_n^{(1)} \big)
  \leq \frac{C}{x_1 a_n^{(2)}  } \Big( n \gp_1(x_1) x_1^{-\gamma_1}  +  e^{- c'  (x_1 /a_n^{(1)})^2} \ind_{\{\ga_1=2\}} \Big) \, . 
\end{equation}


\smallskip
In order to treat the  first two terms in \eqref{threeterms},  we control the probability, for $k \in \bbZ$,
\begin{align}
\label{maxk}
&\bP\big(  \hat \bS_n = \x , M_n^{(1)} \in [2^k x_1, 2^{k+1} x_1) \big) \\
&
=n \sum_{u=2^k x_1}^{2^{k+1} x_1} \sum_{v\in \bbZ} \bP(\bX_1 =(u,v)) \bP\big( \hat \bS_{n-1}  = (x_1-u, x_2-v) + \bdelta_n , M_{n-1}^{(1)} \le 2^{k+1}x_1\big) \, ,
\notag
\end{align}
where we set $\bdelta_n:=\lfloor \bb_n \rfloor - \lfloor \bb_{n-1} \rfloor$, which is uniformly bounded by a constant.
By Assumption~\ref{hyp:2}, we get that, for any $u\in [2^k x_1, 2^{k+1}x_1)$ and $v\in \bbZ$,
\begin{align}
\label{bounduv}
\bP( \bX_1 =(u,v)) & \le c  \gp_1(u) u^{-(1+\gamma_1)} \times \frac{1}{1+|v|} h_u^{(1)}(|v|) \\
&\le  c' 2^{-k(1+\gamma_1) + \eta |k|} \gp(x_1) x_1^{-(1+\gamma_1)} 
 \times \frac{h_{2^k x_1}(|v|)}{1+|v|} \, .
 \notag
\end{align}
We used Potter's bound \cite[Thm.~1.5.6]{cf:BGT} to get that for $x_1$ sufficiently large, for every $\eta>0$ there is a constant $c_{\eta}>0$ such that $\gp(2^k x_1) \le c_{\eta} 2^{\eta |k|} \gp(x_1)$ for any $k\in\bbZ$, together with item (iii) in~\eqref{cond:h}.

Since this bound is uniform over $u\in [2^k x_1, 2^{k+1}x_1)$, we may sum over $u$  the last probability in \eqref{maxk}: note that
\begin{align}
\notag
\sum_{u=2^k x_1}^{2^{k+1} x_1} \bP\big( & \hat \bS_{n-1}  = (x_1-u, x_2-v) + \bdelta_n , M_n^{(1)} \le 2^{k+1}x_1\big) \\[-0.2cm]
&
\leq
\begin{cases}
\bP\big( \hat S_{n-1}^{(2)}  = x_2 - v +\delta_n^{(2)} \big) &\quad \tif k \ge -3 \, ,\\
 \bP\big( \hat S_{n-1}^{(1)} \ge \frac12 x_1 , \hat S_{n-1}^{(2)}  = x_2 - v +\delta_n^{(2)} ,  M_n^{(1)} \le 2^{k+1}x_1\big) & \quad \tif k \le -4\, .
\end{cases}
\label{boundsumu}
\end{align}

\textbullet\ When $k\ge -3$, we therefore get from~\eqref{maxk} that (taking $\eta<\gamma_1$ in \eqref{bounduv})
\begin{align}
\bP  \big(  \hat \bS_n = \x  , M_n^{(1)} &\in [2^k x_1, 2^{k+1} x_1) \big) \\
&\le c' n 2^{-k} \gp(x_1) x_1^{-(1+\gamma_1)} 
\sum_{v\in \bbZ} \frac{h_{2^k x_1}(|v|)}{1+|v|} \bP\big( \hat S_{n-1}^{(2)}  = x_2 - v +\delta_n^{(2)} \big) \notag \\
& \le c' n 2^{-k} \gp(x_1) x_1^{-(1+\gamma_1)} \times \frac{1}{a_n^{(2)}}  \, .
\notag 
\end{align}
We used the local limit theorem to get that there is a constant $C$ such that for any $z\in \bbZ$, $\bP(\hat S_{n-1}^{(2)} = z) \le C/a_n^{(2)}$, and then that $\sum_{v\in \bbZ} h_{2^k x_1}(|v|)/(1+|v|) \le C$ for some constant $C$ not depending on $k$ or $x_1$, thanks to item (ii) in~\eqref{cond:h}. From this, we obtain that
\begin{align}
\label{klarge}
\bP  \big(  \hat \bS_n = \x ,   M_n^{(1)} \ge x_1/8 \big)  &=\sum_{k\ge -3} \bP  \big(  \hat \bS_n = \x , M_n^{(1)} \in [2^k x_1, 2^{k+1} x_1) \big)  \\
& \le \frac{C'}{a_n^{(2)}} \, n \gp(x_1) x_1^{-(1+\gamma_1)}   \, .
\notag
\end{align}

\textbullet\ When $k\le -4$,  we use Claim~\ref{claim:localfuknag}  in \eqref{boundsumu}, so that plugged in~\eqref{maxk} we obtain that
\begin{align*}
\bP & \big(  \hat \bS_n = \x , M_n^{(1)} \in [2^k x_1, 2^{k+1} x_1) \big)
\\
& \le c n 2^{-k(1+\gamma_1+ \eta)} \gp(x_1) x_1^{-(1+\gamma_1)} \sum_{v\in \bbZ}  \frac{h_{2^k x_1}(v)}{1+|v|} \,  \frac{1}{a_n^{(2)}}\bP\big( \hat S_{\lfloor (n-1)/2\rfloor } \ge \tfrac18 x_1 , M_n^{(1)} \le 2^{k+1} x_1 \big)\\
& \le \frac{ c }{a_n^{(2)}}   n \gp(x_1) x_1^{-(1+\gamma_1)}  2^{-k(2+\gamma_1)} \bP\big( \hat S_{\lfloor (n-1)/2\rfloor } \ge x_1/8 , M_n^{(1)} \le 2^{k+1} x_1 \big) \, ,
\end{align*}
where we used again item (ii) in \eqref{cond:h} to bound $\sum_{v\in \bbZ} h_{2^k x_1}(|v|)/(1+|v|)$ by a (uniform) constant, and took $\eta=1$.
Then, we can use Theorem~\ref{thm:fuknagaev} to get that there are constants $c,c'$ such that
uniformly for $k \le -4$ with $2^k x_1 \ge C_9 a_n^{(1)}$,
\begin{align}
\label{Fuksimple}
\bP&\big( \hat S_{\lfloor (n-1)/2\rfloor } \ge  x_1/8 , M_n^{(1)} \le 2^{k+1} x_1 \big) \\
&\le \Big( \frac{c}{x_1}  n \gp_1(2^{k} x_1) (2^{k} x_1)^{1-\gamma_1}  \Big)^{c' 2^{-k}} + e^{-c' x_1/a_n^{(1)}}\ind_{\{\ga_1=1\}} + e^{- c' x_1^2/ n \sigma_1(2^{k} x_1)} \ind_{\{\ga_1=2\}}
\notag \\
&\le \big( c'' 2^{-k} \big)^{- c' 2^{-k}}  + e^{-c' x_1/a_n^{(1)} } \, .
\notag
\end{align}
Indeed, we used that since $2^{k} x_1 \ge C_9 a_n^{(1)}$, we have that $n \gp_1(2^{k} x_1) (2^{k} x_1)^{-\gamma_1}$ is bounded by a constant. Also, in the case $\ga_1=2$, we used that  $\sigma_1(2^{k} x_1) \le \sigma_1(x_1)$, and that by definition of $a_n^{(1)}$ we have $x_1^2/n\sigma(x_1) \ge c (x_1/a_n^{(1)})^2 \sigma_1(a_n^{(1)})/\sigma_1(x) \ge c x_1/a_n^{(1)}$ (the last inequality comes from Potter's bound).

Therefore, summing over $k$ between $-4$ and $- \lfloor \log_2 ( x_1/C_9a_n^{(1)}) \rfloor$, we finally obtain that 
$\bP \big( \hat \bS_n =\x, M_n^{(1)} \in (C a_n^{(1)}, x_1/ 8 ) \big) $ is bounded by a constant times $ \frac{ n}{a_n^{(2)}} \gp(x_1) x_1^{-(1+\gamma_1)}$, times
\begin{align}
\label{sumoverk}
\sum_{k'=4}^{\lfloor \log_2(  x_1/ C_9 a_n^{(1)}) \rfloor } 2^{k' (2+\gamma_1)} \Big( \big( c' 2^{k'} \big)^{- c 2^{k'}}  + e^{-c x_1/a_n^{(1)}} \Big)
\le C+ \Big( \frac{ c x_1}{a_n^{(1)}}\Big)^{3+\gamma_1} e^{-c x_1/a_n^{(1)}} \, .
\end{align}
Note that the second term is bounded by a constant, uniformly for $x_1/a_n^{(1)} \ge C$.
Therefore, we conclude that
\begin{equation}
\label{kmedium}
\bP \big( \hat \bS_n =\x, M_n^{(1)} \in (C a_n^{(1)}, x_1/ 8 ) \big)
\le \frac{ C }{a_n^{(2)}}  n \gp(x_1) x_1^{-(1+\gamma_1)}\, .
\end{equation}

As a conclusion, \eqref{threeterms}, combined with \eqref{smallmax}, \eqref{klarge} and \eqref{kmedium}, gives that
\begin{align}
\notag
\bP \big( \hat \bS_n =\x \big) & \le  \frac{C}{ a_n^{(2)}} \Big( n \gp_1(x_1) x_1^{-(1+\gamma_1)}  +  \frac{1}{a_n^{(1)}}  e^{- c (x_1 /a_n^{(1)})^2} \ind_{\{\ga_1=2\}} \Big) \\
&\le \frac{C}{x_1 \, a_n^{(2)}}  \Big( n \gp_1(x_1) x_1^{-\gamma_1}  +   e^{- c' (x_1 /a_n^{(1)})^2} \ind_{\{\ga_1=2\}} \Big) \, .
\end{align}
Notice that, in the case $\ga_1<2$, we have $\gamma_1=\ga_1$, and $n   \sim (a_n^{(1)})^{\ga_1} \gp_1(a_n^{(1)})^{-1}$, so that the second term is negligible, since the first term is bounded below by a power of $x_1/a_n^{(1)}$.
We also used that $(x_1/a_n^{(1)}) \exp(- c x_1/a_n^{(1)})$ is bounded by a constant times $\exp(- c' x_1/a_n^{(1)})$ with $c'>c$, provided that $x_1\geq C_9 a_n^{(1)}$.

\subsubsection{Case $x_1\ge C_8 a_n^{(1)}$, $x_2 \ge  C_8 a_n^{(2)}$}

Again, we decompose the probability according to the value of $M_n^{(1)},M_n^{(2)}$. As a first step, we write 
\begin{align}
\notag
\bP\big( \hat \bS_n =\x \big)
= \bP\big( \hat \bS_n =\x ,  M_n^{(i)} &\le C_9\, a_n^{(i)} \  i=1,2\big) +  \bP\big( \hat \bS_n =\x , M_n^{(i)} > C_9\,  a_n^{(i)} \  i=1,2\big)\\
& + \bP\big( \hat \bS_n =\x , M_n^{(1)} > C_9\,  a_n^{(1)}, M_n^{(2)} \le  C_9\,  a_n^{(2)} \big) \notag\\
 &\qquad  +\bP\big( \hat \bS_n =\x , M_n^{(1)} \le  C_9\,  a_n^{(1)}, M_n^{(2)} >  C_9\,  a_n^{(2)} \big) \, .
\label{fourterms}
\end{align}

\smallskip
{\bf Term 1.} Let us bound the first term in \eqref{fourterms}.
 We use Claim~\ref{claim:localfuknag} (more precisely \eqref{claim2}), together with Theorem~\ref{thm:fuknagaev} (more precisely~\eqref{consequencefuknag}) to get that
\begin{align}
\label{term1}
\bP\big(  &\hat \bS_n =\x , M_n^{(i)} \le C_{9}\,  a_n^{(i)} \  i=1,2\big) \\
&\le \frac{C}{a_n^{(1)} a_n^{(2)}} \Big(   (n \gp_1(x_1) x_1^{-\gamma_1} )^2 e^{ - c'' x_1/a_n^{(1)}}  +  e^{- c'' (x_1/a_n^{(1)})^2 } \ind_{\{\ga_1=2\}} \Big)^{1/2} \Big( e^{- c x_2/a_n^{(2)}}  \Big)^{1/2}  \notag\\
& \leq \frac{C'}{x_1 x_2} \Big( n \gp_1(x_1) x_1^{-\gamma_1} +  e^{- c' (x_1/a_n^{(1)})^2 } \ind_{\{\ga_1=2\}} \Big)\, . \notag
\end{align}
Note that we also used that \eqref{consequencefuknag} is also bounded by $\exp(-c x_2/a_n^{(2)})$ for the first inequality.
Then, we used that $(a+b)^{1/2}\leq a^{1/2}+ b^{1/2}$ for any $a, b \geq 0$, and then that $\frac{1}{a_n^{(i)}} e^{-c x_i/a_n^{(i)}} \le \frac{1}{x_i} e^{-c' x_i/a_n^{(i)}}$ for $x_i/a_n^{(i)}$  large.

\smallskip
{\bf Term 3.} 
 We now bound the third term  in~\eqref{fourterms} by a constant times $ (x_1 x_2)^{-1} n \gp_1(x_1) x_1^{-\gamma_1}$.
We proceed as for the previous section \eqref{maxk}--\eqref{kmedium}. The analogous of \eqref{maxk} is, for $k\in \mathbb{Z}$
\begin{align*}
&\bP\big( \hat \bS_n =\x , M_n^{(1)} \in [2^k x_1 ,2^{k+1} x_1), M_n^{(2)} \le  C_9 \, a_n^{(2)} \big) \\
& \leq  n \sum_{u=2^k x_1}^{2^{k+1} x_1} \sum_{v\le C_9 a_n^{(2)}} \bP(\bX_1=(u,v))\\[-0.5cm]
& \hspace{3.8cm} \bP\big( \hat \bS_{n-1} =(x_1-u,x_2-v)+\bdelta_n , M_{n-1}^{(1)} \le 2^{k+1} x_1 , M_{n-1}^{(2)} \le C_9 a_n^{(2)} \big) \, .
\end{align*} 
Then, one bounds $\bP(\bX_1=(u,v))$ by using Assumption~\ref{hyp:2} (as in \eqref{bounduv}), and by summing over $u \in [2^k x_1, 2^{k+1} x_1 )$ one needs to control (analogously to \eqref{boundsumu})
\begin{equation}
\label{remainstobound}
\begin{split}
\tif k \ge -3, \quad &\bP\big( \hat S_{n-1}^{(2)} = x_2 - v + \gd_n^{(2)} , M_{n-1}^{(2)} \le C_9 a_n^{(2)} \big)  \, ,\\
\tif k \le -4 , \quad &\bP\big( \hat S_{n-1}^{(1)} \ge  x_1/2 , \hat S_{n-1}^{(2)} = x_2 - v + \gd_n^{(2)} , M_{n-1}^{(1)} \le 2^{k+1} x_1, M_{n-1}^{(2)} \le C_9 a_n^{(2)} \big) \, ,
\end{split}
\end{equation}
uniformly over $v \le C_9 a_n^{(2)}$.

The first probability in~\eqref{remainstobound} is treated by using Proposition~\ref{prop:CD}, together with~\eqref{consequencefuknag} (and the remark below): 
since $x_2-v +\gd_n^{(2)}$ is bounded below by $x_2/2$ uniformly in the range of $v$ considered (and assuming that $C_9<C_8/2$), we get that 
\begin{equation}
\label{term2-k>-3}
\bP\big( \hat S_{n-1}^{(2)} = x_2 - v + \gd_n^{(2)} , M_{n-1}^{(2)} \le C_9 a_n^{(2)} \big) \le \frac{C}{a_n^{(2)}} e^{-c x_2/a_n^{(2)}}  \leq \frac{C}{x_2}\, .
\end{equation}
We used the fact that $x_2 \geq C_8 a_n^{(2)}$ for the last inequality.
Hence, we get that for $k\ge -3$
\begin{align*}
\bP\big( \hat \bS_n = \x , M_n^{(1)} \in [2^k x_1& ,2^{k+1} x_1), M_n^{(2)} \le  C_9 a_n^{(2)} \big)\\
&\le \frac{C}{x_2}\ 2^{-k(1+\gamma_1)+ \eta |k|} n \gp(x_1) x_1^{-(1+\gamma_1)} \times \sum_{v\in \bbZ^d} \frac{h_{2^k x_1}(|v|)}{1+|v|} \, ,
\end{align*}
and the last sum is bounded by a constant uniform in $k,x_1$, thanks to item (ii) in ~\eqref{cond:h}.
Summing over $k \ge -3$, we get that, analogously to~\eqref{klarge},
\begin{equation}
\label{endterm3--1}
\bP\big( \hat \bS_n =  \x , M_n^{(1)} \ge x_1/8, M_n^{(2)} \le  C_9 \,  a_n^{(2)} \big) \le  \frac{C}{x_2}  n \gp(x_1) x_1^{-(1+\gamma_1)}  \, .
\end{equation}

For the second probability in~\eqref{remainstobound} (with $k\le -4$), we invoke Claim~\ref{claim:localfuknag}: one can easily adapt the proof of~\eqref{claim2}, using that $x_2 - v + \gd_n^{(2)} \ge x_2/2$ uniformly for the range of $v$ considered, to get that
\begin{align}
\bP\big( \hat S_{n-1}^{(1)} &\ge   x_1/2, \hat S_{n-1}^{(2)} = x_2 - v + \gd_n^{(2)} , M_{n-1}^{(1)} \le 2^{k+1} x_1, M_{n-1}^{(2)} \le C_9 \, a_n^{(2)} \big)   \notag \\
&\le \frac{C}{a_n^{(2)}} \bP \big( \hat S_{\lfloor n/4 \rfloor}^{(1)} \ge x_1/32 ,  M_{\lfloor n/4 \rfloor}^{(1)} \le 2^{k+1} x_1\big)^{1/2} 
 \bP\big( \hat S_{\lfloor n/4 \rfloor}^{(2)} \ge x_2/32 , M_{\lfloor n/4 \rfloor}^{(2)} \le C_9 a_n^{(2)} \big)^{1/2} \notag \\
 & \le \frac{C}{a_n^{(2)}} \Big(  \big( c 2^{-k}\big)^{c 2^{k}} + e^{- c x_1/a_n^{(1)} }\Big)  e^{- c x_2/a_n^{(2)}}
 \leq \frac{C}{x_2} \Big(  \big( c 2^{-k}\big)^{c' 2^{k}} + e^{- c' x_1/a_n^{(1)} }\Big) \, .
 \label{term2-k<-4}
\end{align}
For the second inequality, we used Theorem~\ref{thm:fuknagaev}, more precisely~\eqref{Fuksimple}.
Therefore, we obtain that for $k\le -4$ with $2^k x_1\ge C_9 a_n^{(1)}$,
\begin{align*}
\bP\big( \hat \bS_n &=  \x , M_n^{(1)} \in [2^k x_1 ,2^{k+1} x_1), M_n^{(2)} \le  C_9 a_n^{(2)} \big)\\
&\le \frac{C}{x_2} 2^{-k(1+\gamma_1)+ \eta |k|} n \gp(x_1) x_1^{-(1+\gamma_1)} \Big(  \big( c 2^{-k}\big)^{c' 2^{k}} + e^{- c' x_1/a_n^{(1)} }\Big) \times \sum_{v\in \bbZ^d} \frac{h_{2^k x_1}(v)}{1+|v|} \, ,
\end{align*}
with the last sum bounded by a constant uniform in $k,x_1$.
Summing over $k$ between $-4$ and $-\lfloor \log_2( x_1/C_9 a_n^{(1)} ) \rfloor$ (as done in~\eqref{sumoverk}), we get that
\begin{equation}
\label{endterm3-2}
\bP\big( \hat \bS_n =  \x , M_n^{(1)} \in (C_9\, a_n^{(1)}, x_1/8), M_n^{(2)} \le  C_9\, a_n^{(2)} \big) 
\le \frac{C}{x_2} n \gp_1(x_1) x_1^{-(1+\gamma_1)} \, .
\end{equation}

To conclude, we have that
\begin{align}
\label{term3}
\bP\big( \hat \bS_n =  \x , M_n^{(1)} \ge  C_9\, a_n^{(1)}, M_n^{(2)} \le  C_9\, a_n^{(2)} \big) &\le \frac{C}{x_2} n \gp_1(x_1) x_1^{-(1+\gamma_1)} \, .
\end{align}

\smallskip
{\bf Term 4.} 
 We now bound the fourth term in \eqref{fourterms}. We stress that the treatment is not completely symmetric to that of Term 3, since we wish to obtain a bound that depends on the tail of the first coordinate (\textit{i.e.}\ on $\gp_1(\cdot)$ and $\gamma_1$), whereas \eqref{term3} above yields the bound $\frac{C}{x_1} n \gp_2(x_2) x_1^{-(1+\gamma_2)}$.
We however proceed analogously: we control
\[ \bP\big( \hat \bS_n =  \x , M_n^{(1)} \le C_9\, a_n^{(1)}, M_n^{(2)} \in [2^k x_1, 2^{k+1}x_1] \big) \, . \]
Then, for $k\ge -3$, instead of~\eqref{term2-k>-3}, we use Proposition~\ref{prop:CD} together with~\eqref{consequencefuknag} to get that
\begin{equation}
\bP(   \hat S^{(1)}_{n-1} =  x_1-v +\gd_n^{(1)} , M_{n-1}^{(1)} \le C_9\, a_n^{(1)} ) \leq \frac{C}{x_1} \Big( n \gp(x_1) x_1^{-\gamma_1}  + e^{- c (x_1/a_n^{(1)})^2} \Big)\, .
\end{equation}
We end up with, analogously to \eqref{endterm3--1},
\begin{align*}
\bP( \hat \bS_n =  \x , M_n^{(1)} \le C_9\,  a_n^{(1)}, M_n^{(2)} &\geq x_2/8 ) \\
&\leq \frac{C}{x_1} \Big( n \gp(x_1) x_1^{-\gamma_1}  + e^{- c (x_1/a_n^{(1)})^2} \Big) \,  n \gp(x_2) x_2^{-(1+\gamma_1)} \, .
\end{align*}
Also, for $k\leq -4$, instead of \eqref{term2-k<-4}, we get 
\begin{align*}
\bP\big(\hat S_{n-1}^{(1)} = x_1 - v + \gd_n^{(1)} &, \hat S_{n-1}^{(2)} \ge  x_2/2,   M_{n-1}^{(1)} \le C_9\, a_n^{(2)} , M_{n-1}^{(2)} \le 2^{k+1} x_1,\big)   \notag \\
&\le  \frac{C}{a_n^{(1)}} \Big(   n \gp(x_1) x_1^{-\gamma_1}  + e^{- c (x_1/a_n^{(1)})^2} \Big) \Big(  \big( c 2^{-k}\big)^{c' 2^{-k}} + e^{- c' x_2/a_n^{(2)} }\Big)\, , 
\end{align*}
and, analogously to \eqref{endterm3-2}, we obtain
\begin{align*}
\bP \big( \hat \bS_n =  \x  , M_n^{(1)} \le C_9\, a_n^{(1)}, M_n^{(2)} &\in (C_9\, a_n^{(2)} , x_2/8) \big) \\
&\leq \frac{C}{x_1} \Big( n \gp(x_1) x_1^{-\gamma_1}  + e^{- c (x_1/a_n^{(1)})^2} \Big) \,  n \gp(x_2) x_2^{-(1+\gamma_1)}\, .
\end{align*}
All together, and since $n \gp(x_2) x_2^{-\gamma_1}$ is bounded by  a constant (since $x_2\geq C_8\, a_n^{(2)}$), we obtain
\begin{align}
\label{term4}
\bP\big( \hat \bS_n =  \x , M_n^{(1)} \le C_9\, a_n^{(1)}, M_n^{(2)} \ge  C_9\, a_n^{(2)} \big) 
& \le \frac{C}{x_1 x_2} \, \Big( n \gp(x_1) x_1^{-\gamma_1}  + e^{- c (x_1/a_n^{(1)})^2} \Big) \, .
\end{align}

\smallskip
{\bf Term 2.} It remains to deal with the second term in~\eqref{fourterms}, which is the most technical.
We will estimate the probabilities, for $k,j\in \mathbb{Z}$
\begin{align}
\label{lastprobability}
\bP\big( \hat \bS_n =  \x , M_n^{(1)} \in [2^k x_1,2^{k+1}x_1) , M_n^{(2)}  \in [2^j x_2, 2^{j+1} x_2) \big) =: P_1(k,j) + P_2(k,j)\, . 
\end{align}
Here, we split the probability into two contributions: either the two maxima in $M_n^{(1)}, M_n^{(2)}$ are attained in one increment (with both coordinates large), see~\eqref{probasamejump}, or the two maxima are attained by separate increments, see~\eqref{probatwojump}.

{\it Part 1.} The first contribution is, using a union bound and the exchangeability of the $\bX_i$'s
\begin{align}
\label{probasamejump}
&P_1(k,j):=\bP\Big( \hat \bS_n =  \x, \exists i\in\llbracket 1, n \rrbracket \text{ s.t. } \bX_i \in [2^k x_1,2^{k+1}x_1) \times [2^j x_2, 2^{j+1} x_2) ,\\
& \hspace{8cm} M_n^{(1)} \le 2^{k+1}x_1 , M_n^{(2)} \le 2^{j+1} x_2 \Big)  \notag\\[-0.5cm]
&\le n \sum_{u= 2^k x_1}^{2^{k+1} x_1} \sum_{v=2^{j} x_2}^{2^{j+1} x_2} \bP(\bX_1=(u,v)) \notag\\
&\hspace{1.9cm} \times\bP\big( \hat \bS_{n-1} = (x_1-u,x_2-v) +\bdelta_n , 
 M_{n-1}^{(1)} \le 2^{k+1}x_1 , M_{n-1}^{(2)} \le 2^{j+1} x_2  \big)\, .
 \notag
\end{align}
Then we use Assumption~\ref{hyp:2} (item (i) in \eqref{cond:h}) to get that there is a constant $C$ such that for any $j,k$, and any $(u,v)\in [2^k x_1,2^{k+1}x_1) \times [2^j x_2, 2^{j+1} x_2)$, we have 
\begin{align}
\label{probauv}
\bP(\bX_1=(u,v)) &\le C \gp_1(2^k x_1) (2^k x_1)^{-(1+\gamma_1)} (2^j x_2)^{-1}\\
& \le c 2^{-k(1+\gamma_1)+\eta |k|} 2^{-j} \frac{\gp_1(x_1) x_1^{-(1+\gamma_1)}}{x_2} \, . \notag
\end{align}
Therefore, in~\eqref{probasamejump}, we can sum over $u,v$ the last probability, and we treat it differently according to whether $k\ge -3$ or not and $j\ge -3$ or not (similarly to \eqref{boundsumu}): after summation over $u,v$, we obtain the following upper bound
\begin{align*}
\tif k\le -4, j\le -4, \quad &  \bP\big( \hat S_{n-1}^{(1)} \ge x_1/2 , \hat S_{n-1}^{(2)} \ge x_2/2 , M_{n-1}^{(1)} \le 2^{k+1} x_1 , M_{n-1} \le 2^{j+1} x_2 \big) \, , \\
\tif k\le -4, j\ge -3, \quad &  \bP\big( \hat S_{n-1}^{(1)} \ge x_1/2 , M_{n-1}^{(1)} \le 2^{k+1} x_1 \big) \, , \\
\tif k\ge -3, j\le -4, \quad &   \bP\big( \hat S_{n-1}^{(2)} \ge x_2/2 , M_{n-1} \le 2^{j+1} x_2 \big) \, ,\\
\tif k\ge -3, j\ge -3, \quad & 1.  
\end{align*}
Then, we can use Theorem~\ref{thm:fuknagaev} to get that for $k \le -4$ with $2^k x_1\ge C_9\, a_n^{(1)}$ we have, with the same argument as for~\eqref{Fuksimple},
\begin{equation}
\label{Fuksimple2}
 \bP\big( \hat S_{n-1}^{(1)} \ge x_1/2 , M_{n-1}^{(1)} \le 2^{k+1} x_1 \big)  \le \big( c 2^{-k}\big)^{- c' 2^{-k}} + e^{- c' x_1/a_n^{(1)}} \, ,
\end{equation}
and similarly for the second coordinate.
 In the case $k\le -4$, $j\le -4$, Cauchy-Schwarz inequality allows us to to reduce to this estimate.

Going back to~\eqref{probasamejump}, and using~\eqref{probauv}, in the case $k, j\geq -3$ we get that
\begin{align}
\sum_{k,j=-3}^{+\infty} P_1(k,j) \le  \sum_{k,j=-3}^{+\infty}  \frac{C}{x_2} n \gp_1(x_1) x_1^{-(1+\gamma_1)}  2^{-k} 2^{-j} \leq \frac{C'}{x_2}  n  \gp_1(x_1) x_1^{-(1+\gamma_1)}\, .
\end{align}
In the case $k \le -4$, $j\ge -3$ (the case $k\ge -3$, $j\le -4$ is symmetric), we get that
\begin{equation*}
P_1(k,j) \le \frac{C}{x_2} n \gp_1(x_1) x_1^{-(1+\gamma_1)} 2^{-k(2+\gamma_1)} 2^{-j} \big( ( c 2^{-k} )^{- c' 2^{-k}} + e^{- c' x_1/a_n^{(1)}} \big)\, .
\end{equation*}
Hence, we obtain (the calculation is analogous to that in~\eqref{sumoverk})
\begin{align}
\label{sumPkj}
\sum_{k=- \lfloor \log_2( x_1/C_9 a_n^{(1)}) \rfloor }^{-4} \sum_{j=-3}^{+\infty} P_1(k,j) &\le 
\frac{C}{x_2} n  \gp_1(x_1) x_1^{-(1+\gamma_1)} \Big( C +  \Big(\frac{c x_1}{a_n^{(1)}} \Big)^{3+\gamma_1}e^{- c x_1/a_n^{(1)}}\Big) \, .
\end{align}
In  the case $k \le -4$, $j\le -4$, we get that
\begin{align*}
P_1(k,j) \le \frac{C}{x_2} n & \gp_1(x_1) x_1^{-(1+\gamma_1)}\\
&\times   2^{-k(2+\gamma_1)} 2^{-j}  \big(  ( c 2^{-k} )^{- c 2^{-k}} + e^{- c x_1/a_n^{(1)}} \big) \big(  ( c 2^{-j} )^{- c' 2^{-j}} + e^{- c' x_1/a_n^{(2)}} \big)\, ,
\end{align*}
and a similar calculation as above gives
\begin{align}
\label{sumPkj2}
\sum_{k=- \lfloor \log_2( x_1/C_9 a_n^{(1)}) \rfloor}^{-4} \sum_{j=- \lfloor \log_2( x_2/C_9 a_n^{(1)}) \rfloor}^{-4} P_1(k,j) &\le 
\frac{C}{x_2}  n\gp_1(x_1) x_1^{-(1+\gamma_1)}  \, .
\end{align}

All together, we obtain that 
\begin{equation}
\label{sumP1kj}
\sum_{k \geq - \log_2( x_1/C_9 a_n^{(1)}) } \sum_{j \geq -  \log_2( x_2/C_9 a_n^{(1)}) }   P_1(k,j)\le 
\frac{C}{x_1 x_2}  n\gp_1(x_1) x_1^{-\gamma_1} \, .
\end{equation}

{\it Part 2.} It remains to control the contribution when the maxima in $M_n^{(1)}, M_n^{(2)}$ are attained by separated increments, \textit{i.e.}
\begin{align}
\label{probatwojump}
&P_2(k,j):=\bP\Big( \hat \bS_n =  \x, \exists i\neq\ell  \in\llbracket 1, n \rrbracket, \text{ s. t. } X_i^{(1)} \in [2^k x_1,2^{k+1}x_1), M_n^{(1)} \le 2^{k+1}x_1, \\
& \hspace{7.05cm}  X_\ell^{(2)} \in [ 2^{j} x_2, 2^{j+1} x_2),  , M_n^{(2)} \le 2^{j+1} x_2 \Big)  \notag\\
&\le \binom{n}{2} \sum_{u= 2^k x_1}^{2^{k+1} x_1} \sum_{v \le 2^{j+1} x_2} \sum_{s\le 2^{k+1} x_1} \sum_{t =2^j x_2}^{2^{j+1} x_2}  \bP(\bX_1=(u,v)) \bP(\bX_1= (s,t)) \notag\\
&\quad \ \  \times \bP\big( \hat \bS_{n-2} = \x - (u,v)-(s,t) +\bdelta_n + \bdelta_{n-1} , 
 M_{n-2}^{(1)} \le 2^{k+1}x_1 , M_{n-2}^{(2)} \le 2^{j+1} x_2  \big)\, .
 \notag
\end{align}

Again, we use Assumption~\ref{hyp:2} to bound the first two probabilities: for the ranges of $u,v$ and $s,t$ considered, using item (iii) in \eqref{cond:h}, we have
\begin{equation}
\label{probauvst}
\begin{split}
\bP(\bX_1=(u,v)) & \le    c 2^{-k(1+\gamma_1)+\eta |k|}  \gp_1(x_1) x_1^{-(1+\gamma_1)}    \times \frac{h_{2^k x_1}^{(1)} ( |v|)}{ 1+|v|} \, , \\
\bP(\bX_1=(s,t)) & \le c 2^{-j(1+\gamma_1)+\eta |j|}  \gp_2(x_2) x_2^{-(1+\gamma_2)}    \times \frac{h_{2^k x_2}^{(2)} ( |s|)}{ 1+|s|}\, .
\end{split}
\end{equation}

Then, we may sum the last probability in \eqref{probatwojump} over $u$ and $t$ in the range considered, and get after summation (using also that for the range of $v$ and $s$ considered we  have $v\le 2^{j+1} x_2$, $s\le 2^{k+1} x_1$)
\begin{align*}
\tif k\le -4, j\le -4, \quad &  \bP\big( \hat S_{n-2}^{(1)} \ge x_1/2 , \hat S_{n-2}^{(2)} \ge x_2/2 , M_{n-2}^{(1)} \le 2^{k+1} x_1 , M_{n-2} \le 2^{j+1} x_2 \big) \, , \\
\tif k\le -4, j\ge -3, \quad &  \bP\big( \hat S_{n-2}^{(1)} \ge x_1/2 , M_{n-2}^{(1)} \le 2^{k+1} x_1 \big) \, , \\
\tif k\ge -3, j\le -4, \quad &   \bP\big( \hat S_{n-2}^{(2)} \ge x_2/2 , M_{n-2} \le 2^{j+1} x_2 \big) \, ,\\
\tif k\ge -3, j\ge -3, \quad & 1 ,
\end{align*}
and to treat these terms, we can again use Theorem~\ref{thm:fuknagaev}, in the same way as for~\eqref{Fuksimple2}.
Then we can sum over $v$ and $s$ and use item (ii) in \eqref{cond:h} to get that $\sum_{v} h_{2^k x_1}^{(i)}(|v|) /(1+|v|)$

Going back to~\eqref{probatwojump}, and starting with the case $k,j\geq -3$, we get
\begin{align*}
\sum_{k, j = -3}^{+\infty} P_2(k,j)
&\leq  \sum_{k, j = -3}^{+\infty}  C \binom{n}{2} \gp_1(x_1) x_1^{-(1+\gamma_1)}  \gp_2(x_2) x_2^{-(1+\gamma_2)} 2^{-k} 2^{-j} \\
& \le C'   n \gp_1(x_1) x_1^{-(1+\gamma_1)}  n \gp_2(x_2) x_2^{-(1+\gamma_2)}\, .
\end{align*}

Similarly, and using~\eqref{Fuksimple2}, we get that if $k\le -4, j\ge -3$ (the  case $k\ge -3$, $j\le -4$ is symmetric)
\begin{align*}
P_2(k,j) & \le C'   n \gp_1(x_1) x_1^{-(1+\gamma_1)} n \gp_2(x_2) x_2^{-(1+\gamma_2)} \times 2^{-k(2+\gamma_1)} 2^{- j} \big( ( c 2^{-k})^{- c' 2^{-k}} + e^{- c' x_1/a_n^{(1)}} \big)\, .
\end{align*}
As above (with the same argument as in~\eqref{sumPkj}), we therefore get that
\begin{align*}
\sum_{k=- \lfloor \log_2(x_1/C_9 a_n^{(1)}) \rfloor }^{-4} \sum_{j=-3}^{+\infty} P_2(k,j) \le 
C  n \gp_1(x_1) x_1^{-(1+\gamma_1)} n \gp_2(x_2) x_2^{-(1+\gamma_2)}\, .
\end{align*}
An identical argument holds in the case $k\le -4, j\le -4$, and we end up with
\begin{align}
\label{sumP2kj}
\sum_{k \geq - \log_2( x_1/C_9 a_n^{(1)}) } \sum_{j \geq -  \log_2( x_2/C_9 a_n^{(1)}) } P_2(k,j)&\le 
C n \gp_1(x_1) x_1^{-(1+\gamma_1)} n \gp_2(x_2) x_2^{-(1+\gamma_2)}
\\
&\le \frac{C}{x_1 x_2} n \gp_1(x_1) x_1^{-\gamma_1}  \, . \notag
\end{align}
For the last inequality, we  used that $x_2\geq C_8 a_n^{(2)}$, so that $n\gp_2(x_2) x_2^{-\gamma_2}$ is bounded by a constant, thanks to the definition~\eqref{def:an} of $a_n^{(2)}$.

Therefore, going back to \eqref{lastprobability}, and using \eqref{sumP1kj}-\eqref{sumP2kj}, we obtain that
\begin{equation}
\label{term2}
\bP\big( \hat \bS_n =  \x , M_n^{(1)} \ge C_9\, a_n^{(1)}, M_n^{(2)} \ge  C_9\, a_n^{(2)} \big)  \le  \frac{C}{x_1 x_2}   n \gp_1(x_1) x_1^{-\gamma_1}  \, .
\end{equation}

\smallskip
{\bf Conclusion.}
Let us collect the estimates~\eqref{term1}, \eqref{term3}, \eqref{term4} and \eqref{term2}: plugged into~\eqref{fourterms}, we finally obtain
\begin{align}
\bP\big( \hat S_n = \x \big) \le \frac{C}{x_1 x_2} \Big(n \gp_1(x_1) x_1^{-\gamma_1}  + e^{- c (x_1/a_n^{(1)})^2}  \ind_{\{\ga_1=2\}}     \Big) \, .
\end{align}
This concludes the proof of Theorem~\ref{thm:locallimit2}, since the same bound applies to any coordinate.


\subsection{Proof of Theorem~\ref{thm:locallimit3}}
\label{sec:localWill}

Again, we prove only the case of the dimension $d=2$ for simplicity.
Recall that we work in the balanced case, so we write $a_n\equiv a_n^{(i)}$ and $\ga\equiv\ga_i$.
Let us assume that $|x_{1}|\geq |x_2|$, so that $ c |x_1| \geq \|\x\| \geq |x_1|$ (the other case is symmetric). Suppose also for simplicity that $x_1$ is positive (so we can drop the absolute value), and  $x_1>C_8\,  a_n$.
We write
\begin{align}
 \label{local3-firststep}
\bP\big( \hat \bS_n =\x \big) = 
\bP\big( \hat \bS_n =\x , M_n^{(1)}\geq  x_1/8 \big) + \bP&\big( \hat \bS_n =\x,  M_n^{(1)} \in ( C_9\, a_n,  x_1/8)  \big) \\
& \qquad + \bP\big(\hat \bS_n =\x , M_n^{(1)} \leq C_9\, a_n\big) \, .
\notag
\end{align}

The last term in \eqref{local3-firststep} can be bounded using Proposition~\ref{prop:CD}, together with~\eqref{consequencefuknag}
\begin{align}
\bP\big(\hat \bS_n =\x , M_n^{(1)} \leq C_9\, a_n\big) & \leq \frac{C}{(a_n)^2} \Big( n  \gp(x_1) x_1^{-\gamma_1}  e^{- c x_1/a_n}  + e^{- c (x_1/a_n)^2} \ind_{\{\ga_1=2\}}\Big) \notag \\
& \leq C  n  \gp(x_1) x_1^{-(2+\gamma_1)} +\frac{C}{(a_n)^2} e^{- c (x_1/a_n)^2} \ind_{\{\ga_1=2\}} \, ,
\label{term1Will}
\end{align}
where we used that $e^{- c x_1/a_n} \leq  (a_n/x_1)^2 $ provided that $x_1/a_n \geq C_8$ with $C_8$ large enough.

For the first term in \eqref{local3-firststep},  because of the exchangeability of the $\bX_i$ and thanks to a union bound, we get
\begin{align}
\label{term2Will}
\bP\big( \hat \bS_n& =\x , M_n^{(1)}\geq x_1/8 \big) \leq \sum_{\y \in \bbZ^2, y_1\geq x_1/8} n \bP(\bX_1 = \y) \bP\big( \hat \bS_{n-1} = \x-\y  +\boldsymbol{\delta}_n\big)  \\
& \leq  C \gp(x_1) x_1^{-(2+\gamma)} \sum_{\y \in \bbZ^2, y_1\geq x_1/8}\bP\big( \hat \bS_{n-1} = \x-\y  +\boldsymbol{\delta}_n\big) \leq C \gp(x_1) x_1^{-(2+\gamma)} \,.
\notag
\end{align}
Here, we used \eqref{eq:Will}: $\bP(\bX_1 = \y)$ is bounded by a constant times  $\gp(\|\y\|) \|\y\|^{-(2+\gamma)}$ for the range of $\y$ under summation (and it is bounded by a constant times $\gp(x_1) x_1^{-(2+\gamma)}$).

It remains to control the middle term in \eqref{local3-firststep}. We write
\begin{align}
\bP\big( \hat \bS_n  =\x, \,   M_n^{(1)}& \in ( C_9\, a_n,  x_1/8)  \big) = \sum_{j=3}^{\lfloor \log_2 ( x_1/C_9 a_n) \rfloor} \bP\big( \hat \bS_n =\x, M_n^{(1)} \in [ 2^{-(j+1)} x_1, 2^{-j} x_1)\big) \notag\\
\le   \sum_{j=3}^{\lfloor \log_2 ( x_1/C_9 a_n) \rfloor }&  \sum_{\y \in \bbZ^2, y_1\geq 2^{-(j+1)} x_1} n \bP(\bX_1 = \y)  \bP\big( \hat \bS_{n-1} = \x-\y  +\boldsymbol{\delta}_n , M_n^{(1)} \leq 2^{-j} x_1\big) \notag \\
\le n &\gp(x_1) x_1^{-(2+\gamma)}  \sum_{j=3}^{\lfloor \log_2 ( x_1/C_9 a_n) \rfloor } 2^{(d+1+\gamma)(j+1)} \bP\big( \hat S^{(1)}_{n-1}  \geq x_1/2 , M_n^{(1)} \leq 2^{-j} x_1\big)\, .
\label{local3:term3}
\end{align}
For the last inequality, we used \eqref{eq:Will} to bound $\bP(\bX_1 = \y) \leq c\gp(\|y\|) \|y\|^{- (2+\gamma)}$: this is bounded, for $\y$ with $y_1\geq 2^{-(j+1)} x_1$, by a constant times $\gp(2^{-(j+1)} x_1) x_1^{- (d+\gamma)} 2^{(j+1)(2+\gamma)}$, with $\gp(2^{-(j+1)} x_1) \leq 2^j \gp(x_1)$ thanks to Potter's bound.
Then, for every $j$, the sum over $\y$ with $y_1\geq 2^{-(j+1)} x_1$  gives rise to $\bP\big( \hat S^{(1)}_{n-1}  \geq x_1/2 , M_n^{(1)} \leq 2^{-j} x_1\big)$ (recall that $\gd_n^{(1)}$ is bounded by a constant).

Then, it remains to use  Theorem~\ref{thm:fuknagaev}, more precisely \eqref{Fuksimple}, to get that the last sum in \eqref{local3:term3} is bounded by 
\begin{align*}
\sum_{j=3}^{\lfloor \log_2 ( x_1/C_9 a_n) \rfloor } 2^{(d+1+\gamma)(j+1)} \big(  (c 2^j)^{- c'2^{j}}  +   e^{- c' x_1/a_n} \big) \, ,
\end{align*}
which is bounded by a constant. Therefore we have that
\begin{equation}
\bP\big( \hat \bS_n  =\x, \,   M_n^{(1)} \in ( C_9\, a_n,  x_1/8)  \big) \leq C n \gp(x_1) x_1^{-(2+\gamma)}\, .
\end{equation}
Together with~\eqref{term1Will}-\eqref{term2Will}, this conclude the proof of Theorem~\ref{thm:locallimit3} (recall $x_1 \geq c \|x \|$).


\section{Proof for case I (\emph{centered}): $\mathbf{b}_n \equiv \mathbf{0}$}
\label{sec:casI}

In this section, we prove Theorem~\ref{thm:ga<1} (and Theorem~\ref{thm:marginald=2} in Section~\ref{sec:marginald=2} below).
Recall the definition \eqref{def:ni} of $n_{i}$: it verifies $a_{n_{i}}^{(i)} \sim |x_i|$. 
We work along the favorite direction, that is we assume that $x_i /a_{n_{1}}^{(i)} \to t_i \in \bbR^*$ for any $i\in \{1,\ldots,d\}$, which is equivalent to having $n_i \sim |t_i|^{\ga_i} n_1$ (the reference coordinate is the first one, but this is only for commodity). 
We fix $\gep>0$, and we decompose $G(\x)$ into three subparts:
\begin{equation}
\label{split3-2}
G(\x)  =  \bigg( \sum_{n=1}^{\gep n_{1} -1} +   \sum_{n=\gep n_{1} }^{ \gep^{-1} n_{1} } + \sum_{n = \gep^{-1} n_{1}+1 }^{+\infty}  \bigg) \bP\big( \bS_n =\x \big) \, .
\end{equation}
The middle part gives the main contribution: we treat it first, before we show that the other two parts are negligible. 
In this section and in the rest of the paper, we often omit the integer part: for instance, we do as if $\gep n_1$ and $\gep^{-1}n_1$ were integers.

\subsection{Main contribution}
Because $n_i\sim |t_i|^{\ga_i} n_1$, we have that for $\gep \le  n/n_1 \le \gep^{-1} $ the probability $\bP(\bS_n=\x)$ falls into the range of application of the local limit theorem \eqref{LLT}.
We have
\begin{align*}
 \sum_{n=\gep n_{1} }^{\gep^{-1} n_{1} } \bP\big( \bS_n =\x \big) = \sum_{n=\gep n_{1} }^{\gep^{-1} n_{1} }\frac{1}{a_{n}^{(1)} \cdots a_n^{(d)} }  \, g_{\bga}\Big(  \frac{x_1}{a_n^{(1)}} , \ldots, \frac{x_d}{a_n^{(d)}} \Big)
+\sum_{n=\gep n_{1} }^{\gep^{-1} n_{1} } \frac{o(1) }{a_{n}^{(1)} \cdots a_n^{(d)} }\, .
\end{align*}
The second term is negligible compared to the first one, so we focus on the first term. 
Since $a_n^{(i)}$ is regularly varying with exponent $1/\ga_i$, we get that uniformly for $n/n_1 \in (\gep,\gep^{-1})$,
\[a_n^{(i)} = (1+o(1))  a_{n_{1}}^{(i)} \times \big( n/n_{1} \big)^{1/\ga_i} \, .\]
Using also that $x_i/a_{n_{1}}^{(i)} = t_i +o(1)$ as $n_{1}\to \infty$, and since $g_{\bga}(\cdot)$ is continuous, we get that
\begin{align*}
\sum_{n=\gep n_{1} }^{\gep^{-1} n_{1} } \bP\big( \bS_n =\x \big) &=  \frac{(1+o(1))}{a_{n_{1}}^{(1)} \cdots a_{n_{1}}^{(d)}}  \sum_{n=\gep n_{1} }^{\gep^{-1} n_{1} }  \Big( \frac{n_{1}}{n} \Big)^{\frac{1}{\ga_1}+\cdots +\frac{1}{\ga_d}}   g_{\bga}\Big( t_1 \big (\frac{n_{1}}{n} \big)^{1/\ga_1} , \ldots,   t_d  \big (\frac{n_{1}}{n} \big )^{1/\ga_d}\Big)
\\
& = (1+o(1)) \frac{n_{1}}{a_{n_{1}}^{(1)}  \cdots a_{n_{1}}^{(d)} } \int_{\gep}^{\gep^{-1}} u^{-\sum_{i=1}^d \ga_i^{-1}}g_{\bga} \big( t_1 u^{-1/\ga_1},\ldots, t_d u^{-1/\ga_d} \big) \dd u \, .
\end{align*}
By a change of variable, we therefore get that the first term in \eqref{split3-2} is
\begin{align}
\label{cas2:main}
(1+o(1)) \frac{n_{1}}{a_{n_{1}}^{(1)}  \cdots a_{n_{1}}^{(d)} } \int_{\gep}^{\gep^{-1}} v^{-2+\sum_{i=1}^d \ga_i^{-1}} g_{\bga} \big( t_1 v^{1/\ga_1},\ldots, t_d v^{1/\ga_d} \big) \dd v \, .
\end{align}

\subsection{Third part in \eqref{split3-2}}

Using the local limit theorem~\eqref{LLT}, we have that there is a constant $C$ such that 
\begin{align*}
\sum_{n > \gep^{-1} n_{1}} \bP(\bS_n =\x) \le C \sum_{n > \gep^{-1} n_{1}}  \frac{1}{a_n^{(1)} \cdots a_n^{(d)}} \leq C' \frac{\gep^{-1} n_{1}}{ a_{\gep^{-1}n_{1} }^{(1)} \cdots a_{\gep^{-1}n_{1} }^{(d)}} \, .
\end{align*}
For the last inequality we used that $(a_n^{(1)} \cdots a_n^{(d)})^{-1}$ is regularly varying with exponent $-\sum_{i=1}^d \ga_i^{-1} <-1$. Using again the regular variation of $a_n^{(i)}$, we  get that there is a constant $c$ such that the second term in \eqref{split3-2} is
\begin{equation}
\label{cas2:largen}
\sum_{n > \gep^{-1} n_{1}} \bP(\bS_n =\x) \le c\,  \gep^{\sum_{i=1}^d \ga_i^{-1} -1} \times \frac{ n_{1}}{ a_{n_{1} }^{(1)} \cdots a_{n_{1} }^{(d)}} \, .
\end{equation}

\subsection{First part in \eqref{split3-2}}
Let us first consider the case $\sum_{i=1}^d \ga_i^{-1} <2$.
By Theorem~\ref{thm:locallimit1}, and since $|x_1| \geq a_n^{(1)}$ for $n\le \gep n_1$, we get that
\begin{align*}
\sum_{n=1}^{\gep n_{1}} \bP(\bS_n =\x)  & \leq \sum_{n=1}^{\gep n_{1}} \frac{C}{a_n^{(1)} \cdots a_n^{(d)} } \big(  n \gp_1(|x_1|) |x_1|^{-\gamma_1} +   e^{- c( |x_1|/a_{n}^{(1)})^2} \ind_{\{\ga_1=2\}}\big) \\
&\leq  \frac{C' (\gep n_1)^2}{a_{\gep n_1}^{(1)} \cdots a_{\gep n_1}^{(d)} } \gp_1(|x_1|) |x_1|^{-\gamma_1} +  \frac{C \gep n_1}{a_{ n_1}^{(1)} \cdots a_{ n_1}^{(d)} }  e^{- c' ( |x_1|/a_{\gep n_1}^{(1)})^2} \ind_{\{\ga_1=2\}}\, .
\end{align*}
Here, we used for the first term that $n/(a_n^{(1)} \cdots a_n^{(d)})$ is regularly varying with exponent $1-\sum_{i=1}^d \ga_i^{-1}>-1$. For the second term, we bounded $(a_n^{(1)} \cdots a_n^{(d)})^{-1}$ by $(a_{n_1}^{(1)} \cdots a_{n_1}^{(d)})^{-1} (n_1/n)^{2}$  and also $e^{- c ( |x_1|/a_{n}^{(1)})^2}$ by $(n/n_1)^2 e^{- c' ( |x_1|/a_{n}^{(1)})^2}$ provided that $n_1/n$ is large enough (\textit{i.e.}\ $\gep$ small enough).
Then, we use that $(a_n^{(1)} \cdots a_n^{(d)})^{-1}$  is regularly varying with exponent $-\sum_{i=1}^d \ga_i^{-1}$, that $n_1 \gp_1(|x_1|) |x_1|^{-\gamma_1}$ is bounded above by a constant (thanks to the definition  \eqref{def:an} of $a_n^{(1)}$, together with $|x_1|=a_{n_1}^{(1)}$), and that $|x_1| /a_{\gep n_1}^{(1)} \leq c \gep^{-1/2\ga_1}$ (by Potter's bound): we finally end up with
\begin{equation}
\label{cas2:smalln1}
\sum_{n=1}^{\gep n_{1}} \bP(\bS_n =\x)  \leq \frac{C''  n_1}{a_{n_1}^{(1)} \cdots a_{ n_1}^{(d)} }\times \big(  \gep^{2- \sum_{i=1}^d \ga_i^{-1}}  + \gep e^{- c\gep^{-1/\ga_1}}\big)\, .
\end{equation}

In the case where $\sum_{i=1}^d \ga_i^{-1} \geq 2$, we need to use Assumption~\ref{hyp:2}.  For $n\le \gep n_{1}$ we have that $|x_i| \ge c |t_i| a_{n_{1}}^{(i)} \ge c' t_i \gep^{1/\ga_i} a_{n}^{(i)}$, and we get that for $n_1$ sufficiently large, by Theorem~\ref{thm:locallimit2}
\begin{align*}
\sum_{n=1}^{\gep n_{1}} \bP(\bS_n =\x) & \le \sum_{n=1}^{\gep n_{1}} \frac{c}{|x_1| \cdots |x_d|} \big( n \gp_1(|x_1|) |x_1|^{-\gamma_1}
 + e^{- c (|x_1|/a_n^{(1)})^2}  \ind_{\{\ga_1=2\}} \big) \\
 &\le  \frac{ c'}{ a_{n_{1} }^{(1)} \cdots a_{n_{1} }^{(d)}}  \big( (\gep n_{1})^2 \gp_1(|x_1|) |x_1|^{-\gamma_1} + \gep n_1 e^{-c' \gep^{-1/\ga_1} } \ind_{\{\ga_1=2\}}\big)\, .
\end{align*}
For the second inequality, we used that $|x_i| \ge c a_{n_{1}}^{(i)}$ for all $i$ (since $|t_i|> 0$), and that $|x_1|/a_n^{(1)} \geq c |x_1|/a_{\gep n_1}^{(1)} \geq c' \gep^{1/2\ga_1}$ for all $n\leq \gep n_1$ (thanks to Potter's bound). 
Then, since that $n_{1} \gp_1(|x_1|) |x_1|^{-\gamma_1}$ is bounded by a constant, we get that 
\begin{equation}
\label{cas2:smalln}
\sum_{n=1}^{\gep n_{1}} \bP(\bS_n =\x)  
\le  \frac{ c'  n_{1}}{ a_{n_{1} }^{(1)} \cdots a_{n_{1} }^{(d)}} \big( \gep^2 + \gep e^{-c \gep^{-1/\ga_1}}  \big) \, .
\end{equation}

\subsection*{Conclusion}
Collecting~\eqref{cas2:main} together with \eqref{cas2:largen} and \eqref{cas2:smalln1}-\eqref{cas2:smalln}, then letting $n\to+\infty$ and finally $\gep\downarrow 0$, we obtain~\eqref{casIIfavorite}.

\subsection{Proof in the marginal case $d=2$, $\bga=(2,2)$}
\label{sec:marginald=2}

Again, we work along the favorite direction (see \eqref{def:favdir2}), so that in particular we have $n_1 \sim \lambda n_2$ for some constant $\lambda >0$ (recall the definition~\eqref{def:ni} of $n_i$).
For $\gep>0$ fixed, we split the Green function as
\begin{equation}
\label{marginald=2}
G(\x) = \sum_{n=1}^{\gep^{-1} n_1} \bP(\bS_n = \x) + \sum_{n= \gep^{-1} n_1 +1}^{+\infty} \bP(\bS_n =\x) \, . 
\end{equation}

The main contribution comes from the second sum.  Thanks to the local limit theorem~\eqref{LLT}, we get that for $n$ sufficiently large
\begin{align*}
\bP(\bS_n=\x) = \frac{1}{a_n^{(1)} a_n^{(2)}}  \Big(  g_{\bga} \Big( \frac{ x_1}{a_n^{(1)}} ,\frac{x_2}{ a_n^{(2)}} \Big) +o(1) \Big) \, .
\end{align*}
For $n > \gep^{-1} n_1$ we have that $ |x_1| / a_n^{(1)} \le 2 \gep^{1/2}$ for $|x_1|$ large enough (thanks to the definition of $n_1$ and the fact that $a_n^{(1)}$ is regularly varying with exponent $1/2$), and also we have $n \geq \tfrac12 \lambda \gep^{-1} n_2$ so  that $|x_2|/a_n^{(2)} \leq 2 \lambda^{-1/2} \gep^{1/2}$. All together, and since $g_{\bga}$ is continuous at $0$, for every $\eta>0$ we can choose $\gep$ small enough so that for large enough $|x_1|\asymp |x_2|$
\[\frac{g_{\bga}(0,0)-\eta}{a_n^{(1)} a_n^{(2)}}  \leq \bP(\bS_n =\x) \leq  \frac{g_{\bga}(0,0)+\eta}{a_n^{(1)} a_n^{(2)}} \, .\]
Then, since $v \mapsto \sum_{n\geq v} (a_n^{(1)} a_n^{(2)})^{-1}$ is slowly varying, we get that for $n_1$ large enough
\begin{equation}
\sum_{n= \gep^{-1} n_1 +1}^{+\infty} \bP(\bS_n =\x) \le (g_{\bga}(0,0)+2\eta) \sum_{n \geq n_1} \frac{1}{a_n^{(1)} a_n^{(2)}} \, ,
\end{equation}
and a similar lower bound holds, with $2\eta$ replaced by $-2\eta$.

We now treat the first sum in \eqref{marginald=2}. 
First, for $n\leq \gep n_1$,  Theorem~\ref{thm:locallimit1} gives  that
\[\bP(\bS_n = \x) \leq  \frac{C}{a_n^{(1)} a_n^{(2)}} \times \big( n \gp_1(x_1) x_1^{-\gamma} + e^{- c (x_1/a_n^{(1)})^2} \big) \, . \]
Exactly as what is done above to obtain \eqref{cas2:smalln}, we obtain that
\begin{align*}
\sum_{n=1}^{\gep n_{1}} \bP(\bS_n =\x)  
&\le  \frac{ c'  }{ a_{n_{1} }^{(1)}  a_{n_{1} }^{(2)} } \Big( (\gep n_{1})^2 \gp_1(x_1)^{-\gamma} + \gep n_1 e^{-c \gep^{-1} (x_1/a_{n_1}^{(1)})^2 } \Big) \le \frac{ c_{\gep}  n_{1}}{ a_{n_{1} }^{(1)}  a_{n_{1} }^{(2)}}  \, .
\end{align*}
For $  \gep n_1\le n\le \gep^{-1} n_1 $, the local limit theorem~\eqref{LLT} gives that there is a constant $C>0$ such that 
\begin{equation}
\sum_{n=\gep n_1}^{\gep^{-1} n_1} \bP(\bS_n =\x)   \leq \sum_{n=\gep n_1}^{\gep^{-1} n_1} \frac{C}{a_{n }^{(1)}  a_{n }^{(2)} } \le \frac{C_{\gep} n_1}{a_{n_{1} }^{(1)}  a_{n_{1} }^{(2)} } \, .
\end{equation}
As a conclusion, we get that there is some constant $C'_{\gep}$ such that
\begin{equation}
\sum_{n=1}^{\gep^{-1}n_1} \bP(\bS_n = \x) \leq C'_{\gep} \frac{n_1}{a_{n_1 }^{(1)}  a_{n_1 }^{(2)} } = o\Big(  \sum_{n \geq n_1} \frac{1}{a_n^{(1)} a_n^{(2)}} \Big)\, .
\end{equation}
The last identity comes from \cite[Prop. 1.5.9.a.]{cf:BGT}, since $(a_n^{(1)} a_n^{(2)})^{-1}$ is regularly varying with exponent $-1$, and summable.
This concludes the proof of Theorem~\ref{thm:marginald=2}.

\subsubsection*{About the balanced case}
Let us write $a_n \equiv a_n^{(i)}$, and $L(\cdot), \sigma(\cdot)$ in place of $L_i(\cdot), \sigma_i(\cdot)$.
The walk $\bS_n$ is transient if and only if $\sum_{n=1}^{+\infty} (a_n)^{-2} <+\infty$. We may compare the sum to the integral $\int_{1}^{+\infty} (a_t)^{-2} \dd t$ which by a change of variable $u = a_t$ (by definition of $a_t$ we have $t \sim  (a_t)^2/\sigma(a_t)$), $\dd t \sim 2 u \dd u /\sigma(u) $: we get that  $\int_{1}^{+\infty} (a_t)^{-2} \dd t < +\infty$ if and only if $\int_{1}^{+\infty} \frac{ 2  \dd u }{u \sigma(u)} <+\infty$.
With the same change of variable, and using that $a_{n_1}\sim |x_1|$, we get that
\[\sum_{n\geq n_1}^{+\infty} \frac{1}{(a_n)^2} \sim \int_{n_1}^{+\infty} \frac{\dd t}{ (a_t)^2} \sim \int_{|x_1|}^{+\infty} \frac{ 2 \dd u }{u \sigma(u)}  \, , \]
which gives the announced result.

\section{Proof for case II (\emph{non-zero mean}):  $\mu_i \neq 0$ for some $i$ with $\ga_i>1$}
\label{sec:casII}

In this section, we prove Theorem~\ref{thm:ren1}.
Recall our notations:  we have $\mu_i=0$ ($b_n^{(i)} \equiv 0$) for $i>i_0$, and  $\mu_{i_0}\neq 0$ with $\ga_{i_0}>1$. We set $n_{i_0}:= x_{i_0}/\mu_{i_0}$ ($x_{i_0}$ and $\mu_{i_0}$ need to have the same sign), and we also denote $m_{i_0} := a_{n_{i_0}}^{(i_0)}$, so that the typical number of steps for the $i_0$-th coordinate to visit $x_{i_0}$ is $n_{i_0} + O(m_{i_0})$. For simplicity, we work with $\mu_{i_0}, x_{i_0}>0$.
We consider the case where  $\|\x\| \to +\infty$ along the favorite direction, recall \eqref{def:favdir}. 

We fix $\gep>0$, and decompose $G(\x)$ into three subparts:
\begin{equation}
\label{split3}
G(\x)  =  \bigg( \sum_{n<n_{i_0} - \gep^{-1} m_{i_0}} +   \sum_{n=n_{i_0} - \gep^{-1} m_{i_0} }^{n_{i_0} + \gep^{-1} m_{i_0}} + \sum_{n>n_{i_0}+ \gep^{-1} m_{i_0}}  \bigg) \bP\big( \bS_n =\x \big) \, .
\end{equation}
The main contribution is the second part, that we treat first, before we show that the two other parts are negligible.

\subsection{Main contribution}
\label{sec:casIImain}

Since the summation index ranges from $ n_{i_0} - \gep^{-1} m_{i_0}$ to $n_{i_0} + \gep^{-1} m_{i_0}$ and because we work in the favorite direction, we obtain that $\bP(\bS_n =\x)$ falls into the range of application of the local limit theorem \eqref{LLT}.
We have that
\begin{align*}
\sum_{n=n_{i_0} - \gep^{-1} m_{i_0}}^{n_{i_0} + \gep^{-1} m_{i_0}} \bP\big( \bS_n =\x \big) 
 = \sum_{n=n_{i_0} - \gep^{-1} m_{i_0}}^{n_{i_0} + \gep^{-1} m_{i_0}} \frac{1}{a_{n}^{(1)} \cdots a_n^{(d)} }  \, g_{\bga}&\Big(  \frac{x_1 - b_n^{(1)}}{a_n^{(1)}} , \ldots, \frac{x_d -b_n^{(d)}}{a_n^{(d)}} \Big)
\\
& +\sum_{n=n_{i_0} - \gep^{-1} m_{i_0}}^{n_{i_0} +\gep^{-1} m_{i_0}}  \frac{o(1) }{a_{n}^{(1)} \cdots a_n^{(d)} } \, .
\end{align*}
The second term is negligible compared to $a_{n_{i_0}}^{(i_0)}/(a_{n_{i_0}}^{(1)} \cdots a_{n_{i_0}}^{(d)})$, so we focus on the first term.
Let us consider the different terms $(x_i- b_n^{(i)} )/ a_n^{(i)}$ for the range considered. First of all, notice that $n= (1+o(1)) n_{i_0}$, since $m_{i_0}=o(n_{i_0})$: it gives in particular that $a_{n}^{(i)} = (1+o(1)) a_{n_{i_0}}^{(i)}$.

\smallskip
$\ast$ If $\ga_i>1$, then $b_n^{(i)} = \mu_i  n = b_{n_{i_0}}^{(i)} + \mu_i  (n- n_{i_0})$, and hence
\begin{equation}
\label{lltcas1}
\frac{x_i -b_n^{(i)}}{a_n^{(i)}} = (1+o(1)) \frac{x_i -b_{n_{i_0}}^{(i)} }{a_{n_{i_0}}^{(i)}} + (1+o(1)) \mu_i \frac{n-n_{i_0}}{a_{n_{i_0}}^{(i)}} = t_i + \mu_i \frac{n-n_{i_0}}{a_{n_{i_0}}^{(i)}} +o(1)
\end{equation}
uniformly for $|n- n_{i_0}| \leq \gep^{-1} m_{i_0}$. We used that $(x_i -b_{n_{i_0}}^{(i)})/a_{n_{i_0}}^{(i)} \to t_i \in \bbR$, cf.~\eqref{def:favdir}.

\smallskip
$\ast$ If $b_n^{(i)} \equiv 0$ (in particular if $i>i_0$), then more directly, using~\eqref{def:favdir}
\begin{equation}
\label{lltcas2}
\frac{x_i -b_n^{(i)}}{a_n^{(i)}} = (1+o(1)) \frac{x_i}{a_{n_{i_0}}^{(i)}} = t_i +o(1) \, .
\end{equation}

\smallskip
$\ast$ The last case we need to consider is when $\ga_i=1$. Then $b_n^{(i)} = n \mu_i(a_n^{(i)}) $ and we have
\[ \frac{1}{a_{n_{i_0}}^{(i)}} \big| b_n^{(i)} - b_{n_{i_0}}^{(i)} \big| \le   \frac{ |n-n_{i_0}| }{a_{n_{i_0}}^{(i)}}\,  |\mu_i(a_{n_{i_0}}^{(i)})| + \frac{n}{a_{n_{i_0}}^{(i)}} |\mu_i(a_n^{(i)}) - \mu_i(a_{n_{i_0}}^{(i)})| \, .\]
The first term goes to $0$ as $n_{i_0}\to+\infty$, uniformly for $|n-n_{i_0}| \le \gep^{-1}m_{i_0}$: indeed, $m_{i_0}$ is regularly varying in $n_{i_0}$ with exponent $1/\ga_{i_0} <1$, in contrast with $a_{n_{i_0}}^{(i)}$ which is regularly varying with exponent $1$ (and $|\mu_i(a_{n_{i_0}}^{(i)})|$ is  a slowly varying function). 
For the second term, we have $n\sim n_{i_0}$, and $n_{i_0}/a_{n_{i_0}}^{(i)} \sim  L_i(a_{n_{i_0}}^{(i)})^{-1}$: we can use Claim~\ref{claim:mu} to get that the ratio $|\mu_i(a_n^{(i)}) - \mu_i(a_{n_{i_0}}^{(i)}) | /  L_i(a_{n_{i_0}}^{(i)})$ goes to $0$, since $a_{n}^{(i)}/a_{n_{i_0}}^{(i)}$ goes to $1$.
We therefore obtain that uniformly for $|n-n_{i_0}| \le \gep^{-1}m_{i_0}$, using again \eqref{def:favdir},
\begin{equation}
\label{lltcas3}
\frac{x_i -b_n^{(i)}}{a_n^{(i)}} = \frac{x_i- b_{n_{i_0}}^{(i)} }{ a_{n_{i_0}}^{(i_0)}}  + o(1)  = t_i +o(1)\, .
\end{equation}

Combining all the possibles cases in \eqref{lltcas1}-\eqref{lltcas2}-\eqref{lltcas3}, and recalling Section~\ref{sec:convention} (\textit{i.e.}\ $a_{n_{i_0}}^{(i)} \sim a_{i,i_0} a_{n_{i_0}}^{(i_0)}$), we get that for all $i$, uniformly for $|n-n_{i_0}| \le \gep^{-1}m_{i_0}$,
\begin{equation}
\label{test:favdir1}
\frac{x_i -b_n^{(i)}}{a_n^{(i)}} = t_i +  \kappa_i  \frac{n-n_{i_0}}{m_{i_0}} + o(1) \quad \text{with } \kappa_i = a_{i,i_0} \mu_i \ind_{\{i\geq i_0 \}}
\end{equation}
(recall $m_{i_0}=a_{n_{i_0}}^{(i_0)}$ and that if $\ga_i<\ga_{i_0}$ we have  $a_{i,i_0}=0$).
Because $g_{\bga}$ is continuous, we obtain that
\begin{align}
\sum_{n=n_{i_0} -  \gep^{-1} m_{i_0} }^{n_{i_0} + \gep^{-1} m_{i_0}}  \!\!\!  \bP\big( \bS_n =\x \big) &= \frac{1+o(1)}{a_{n_{i_0}}^{(1)} \cdots a_{n_{i_0}}^{(d)}}  \sum_{n=n_{i_0} - \gep^{-1} m_{i_0}}^{n_{i_0} + \gep^{-1} m_{i_0}} \!\!\!
g_{\bga}\Big( t_1+ \kappa_1 \frac{n-n_{i_0}}{m_{i_0}} , \ldots, t_d+ \kappa_d \frac{n-n_{i_0}}{m_{i_0}} \Big)  \notag \\
& = (1+o(1)) \frac{ a_{n_{i_0}}^{(i_0)}}{a_{n_{i_0}}^{(1)} \cdots a_{n_{i_0}}^{(d)}}  \int_{-\gep^{-1}}^{\gep^{-1}} g_{\bga}\big(  t_1 +\kappa_1 u , \ldots, t_d+ \kappa_d u \big) \dd u\, .
\label{conclusion:main}
\end{align}
The last identity holds thanks to a Riemann sum approximation, as $m_{i_0}=a_{n_{i_0}}^{(i_0)} \to+\infty$.

\subsection{Last part in \eqref{split3}}
\label{sec:casIIlast}

We prove that there are constants $C_{10},C_{11}$ such that 
for every $r \ge C_{10}\, m_{i_0}$
\begin{equation}
\label{lastsplit3}
\sum_{n=n_{i_0} + r}^{+\infty} \bP(\bS_n = \x) \le \frac{C_{11}\, a_{n_{i_0}}^{(i_0)}}{a_{n_{i_0}}^{(1)} \cdots a_{n_{i_0}}^{(d)}}  \Big( \frac{r}{m_{i_0}} n_{i_0} \gp_{i_0}(r) r^{-\gamma_{i_0}}  + e^{- c (r /m_{i_0})^2}\ind_{\{\ga_{i_0}=2\}}\Big) .
\end{equation}
For $r_{\gep}:=\gep^{-1} m_{i_0} =\gep^{-1}a_{n_{i_0}}^{(i_0)}$, we have $n_{i_0} \gp_{i_0}(r_{\gep}) r_{\gep}^{-\gamma_{i_0}} \le c \gep^{\gamma_{i_0}}$ thanks to the definition \eqref{def:an} of $a_{n_{i_0}}^{(i_0)}$: we therefore get that
\begin{equation}
\label{conclusionlargen1}
\sum_{n=n_{i_0}+ \gep^{-1} m_{i_0} }^{+\infty}\bP\big(  \bS_n =\x \big) \le \frac{C'_{11}\, a_{n_{i_0}}^{(i_0)}}{a_{n_{i_0}}^{(1)}\cdots a_{n_{i_0}}^{(d)} } \big( \gep^{\gamma_{i_0}-1} + e^{- c\gep^{-2}}\big)\, .
\end{equation}

Let us now prove~\eqref{lastsplit3}.
For any $n\ge n_{i_0} + r$ with $r\ge C_{10} m_{i_0}$ and $C_{10}$ large enough, we have that $x_{i_0}-b_{n}^{(i_0)} = (n_{i_0}-n)\mu_{i_0} \le  - a_n^{(i_0)}$: we can use Theorem~\ref{thm:locallimit1} to get that 
\begin{align*}
\bP\big(  \bS_n =\x \big) \le \frac{C}{a_{n}^{(1)} \cdots a_n^{(d)}} \Big( n \gp_{i_0}(  n-n_{i_0} ) (n-n_{i_0})^{-\gamma_{i_0}} +  e^{ - c (n-n_{i_0})^2 /(a_n^{(i_0)} )^2 } \ind_{\{\alpha_{i_0}=2\}}   \Big) \, .
\end{align*}
(We used that $a_{n}^{(i_0)}\leq c' m_{i_0}$ for $n\geq n_{i_0}$.)
We therefore get
\begin{align}
\sum_{n=n_{i_0}+r}^{+\infty}\bP\big(  \bS_n =\x \big) \le &  \sum_{n=n_{i_0}+r }^{+\infty} \frac{C n }{a_{n}^{(1)} \cdots a_{n}^{(d)}}  \gp_{i_0}(n-n_{i_0}) (n-n_{i_0})^{-\gamma_{i_0}}
\notag\\
&\qquad  \qquad
  +\sum_{n=n_{i_0}+r }^{+\infty} \frac{C}{a_{n}^{(1)} \cdots a_{n}^{(d)}} e^{ - c (n-n_{i_0})^2/(a_n^{(i_0)} )^2 } \ind_{\{\alpha_{i_0}=2\}}\, .
\label{largen1}
\end{align}
Let us deal with the first sum. If $r\le n_{i_0}$, it is bounded above by a constant times
\begin{align}
\sum_{n=n_{i_0}+r }^{2 n_{i_0}}   & \frac{n}{a_{n}^{(1)} \cdots a_n^{(d)} } \gp_{i_0}(  n-n_{i_0} ) (n-n_{i_0})^{-\gamma_{i_0}}  + \sum_{n > 2n_{i_0}}  \frac{n}{a_{n}^{(1)} \cdots a_n^{(d)}} \gp_{i_0}(  n) n^{-\gamma_{i_0}} \notag\\
&\le  \frac{ c n_{i_0}}{a_{n_{i_0}}^{(1)} \cdots a_{n_{i_0}}^{(d)}}  \gp_{i_0}(r) r^{1-\gamma_{i_0}}
 +  c \gp_{i_0}(n_{i_0}) \frac{n_{i_0}^{2-\gamma_{i_0}}}{a_{n_{i_0}}^{(1)} \cdots a_{n_{i_0}}^{(d)} } \le  \frac{ c' n_{i_0}}{a_{n_{i_0}}^{(1)} \cdots a_{n_{i_0}}^{(d)}}  \gp_{i_0}(r) r^{1-\gamma_{i_0}}\, .
 \label{largen1-bis}
\end{align}
For the first sum we used that $\gamma_{i_0}>1$, and for the second sum that the sequence under summation is regularly varying with index $1 - \gamma_{i_0}- \sum_{i=1}^d  \ga_{i_0}^{-1}<-1$.
In the case $r\ge n_{i_0}$, the first term in \eqref{largen1} is bounded by a constant times
\[  c \gp_{i_0}(r) \frac{r^{2-\gamma_{i_0}}}{a_{r}^{(1)} \cdots a_{r}^{(d)} }
\le \frac{c' n_{i_0}}{ a_{n_{i_0}}^{(1)} \cdots a_{n_{i_0}}^{(d)}} \gp_{i_0}(r) r^{1-\gamma_{i_0}} \, .
\]

For the second sum in \eqref{largen1} ($\ga_{i_0}=2$),  we get in the case $r\leq n_{i_0}$ that it is bounded by
\begin{align*}
&\frac{C}{a_{n_{i_0}}^{(1)} \cdots a_{n_{i_0}}^{(d)} }  \Big( \sum_{n=n_{i_0}+r }^{2n_{i_0}}  e^{ - c ((n-n_{i_0}) / m_{i_0})^2   }
 + \sum_{n> 2n_{i_0}} e^{- c (n /a_{n}^{(i_0)})^2 } \Big) \\ 
& \qquad \qquad \le   \frac{C }{a_{n_{i_0}}^{(1)} \cdots a_{n_{i_0}}^{(d)} }  \Big( m_{i_0} e^{ - c (r / m_{i_0})^2 } + m_{i_0} e^{- c' (n_{i_0} /m_{i_0})^2   }  \Big) \le  \frac{C a_{n_{i_0}}^{(i_0)}}{a_{n_{i_0}}^{(1)} \cdots a_{n_{i_0}}^{(d)} } e^{ - c (r / m_{i_0})^2 } \, .
\end{align*}
For $r\ge n_{i_0}$ the same bound holds with $a_{n_{i_0}}^{(i_0)} e^{ - c (r / m_{i_0})^2 }$ replaced by $a_r^{(i_0)} e^{ - c (r / a_r^{(i_0)})^2 }$, which is bounded by $a_{n_{i_0}}^{(i_0)} e^{ - c (r / a_r^{(i_0)})^2 }$. This term is therefore negligible compared to~\eqref{largen1-bis} (in the case $r\geq n_{i_0}$).  This concludes the proof of \eqref{lastsplit3}.



\subsection{First part in \eqref{split3}}
We again split the sum into two parts:
\begin{equation}
\label{smalln}
\sum_{n=1}^{n_{i_0} -  \gep^{-1} m_{i_0} } \bP(\bS_n =\x) = \sum_{n=1}^{n_{i_0}/2}  \bP(\bS_n =\x) +  \sum_{n=n_{i_0}/2+1}^{\gep^{-1} m_{i_0} }\bP(\bS_n =\x)\, .
\end{equation}

The second part in \eqref{smalln} can be treated in the same manner as for \eqref{lastsplit3}: we have for any $C_{10} m_{i_0} \le  r\le n_{i_0}/2$ 
\begin{equation}
\label{firstsplit3}
 \sum_{n=n_{i_0}/2+1}^{n_{i_0}  - r }\bP(\bS_n =\x) \le \frac{C a_{n_{i_0}}^{(i_0)}}{a_{n_{i_0}}^{(1)} \cdots a_{n_{i_0}}^{(d)}}  \Big( \frac{r}{m_{i_0}}  n_{i_0}\gp_{i_0}(r) r^{-\gamma_{i_0}} + e^{- c (r/m_{i_0})^2}\ind_{\{\ga_{i_0}=2\}}\Big)\, .
\end{equation}
Indeed, this comes from the same argument as for~\eqref{largen1}-\eqref{largen1-bis}---we are able to use Theorem~\ref{thm:locallimit1} since $x_{i_0}- b_n^{(i_0)} \geq a_n^{(i_0)}$ for $n\leq n_{i_0}- C_{10} m_{i_0}$.
Then, using~\eqref{firstsplit3} with $r_{\gep}= \gep^{-1}m_{i_0}$, we get as for~\eqref{conclusionlargen1} that
\begin{equation}
\label{conclusionsmalln1}
\sum_{n=n_{i_0}/2+1}^{n_{i_0}  - \gep^{-1} m_{i_0}}\bP(\bS_n =\x) 
 \le \frac{C a_{n_{i_0}}^{(i_0)}}{ a_{n_{i_0}}^{(1)} \cdots a_{n_{i_0}}^{(d)}} \big( \gep^{\gamma_1-1} +e^{-c\gep^{-2}}\big)\, .
\end{equation}

It remains to control the first term in~\eqref{smalln}, and this is where we use one of our assumptions in Theorem~\ref{thm:ren1}.

\paragraph{\it (i) If $\sum_{i=1}^{d} \ga_i^{-1} <2$ or (ii) $\gamma_{i_0}>\sum_{i\neq i_0} \ga_i^{-1}$}
We invoke Theorem~\ref{thm:locallimit1}: there is a constant $C$ such that uniformly for $n\le n_{i_0}/2$ (so $x_{i_0}-b_n^{(i_0)} \ge  c n_{i_0} \ge a_{n}^{(i_0)}$) we have
\begin{equation}
\label{smallnlocal}
 \bP\big( \bS_n =\x \big) \le \frac{C}{ a_n^{(1)} \cdots a_n^{(d)}} \Big( n \gp_{i_0}( n_{i_0})  n_{i_0}^{-\gamma_{i_0}} + e^{ - c (n_{i_0}/a_n^{(i_0)})^2 } \ind_{\{\alpha_1=2\}}  \Big)\, .
\end{equation}
First of all, bounding below $a_n^{(i)}$ by a constant, we have
\begin{equation}
\label{termexpo}
\sum_{n=1}^{n_{i_0}/2}  \frac{1}{ a_n^{(1)} \cdots a_n^{(d)}} e^{ - c (n_{i_0}/a_{n}^{(i_0)})^2 } \ind_{\{\alpha_1=2\}} \le C n_{i_0} e^{- c (n_{i_0} /m_{i_0})^2}  = o(1) \frac{a_{n_{i_0}}^{(i_0)}}{ a_{n_{i_0}}^{(1)} \cdots a_{n_{i_0}}^{(d)} }\, ,
\end{equation}
the last identity being valid since $e^{- c (n_{i_0} / m_{i_0})^2}$ decays faster than any power of $n_{i_0}$.

$\ast$ If $\sum_{i=1}^{d} \ga_i^{-1} <2$ we get that $n/(a_n^{(1)} \cdots a_n^{(d)})$ is regularly varying with exponent larger than $-1$, so that
\begin{equation}
\label{termnonexpo1}
\sum_{n=1}^{n_{i_0}/2}  \frac{C}{ a_n^{(1)} \cdots a_n^{(d)}}n \gp_{i_0}( n_{i_0})  n_{i_0}^{-\gamma_{i_0}} \le  \frac{ C n_{i_0}^{2-\gamma_{i_0}} \gp_{i_0}(n_{i_0}) }{  a_{n_{i_0}}^{(1)} \cdots a_{n_{i_0}}^{(d)}}  =  o(1) \frac{a_{n_{i_0}}^{(i_0)}}{ a_{n_{i_0}}^{(1)} \cdots a_{n_{i_0}}^{(d)} } \, .
\end{equation}
To obtain the $o(1)$, we used that $\gp_{i_0}( n_{i_0})  n_{i_0}^{2-\gamma_{i_0}} / a_{n_{i_0}}^{(i_0)}$ is regularly varying with exponent $2-\gamma_{i_0}-1/\ga_{i_0} <0$ (recall that $\gamma_{i_0}\geq \ga_{i_0} >1$).

$\ast$ In the case where $\sum_{i=1}^{d} \ga_i^{-1} \ge 2$ with $\gamma_{i_0} > \sum_{i\neq i_0} \ga_i^{-1}$, then 
\begin{equation}
\label{termnonexpo2}
\sum_{n=1}^{n_{i_0}/2}  \frac{n}{ a_n^{(1)} \cdots a_n^{(d)}}n \gp_{i_0}( n_{i_0})  n_{i_0}^{-\gamma_{i_0}} \le \tilde \gp( n_{i_0}) n_{i_0}^{-\gamma_{i_0}} = o(1) \frac{a_{n_{i_0}}^{(i_0)}}{ a_{n_{i_0}}^{(1)} \cdots a_{n_{i_0}}^{(d)} } 
\end{equation}
since the sum grows like a slowly varying function $\tilde \gp$ (or remains bounded).  The  $o(1)$ comes from the fact that  $a_{n_{i_0}}^{(i_0)} / (a_{n_{i_0}}^{(1)} \cdots a_{n_{i_0}}^{(d)})$  is regularly varying with exponent larger than~$-\gamma_{i_0}$.

%

\paragraph{\it (iii) If Assumption~\ref{hyp:2} holds}
For $n\le n_{i_0}/2$, we have $x_{i_0} -b_n^{(i_0)} \ge c n_{i_0} \ge a_{n}^{(i_0)}$:  Theorem~\ref{thm:locallimit2} gives that 
\[
\bP\big( \bS_n =\x \big) \le \frac{c}{ \prod_{i=1}^d\big(  |x_i-b_n^{(i)}|\vee a_{n}^{(i)} \big)} \Big(   n \gp(n_{i_0}) n_{i_0}^{-\gamma_{i_0}} + e^{- c n_{i_0} / a_{n}^{(i_0)}} \ind_{\{\ga_{i_0} =2\}}\Big)\, .
\]
We notice that if $b_n^{(i)}\equiv 0$, then since we assumed that $t_i \neq 0$, we have that $|x_i| \ge c a_{n_{i_0}}^{(i)}$ for all $n\le n_{i_0}/2$, provided $n_{i_0}$ is large. Otherwise, we write $|x_i- b_n^{(i)}| \ge |b_{n_{i_0}}^{(i)} - b_n^{(i)}| - |x_i -b_{n_{i_0}}^{(i)}|$. Since we work along the favorite direction we have $|x_i -b_{n_{i_0}}^{(i)}| \leq c a_{n_{i_0}}^{(i)}$. 
Also, using that $b_n^{(i)}$ is regularly varying with exponent $1$ we have that $|b_{n_{i_0}}^{(i)} - b_n^{(i)}| \ge c |b_{n_{i_0}}^{(i)}|$ uniformly for $n \le n_{i_0}/2$. Therefore, using that $|b_n^{(i)}|/a_n^{(i)} \to +\infty$, for $n\le n_{i_0}/2$ we have $|x_i-b_n^{(i)}| \ge c' b_{n_{i_0}}^{(i_0)} \geq a_{n_{i_0}}^{(i_0)}$ provided that $n_{i_0}$ is large enough.
All together, recalling also that $|x_{i_0} - b_{n_{i_0}}^{(i_0)} | \ge c n_{i_0}$, we get that
\[
\bP\big( \bS_n =\x \big) \le \frac{c' a_{n_{i_0}}^{(i_0)}}{ a_{n_{i_0}}^{(1)} \cdots a_{n_{i_0}}^{(d)}}    n \gp(n_{i_0}) n_{i_0}^{-(1+\gamma_{i_0})} +  \frac{c'}{ a_{n_{i_0}}^{(1)} \cdots a_{n_{i_0}}^{(d)}}e^{- c (n_{i_0} / m_{i_0})^2} \ind_{\{\ga_{i_0} =2\}}\, .
\]
Summing the second term over $n\leq n_{i_0}/2$, \eqref{termexpo} already gives that it is negligible compared to \eqref{conclusion:main}.
For the other term, we  have  
\begin{align}
\sum_{n=1}^{n_{i_0}/2} \frac{a_{n_{i_0}}^{(i_0)}}{ a_{n_{i_0}}^{(1)} \cdots a_{n_{i_0}}^{(d)}}    n \gp(n_{i_0}) n_{i_0}^{-(1+\gamma_{i_0})} \le \frac{C a_{n_{i_0}}^{(i_0)}}{ a_{n_{i_0}}^{(1)} \cdots a_{n_{i_0}}^{(d)}}  n_{i_0}^{1-\gamma_{i_0}} \gp(n_{i_0}) = o(1) \frac{a_{n_{i_0}}^{(i_0)}}{ a_{n_{i_0}}^{(1)} \cdots a_{n_{i_0}}^{(d)} } \, ,
\label{conclusionsmalln0}
\end{align}
the last identity holding since $\gamma_{i_0}>1$.

As a conclusion, we obtain in all cases (i)-(ii)-(iii) that
\begin{equation}
\label{conclusionsmalln2}
\sum_{n=1}^{n_{i_0}/2} \bP(\bS_n =\x) = o(1) \frac{a_{n_{i_0}}^{(i_0)}}{ a_{n_{i_0}}^{(1)} \cdots a_{n_{i_0}}^{(d)} } \, .
\end{equation}

\subsection*{Conclusion}

Combining \eqref{conclusion:main} with \eqref{conclusionlargen1}, \eqref{conclusionsmalln1} and \eqref{conclusionsmalln2}, then letting $n\to+\infty$ and finally $\gep\downarrow 0$, we get the conclusion~\eqref{casIfavorite}.

\section{Proof for case III: $\ga_{i_0}=1$}
\label{sec:casIII}

In this section, we prove Theorem~\ref{thm:ga=1}.

\subsection{Preliminaries}
\label{sec:casIII-prelim}
Recall that $b_n^{(i)}\equiv 0$ for all $i < i_0$, and that $\ga_{i_0}=1$ with $b_n^{(i_0)} \not \equiv 0$ (and  $|b_{n}^{(i_0)}|/a_{n}^{(i_0)}\to +\infty$). Recall that $n_{i_0}$ is such that $b_{n_{i_0}}^{(i_0)} =x_{i_0}$ ($b_{n_{i_0}}^{(i_0)}$ and $x_{i_0}$ need to have the same sign). 
First, we stress that $\mu_{i_0}(a_n^{(i_0)}) \sim \mu_{i_0}(|b_n^{(i_0)}|)$ (this is trivial if $\mu_{i_0}\in \bbR^*$, and we refer to Lemma~4.3 in \cite{cf:B17} in the case $|\mu_{i_0}|=+\infty$ or $0$ with $p_{i_0}\neq q_{i_0}$): we get that $\mu_{i_0}(a_{n}^{(i_0)}) \sim \mu_{i_0}(|x_{i_0}|)$, and hence $x_{i_0}=b_{n_{i_0}}^{(i_0)} \sim n_{i_0} \mu_{i_0} (|x_{i_0}|)$. We therefore conclude that $n_{i_0} \sim x_{i_0}/\mu_{i_0}(|x_{i_0}|)$ as $|x_{i_0}|\to+\infty$ (provided that $x_{i_0}$ and $\mu_{i_0}(|x_{i_0}|)$ have the same sign).

We also define
\begin{align}
\label{def:m}
m_{i_0}  := \frac{a_{n_{i_0}}^{(i_0)}}{| \mu_{i_0}(a_{n_{i_0}}^{(i_0)})| }\sim n_{i_0} \frac{L_{i_0}(a_{n_{i_0}}^{(i_0)})}{ |\mu_{i_0}(a_{n_{i_0}}^{(i_0)} )|} = o(n_{i_0})\, .
\end{align}
We used the definition \eqref{def:an} of $a_{n}^{(i_0)}$ for the asymptotic equivalence, and then that $L_{i_0}(x) = o(\mu_{i_0}(x))$ (both if $\mu_{i_0}\in \bbR^*$ or $p_{i_0}=q_{i_0}$, thanks to \cite[Prop.~1.5.9.a]{cf:BGT}). We stress that the typical number of steps for the $i_0$-th coordinate to reach $x_{i_0}$ is $n_{i_0}+O(m_{i_0})$. The intuition is that, when looking for which $n$ we have that $x_{i_0}- b_n^{(i_0)} (= b_{n_{i_0}}^{(i_0)} - b_n^{(i_0)})$ is of order $a_{n}^{(i_0)} \sim  a_{n_{i_0}}^{(i_0)}$,  and using that $|b_n - b_{n_{i_0}}|$ is roughly of the order of $(n-n_{i_0}) |\mu_{i_0}(a_{n_{i_0}}^{(i_0)})|$ (see \eqref{eq:rangen} below for details) we find that $n - n_{i_0}$ has to be of the order of $ a_{n_{i_0}}^{(i_0)} / |\mu_{i_0}(a_{n_{i_0}}^{(i_0)})| =:m_{i_0}$.
This intuition is confirmed in \cite{cf:AA,cf:HR79}, where it is shown that $(N(x_{i_0}) - n_{i_0})/m_{i_0}$ converges in distribution as $|x_{i_0}| \to +\infty$, where $N(x_{i_0}) = \inf\{n, \, S_n^{(i_0)}>x_{i_0}\}$ is the first passage time to $x_{i_0}$ (if $x_{i_0}>0$).

Then we fix $\gep>0$, and again split $G(\x)$ into three parts:
\begin{equation}
\label{split3-3}
G(\x)  =  \bigg( \sum_{n<n_{i_0} - \gep^{-1} m_{i_0} } +   \sum_{n=n_{i_0} - \gep^{-1} m_{i_0}}^{n_{i_0} + \gep^{-1} m_{i_0} } + \sum_{n>n_{i_0}+ \gep^{-1}m_{i_0} } \bigg) \bP\big( \bS_n =\x \big) \, .
\end{equation}
As suggested above, the main contribution is the middle sum.
In the following, we work with $x_{i_0}$ and $\mu_{i_0}(x_{i_0})$ positive, simply to avoid the use of absolute values.

\subsection{Main contribution}
For the middle sum in \eqref{split3-3}, the fact that we work along the favorite scaling tells that we can use the local limit theorem~\eqref{LLT}, and get that
\begin{equation}
\label{maintermga=1}
\sum_{n=n_{i_0} - \gep^{-1}  m_{i_0} }^{n_{i_0} + \gep^{-1}  m_{i_0} }  \bP\big( \bS_n =\x \big)
=  \sum_{n=n_{i_0} - \gep^{-1}  m_{i_0} }^{n_{i_0} + \gep^{-1}  m_{i_0} } 
\frac{1+o(1)}{a_n^{(1)}\cdots a_n^{(d)}} g_{\bga}\Big( \frac{x_{i_0}-b_n^{(1)}}{a_n^{(1)}}, \ldots, \frac{x_d - b_n^{(d)} }{a_{n}^{(d)}} \Big).
\end{equation}
Note that, for the range of $n$ considered, we have that $n = (1+o(1)) n_{i_0}$ (recall~\eqref{def:m}), so that $a_n^{(i)} = (1+o(1)) a_{n_{i_0}}^{(i)}$.
 
$\ast$ if $i>i_0$ or $\ga_i <1$ then  $b_n^{(i)} \equiv 0$: thanks to  our assumption \eqref{def:favdir3}, we get that
\[  \frac{x_i}{a_{n_{i_0}}^{(i)}}= (1+o(1)) t_i \, .\]

 $\ast$ if $\ga_i=1$, then we get that for $|n-n_{i_0} |\leq \gep^{-1} m_{i_0}$ 
\begin{equation}
\label{eq:rangen}
\frac{x_i-b_n^{(i)} }{ a_{n_{i_0}}^{(i)}}  = \frac{x_i-b_{n_{i_0}}^{(i)} }{ a_{n_{i_0}}^{(i)}} +  \frac{(n_{i_0} -n)\mu_i(a_{n_{i_0}}^{(i)})}{a_{n_{i_0}}^{(i)}} +  \frac{n_{i_0}}{a_{n_{i_0}}^{(i)}} \big( \mu_i(a_{n_{i_0}}^{(i)}) - \mu_i(a_n^{(i)}) \big)\,. 
\end{equation}
The first part goes to $t_i\in \bbR$, thanks to~\eqref{def:favdir3}.
The last part goes to $0$ thanks to Claim~\ref{claim:mu} (recall $n_{i_0} \sim a_{n_{i_0}}^{(i)} L_i(a_{n_{i_0}}^{(i)})^{-1}$), since $a_n^{(i)}/a_{n_{i_0}}^{(i)}$ goes to $1$. For the middle part, we use assumption~\eqref{tildea} to get that $\mu_i(a_{n_{i_0}}^{(i)})/a_{n_{i_0}}^{(i)} \sim \tilde a_{i,i_0}  m_{i_0}$. In the end, we obtain that uniformly for $|n-n_{i_0}| \leq \gep^{-1}m_{i_0}$
\begin{equation}
\label{tes:favdir2}
\frac{x_i-b_n^{(i)} }{ a_{n_{i_0}}^{(i)}}  = t_i + \frac{ \tilde a_{i,i_0}}{ m_{i_0}} (n_{i_0} -n) +o(1) \, .
\end{equation}
Using this in the sum \eqref{maintermga=1}, and with the definition $\tilde \kappa_i =\tilde a_{i,i_0} \ind_{\{i\geq i_0\}}$ (with $\tilde a_{i,i_0}=0$ if $\ga_i<1$),  we get thanks to the continuity of $g_{\bga}$ that the right-hand side of~\eqref{maintermga=1} is
\begin{align}
 \frac{1+o(1)}{a_{n_{i_0}}^{(1)} \cdots a_{n_{i_0}}^{(d)} } &
\sum_{n=n_{i_0} - \gep^{-1} m_{i_0} }^{n_{i_0} +  \gep^{-1} m_{i_0} }   g_{\bga} \Big( t_1+ \frac{ \tilde \kappa_1}{m_{i_0}} (n_{i_0} -n), \ldots, t_d+ \frac{ \tilde \kappa_d }{ m_{i_0} } (n_{i_0} -n) \Big)  \notag\\
&   = (1+o(1))\frac{ m_{i_0}}{ a_{n_{i_0}}^{(1)} \cdots a_{n_{i_0}}^{(d)} } \int_{-\gep^{-1}}^{\gep^{-1}} g_{\bga} \big( t_1+ \tilde \kappa_1 u , \ldots, t_d+  \tilde \kappa_d u \big) \dd u\, ,
\label{conclusionmainga=1}
\end{align}
where we used a Riemann-sum approximation to obtain the last integral.

\subsection{Third term in \eqref{split3-3}}

First of all, let us stress that there is a constant $C_{12}$ such that  for $2 n_{i_0} \ge n \ge n_{i_0} + C_{12}\, m_{i_0}$, we have  
\begin{align}
x_{i_0}-b_{n}^{(i_0)} &= (n_{i_0} - n ) \mu_{i_0}(a_{n_{i_0}}^{(i_0)}) + n \big(\mu_{i_0}(a_{n_{i_0}}^{(i_0)}) - \mu_{i_0}(a_n^{(i_0)}) \big) \notag\\
& \leq  c (n_{i_0} - n ) \mu_{i_0}(a_{n_{i_0}}^{(i_0)}) \le -  a_n^{(i_0)}\, .
\label{remarquex-b}
\end{align}
We used that $\mu_{i_0}(a_{n_{i_0}}^{(i_0)}) - \mu_{i_0}(a_n^{(i_0)}) = O(L_{i_0}(a_n^{(i_0)})  )$ for $n_{i_0}\leq n \le 2 n_{i_0}$ (see Claim~\ref{claim:mu}), and the fact that $n_{i_0} L_{i_0}(a_n^{(i_0)}) \leq c a_{n_{i_0}}^{(i_0)}$. For $n\geq 2n_{i_0}$, we have that $x_{i_0}-b_{n}^{(i_0)} \leq - c b_{n_{i_0}}^{(i_0)} \leq - a_{n_{i_0}}^{(i_0)}$: the first inequality comes from the fact that $b_{n}^{(i_0)}$ is regularly varying with exponent~$1$, the second from the fact that $b_{n}^{(i_0)}/a_n^{(i_0)} \to +\infty$.

Therefore, we can use Theorem~\ref{thm:locallimit2} (recall $\ga_{i_0}=1$) to get that for all $n \ge n_{i_0} + C_{12} m_{i_0}$,
\begin{align}
\label{use1:locallimit2}
\bP\big(  \bS_n =\x \big) \le \frac{C a_{n}^{(i_0)}}{a_{n}^{(1)} \cdots a_n^{(d)}}  \, n L_{i_0}\big( |x_{i_0}-b_{n}^{(i_0)}| \big) \big| x_{i_0}-b_{n}^{(i_0)}   \big|^{-2}  \, .
\end{align}
Then, for any $n_{i_0}\geq r \ge C_{12} \, m_{i_0}$, setting $j=n-n_{i_0}$ so $|x_{i_0}-b_{n}^{(i_0)}| \ge c j \mu_{i_0}( a_{n_{i_0}}^{(i_0)})$ (see \eqref{remarquex-b}), 
\begin{align}
\sum_{n=n_{i_0}+r}^{2 n_{i_0}}\bP\big(  \bS_n =\x \big) & \le  \frac{C a_{n_{i_0}}^{(i_0)}}{a_{n_{i_0}}^{(1)} \cdots a_{n_{i_0}}^{(d)}}   \sum_{j=r }^{2 n_{i_0}}  n_{i_0}  L_{i_0}\big( j \mu_{i_0}(a_{n_{i_0}}^{(i_0)}) \big) \big(  j \mu_{i_0}(a_{n_{i_0}}^{(i_0)})\big)^{-2} \notag\\
&\leq \frac{C'  a_{n_{i_0}}^{(i_0)} }{a_{n_{i_0}}^{(1)} \cdots a_{n_{i_0}}^{(d)}}  \times n_{i_0}\, \frac{1}{\mu_{i_0}(a_{n_{i_0}}^{(i_0)})}    L_{i_0} \big( r\mu_{i_0}(a_{n_{i_0}}^{(i_0)}) \big) \big( r \mu_{i_0}(a_{n_{i_0}}^{(i_0)})\big)^{-1} \, .
\label{8.9}
\end{align}
Then, setting $r=\gep^{-1} m_{i_0}$, and using the definition \eqref{def:an} of $a_n^{(i_0)}$, we get that 
\begin{equation}
\label{conclusionalpha1-3}
\sum_{n=n_{i_0}+\gep^{-1 }m_{i_0}}^{2 n_{i_0}} \bP\big(  \bS_n =\x \big) \leq   \frac{C'' \gep  m_{i_0}}{a_{n_{i_0}}^{(1)} \cdots a_{n_{i_0}}^{(d)}}   \, .
\end{equation} 
For the sum with $n\ge 2 n_{i_0}$, we use \eqref{use1:locallimit2} with the fact that $x_{i_0} -b_n^{(i_0)}\leq - cb_n^{(i_0)}$: 
\begin{align}
\label{lastsplitga=1}
\sum_{n=2n_{i_0}}^{+\infty}\bP\big(  \bS_n =\x \big)
& \leq \sum_{n=2 n_{i_0}}^{+\infty} \frac{ C  a_{n}^{(i_0)} }{a_{n}^{(1)} \cdots a_{n}^{(d)}} \times   n L_{i_0}(b_n^{(i_0)}) (b_n^{(i_0)})^{-2} \,  \notag\\
& \leq \frac{ C'  a_{n_{i_0}}^{(i_0)} }{a_{n_{i_0}}^{(1)} \cdots a_{n_{i_0}}^{(d)}} \times   n_{i_0}^2 L_{i_0}(b_{n_{i_0}}^{(i_0)}) (b_{n_{i_0}}^{(i_0)})^{-2} = \frac{o(1) m_{i_0}}{a_{n_{i_0}}^{(1)} \cdots a_{n_{i_0}}^{(d)}}\, .
\end{align}
For the summation, we used  the sequence under summation is regularly varying with exponent smaller than $-1$.
For the last $o(1)$, we used that $b_{n_{i_0}}^{(i_0)} \sim n_{i_0}\mu_{i_0}(a_{n_{i_0}}^{(i_0)}) \sim n_{i_0}\mu_{i_0}(b_{n_{i_0}}^{(i_0)})$, and that $L_{i_0}(b_{n_{i_0}}^{(i_0)})/\mu_{i_0}(b_{n_{i_0}}^{(i_0)}) \to 0$.

\subsection{First term in \eqref{split3-3}}

As in \eqref{remarquex-b}, we have $x_{i_0}- b_n^{(i_0)} \geq a_{n}^{(i_0)}$ for $n_{i_0}/2 \le n\leq n_{i_0} - C_{12}\, m_{i_0}$. Hence Theorem~\ref{thm:locallimit2} gives the same bound as~\eqref{use1:locallimit2}: for any $r \geq C_{12}\, m_{i_0}$, setting $j= n_{i_0}- n $ ($x_{i_0} -b_{n}^{(i_0)} \leq - c j \mu_{i_0}(a_{n_{i_0}}^{(i_0)})$), we have as in \eqref{8.9}
\begin{align}
\sum_{n=n_{i_0}/2}^{n_{i_0}- r}\bP\big(  \bS_n =\x \big) & \le \frac{C a_{n_{i_0}}^{(i_0)} }{a_{n_{i_0}}^{(1)} \cdots a_{n_{i_0}}^{(d)}} \sum_{j=r}^{n_{i_0}/2} n_{i_0} L_{i_0}\big( j \mu_{i_0}(a_{n_{i_0}}^{(i_0)})\big) \big( j \mu_{i_0}(a_{n_{i_0}}^{(i_0)}  \big)^{-2}  \notag\\
&\leq  \frac{C m_{i_0}}{a_{n_{i_0}}^{(1)} \cdots a_{n_{i_0}}^{(d)}} \times  n_{i_0}  L_{i_0} \big( r\mu_{i_0}(a_{n_{i_0}}^{(i_0)}) \big) \big( r \mu_{i_0}(a_{n_{i_0}}^{(i_0)})\big)^{-1} \, .
\label{conclusionalpha1-4}
\end{align}
With $r=\gep^{-1} m_{i_0}$, we obtain the same upper bound as in \eqref{conclusionalpha1-3}.

For the term $n\leq n_{i_0}/2$, we use that $x_{i_0} -b_n^{(i_0)}\geq c b_{n_{i_0}}^{(i_0)}$ provided that $n_{i_0}$ is large enough. 
Since we work along the favorite direction \eqref{def:favdir3}, we have that $|x_i -b_n^{(i)} | \ge c a_{n_{i_0}}^{(i)}$ for all $n\leq n_{i_0}/2$ (recall we assumed $t_i \neq 0$ if $b_n^{(i)}\equiv 0$):
Theorem~\ref{thm:locallimit2} gives that 
\[\bP\big(  \bS_n =\x \big) \leq  \frac{C a_{n_{i_0}}^{(i_0)}}{a_{n_{i_0}}^{(1)} \cdots a_{n_{i_0}}^{(d)}} \,  n   L_{i_0}(b_{n_{i_0}}^{(i_0)}) (b_{n_{i_0}}^{(i_0)})^{-2} \, . \]
Hence, similarly to~\eqref{lastsplitga=1}, we get that
\begin{align}
\sum_{n=1}^{n_{i_0}/2} \bP\big(  \bS_n =\x \big)
& \leq \frac{C' a_{n_{i_0}}^{(i_0)} }{a_{n_{i_0}}^{(1)} \cdots a_{n_{i_0}}^{(d)}} \times n_{i_0}^2 L_{i_0}(b_{n_{i_0}}^{(i_0)}) (b_{n_{i_0}}^{(i_0)})^{-1}  = \frac{o(1) m_{i_0}}{a_{n_{i_0}}^{(1)} \cdots a_{n_{i_0}}^{(d)}}   \, .
\label{conclusionalpha1-5}
\end{align}

\subsection*{Conclusion}
As for the previous sections, collecting \eqref{conclusionmainga=1} together with \eqref{conclusionalpha1-3}-\eqref{lastsplitga=1} and \eqref{conclusionalpha1-4}-\eqref{conclusionalpha1-5}, then letting $n\to+\infty$ and finally $\gep\downarrow 0$, we get the conclusion~\eqref{casIIIfavorite}.

\section{Proofs when $\mathbf{x}$ is away from the favorite direction or scaling}
\label{sec:away}

In this section, we prove the renewal estimates when away from the favorite direction or scaling, \textit{i.e.}\ we prove Theorems~\ref{thm:awayI} (in Section~\ref{sec:thmawayI}), \ref{thm:casII-IIIb} (in Section~\ref{sec:casII-IIIawayb}), and \ref{thm:casII-IIIa} (in Section~\ref{sec:casII-IIIawaya}).
Again, let us work with all $x_i$'s positive in this section, to avoid the use of  absolute values.
Recall also that we work in dimension $d=2$ with $\bga\neq (2,2)$ and under Assumption~\ref{hyp:2}.

\subsection{Case I, proof of Theorem~\ref{thm:awayI}}
\label{sec:thmawayI}

Recall that $n_i$ is defined up to asymptotic equivalence by $a_{n_i}^{(i)} \sim x_i$, and $i_0, i_1$ are such that $n_{i_0} = \min\{ n_1, n_2\}$, $n_{i_1} = \max\{n_1,n_2\}$. In such a way, we have that $x_i/a_{n_{i_0}}^{(i)} \ge c$ for $i=1,2$.
Let us work in the case where $x_{i_1}/a_{n_{i_0}}^{(i_1)} \ge C_{13}$ for some large constant $C_{13}$ (otherwise one falls in the favorite scaling~\eqref{def:favdir2}): it is equivalent to having $n_{i_1} /n_{i_0}$  larger than some large constant $C'_{13}$.
We let $n_{i_0} \le m \le n_{i_1}$ (we optimize its value below), and decompose $G(\x)$ into two parts
\begin{equation}
\label{awaysplit2}
G(\x) = \sum_{n=1}^m \bP(\bS_n = \x) + \sum_{n=m+1}^{\infty} \bP(\bS_n =\x)\, .
\end{equation}

For the first part, since $x_{i_1}\geq c a_{n}^{(i_1)}$ for $n\leq m \le n_{i_1}$, Theorem~\ref{thm:locallimit2} gives us that
\begin{align*}
\bP(\bS_n =\x) &\le \frac{C}{  (x_1 \vee a_n^{(1)})  (x_2 \vee a_{n}^{(2)})}  \Big( n \gp_{i_1}(x_{i_1}) x_{i_1}^{-\gamma_{i_1}} + e^{- c (x_{i_1} /a_n^{(i_{1})})^2} \ind_{\{\ga_{i_1}=2\}}\Big) \\
&\le   \frac{C}{  a_{n_{i_0}}^{(i_0)} x_{i_1} } \,  \Big( \frac{n}{n_{i_1}} + e^{- c (a_{n_{i_1}}^{(i_1)}/a_n^{(i_1)})^2} \ind_{\{\ga_{i_1}=2\}} \Big) \leq   \frac{C}{  a_{n_{i_0}}^{(1)}  a_{n_{i_0}}^{(2)} } \cdot  \frac{a_{n_{i_0}}^{(i_1)}}{a_{n_{i_1}}^{(i_1)}}  \frac{n}{n_{i_1}} \, .
\end{align*}
For the second inequality, we used that $x_{i_0} \vee a_{n}^{(i_0)} \geq a_{n_{i_0}}^{(i_0)}$, that  $x_{i_1} \geq a_{n}^{(i_1)}$ for all $n \leq n_{i_0}$, and also that  $x_{i_1} \sim a_{n_{i_1}}^{(i_{1})}$, so that in particular $\gp_{i_1}(x_{i_1}) x_{i_1}^{-\gamma_{i_1}} \le c /n_{i_1} $ by definition of $a_{n_{i_1}}^{(i_1)}$. 
Therefore, we get that
\begin{align}
\sum_{n=1}^{m} \bP(\bS_n =\x) &\le \frac{C' n_{i_0} }{ a_{n_{i_0}}^{(1)}  a_{n_{i_0}}^{(2)}} \times \frac{a_{n_{i_0}}^{(i_1)}}{a_{n_{i_1}}^{(i_1)}} \frac{m}{n_{i_1}} \frac{m}{n_{i_0}} \, .
\label{smallerm}
\end{align}

For the second part in \eqref{awaysplit2}, we fix some $\gd>0$ (small), and we use the local limit theorem \eqref{LLT} to get that (using that $\ga_1^{-1}+\ga_2^{-1}>1$)
\begin{align}
\sum_{n=m+1}^{+\infty} \bP( \bS_n = \x ) & \le \sum_{n=m +1}^{+\infty} \frac{C}{ a_{n}^{(1)}  a_n^{(2)}}   \le \frac{c m}{ a_{m}^{(1)}  a_{m}^{(2)}} \notag\\
&\le \frac{ c_{\gd} n_{i_0}}{a_{n_{i_0}}^{(1)}  a_{n_{i_0}}^{(2)} } \times \Big( \frac{m}{n_{i_0}}\Big)^{-\kappa_{\gd}}
\label{largerm}
\end{align}
with $\kappa_{\gd} = \ga_1^{-1} + \ga_2^{-1}-1 -\gd >0$. We used in the last inequality that $n/(a_n^{(1)} a_n^{(2)})$ is regularly varying with exponent $1 - \ga_1^{-1} -\ga_2^{-1}$, together with Potter's bound.

Then, it remains to optimize our choice of $m$: combining~\eqref{largerm} with \eqref{smallerm}, and using that there is a constant $c_{\gd}$ such that $ a_{n_{i_0}}^{(i_1)}/ a_{n_{i_1}}^{(i_1)} \le c_{\gd}(n_{i_1}/n_{i_0})^{-1/\ga_{i_1} +\gd}$, we get that
\begin{align}
  \label{optimizeinn}
  G(\x) & \le \frac{C n_{i_0} }{ a_{n_{i_0}}^{(1)}  a_{n_{i_0}}^{(2)}} \Big(   \Big(\frac{n_{i_1}}{n_{i_0}} \Big)^{-1-\frac{1}{\ga_{i_1}} +\gd}  \Big(\frac{m}{n_{i_0}}\Big)^2  + \Big(\frac{m}{n_{i_0}} \Big)^{-\kappa_{\gd}}\Big) \\
  & \le   \frac{C n_{i_0} }{ a_{n_{i_0}}^{(1)} \cdots a_{n_{i_0}}^{(d)}}   \times \Big(\frac{n_{i_1}}{n_{i_0}} \Big)^{-\kappa_{\gd}(1+1/\ga_i +\gd)/(2+\kappa_{\gd})}\, .
\notag
\end{align}
For the last inequality, we optimized in $m$, by choosing $m/n_{i_0} = (n_{i_1}/n_{i_0})^{ (1+1/\ga_{i_1} -\gd)/(2+\kappa_{\gd})}$. This gives the first part of the statement, \textit{i.e.}~\eqref{casIIaway}.

\smallskip
In the case where $(\bS_n)_{n\geq 0}$ is a renewal process, \textit{i.e.}\ $\bX_1\geq 0$, (in particular $\ga_{i}\in (0,1)$), then one has a much sharper estimate than~\eqref{largerm}. Indeed, we have the following large deviation result: for all $n \ge C'_{13} n_1$ (so that $x_{i_0} \leq   a_{n}^{(i_0)}/4$),
\begin{equation}
\bP(\bS_n = \x)  \le \frac{C}{a_n^{(1)}  a_n^{(2)}} \bP(S_{\lfloor n/2 \rfloor}^{(i_0)} \le x_{i_0}/2 ) \le  \frac{C}{a_n^{(1)}  a_n^{(2)}} e^{- c (x_{i_0}/a_n^{(i_0)})^{- \zeta_{i_0} } } \, ,
\end{equation} 
for some exponent $\zeta_{i_0}$ (whose value depends on $\ga_{i_0}$).
The first inequality follows from the same argument as for~\eqref{splithalf}-\eqref{firstpiece}, we leave the details to the reader. The second inequality is a  large deviation result, see for instance~\cite[Lemmas~A.3]{cf:AB}.
As a consequence, we have the analogous bound as in \eqref{largerm},
\[
\sum_{n=m}^{+\infty} \bP(\bS_n = \x)  \le   \frac{c_{\gd} n_{i_0}}{a_{n_{i_0}}^{(1)} a_{n_{i_0}}^{(2)}}  e^{  - c (a_{n_{i_0}}^{(i_0)} / a_m^{(i_0)} )^{- \zeta_{i_0} }  } \, .
\]
The last part decays faster than any power of $m/n_{i_0}$, so that \eqref{optimizeinn} holds with $\kappa_{\gd}$ replaced by an arbitrarily large constant. This gives the second statement of Theorem~\ref{thm:awayI}.

\subsection{Case II-III (a), proof of Theorem~\ref{thm:casII-IIIb}}
\label{sec:casII-IIIawayb}

Here, we assume $\ga_1, \ga_2 \geq 1$ and  for $i=1,2$:  either $\mu_i \in \bbR^*$, or $\ga_i=1$ with $p_i\neq q_i$.
Recall that $n_i$ is such that $b_{n_i}^{(i)} =x_i$ (we work in the case where they are both positive), and $m_i =a_{n_i}^{(i)} /\mu_i(a_{n_i}^{(i)})$ (one can replace $\mu_{i} (a_{n_i}^{(i)})$ by $\mu_{i}$ if $\ga_i>1$). Note that in any case we have $m_i = o(n_i)$ as $n_i\to+\infty$.

\subsubsection{A preliminary estimate}
We prove first the following result, that will be useful in this section and the next one: for $i=1,2$
\begin{equation}
\label{sumlog1}
\begin{split}
\tif   m \in [m_i ,n_i], & \quad \sum_{n=n_{i}- m}^{ n_{i} + m}  \frac{1}{a_n^{(i)} \vee |x_i - b_n^{(i)}|}\leq  \frac{C}{\mu_{i}(a_{n_i}^{(i)})} \log \big(1+ \frac{m}{m_{i}} \big); \\
\tif m \ge 2 n_i,&  \quad \sum_{n=1}^{m}  \frac{1}{a_n^{(i)} \vee |x_i - b_n^{(i)}|} \leq  \frac{C}{ \mu_{i}(a_{n_i}^{(i)}) \wedge \mu_i( a_m^{(i)})} \log  m \, .
\end{split}
\end{equation}

To show this, recall that $|x_i-b_n^{(i)}| \leq  ca_{n_i}^{(i)}$ for all $n$ between $n_i-m_i$ and $n_i+m_i$ (see \eqref{test:favdir1} and \eqref{tes:favdir2}).
Therefore, we have that
\[ \sum_{n=n_{i}- m_i}^{ n_{i} + m_i}  \frac{1}{a_n^{(i)} \vee |x_i - b_n^{(i)}|} \leq  \frac{c m_i}{a_n^{(i)}} \leq \frac{c}{\mu_i(a_{n_i}^{(i)} )} \, .\]

Then, for $n_i + m_i \le n\leq 2n_i$ (so $|x_i-b_n^{(i)}| \geq  c' a_{n_i}^{(i)}$, see \eqref{remarquex-b} in the case $\ga_i =1$), setting $j=n-n_i$ we have $|x_i - b_n^{(i)}|\geq c j\mu_i(a_{n_i}^{(i)})$ (this is trivial if $\ga_i>1$): we obtain for $ m\in [m_i, n_i]$
\[\sum_{n=n_i+ m_i}^{n_i+m} \frac{1}{a_n^{(i)} \vee |x_i - b_n^{(i)}|} 
\leq \sum_{j=m_i}^{m} \frac{c}{j\mu_i(a_{n_i}^{(i)})} \leq \frac{c}{\mu_i(a_{n_i}^{(i)})} \log \big(1+\frac{m}{m_i} \big) \, .
\]
A similar argument holds for the sum between $n-m$ and $n- n_i$, and this concludes the proof of the first line of \eqref{sumlog1}.

For $n\ge 2n_i$, we have $|x_i - b_{n}^{(i)}| \geq c b_n^{(i)}$ (which is larger than $a_n^{(i)}$, recall $b_n^{(i)}/a_n^{(i)}\to+\infty$). Therefore, using $b_n^{(i)} = n \mu_i(a_n^{(i)})$, we get that
\[
\sum_{n=2 n_i}^{m} \frac{1}{a_n^{(i)} \vee |x_i - b_n^{(i)}|} \leq \sum_{n=2n_i}^{m}\frac{c}{n \mu_i(a_n^{(i)})} \leq \frac{C}{\mu_i(x_1) \wedge \mu_i(b_{m}^{(i)})} \log (m/n_i)
\]
The last inequality comes from the fact that: either $\mu_i$ is a positive constant (so the bound is trivial); or $\mu_{i}=+\infty$ (if $\ga_i=1$, $p_i>q_i$) so $\mu_i(x) \sim (p_i-q_i) \ell(x)$ with $\ell(\cdot)$ non-decreasing, so $\mu_i(a_n^{(i)})\geq c\mu_i(a_{n_i}^{(i)})$; or $\mu_i=0$ (if $\ga_i=1$, $p_i<q_i$) so $\mu_i(x) \sim (q_i-p_i) \ell^{\star}(x)$ with $\ell^{\star}(\cdot)$ non-increasing, so $\mu_i(a_n^{(i)})\geq c\mu_i(a_{m}^{(i)})$.

\subsubsection{The case $n_1\leq n_2\leq 2n_1$}

Let us mention that the case $n_2\leq n_1\leq 2 n_2$ would be treated symmetrically.
We let $m= (n_2- n_1)/2$ so that $n_1+m = n_2-m$, and we assume that $m\geq m_1\vee m_2$ (otherwise we are in the favorite direction).
We write
\begin{equation}
\label{almostbalanced}
G(\x) = \Big( \sum_{n=0}^{n_1+m} + \sum_{n=n_2-m+1}^{+\infty} \Big) \bP(\bS_n = \x) \, .
\end{equation}

For the first term, we split it according to whether $n < n_1-m$ or $n\geq n_1-m$.  For $n \in (n_1-m, n_1+m)$, and  since $n \leq n_2-m_2$, we have  that $|x_2-b_n^{(2)}| \geq c m \mu_2(a_{n_2}^{(2)})$ (see~\eqref{remarquex-b}):  Theorem~\ref{thm:locallimit2} gives that
\begin{align*}
\bP(\bS_n = \x) \leq &\frac{C}{a_n^{(1)}\vee |x_1- b_n^{(1)} |} \\
& \times \Big(  n_2 \gp_{2}\big( m\mu_2(a_{n_2}^{(2)}) \big)  \big( m \mu_2(a_{n_2}^{(2)}) \big)^{-(1+\gamma_{2})}  + \frac{1}{a_n^{(2)}} e^{- c  (m \mu_2(a_{n_2}^{(2)}) /a_n^{(2)})^2} \ind_{\{\ga_{1}=2\}} \Big) \, . 
\end{align*}
Then, by \eqref{sumlog1}, and using that $n_2 \gp(a_{n_2}^{(2)}) (a_{n_2}^{(2)})^{-(1+\ga_2)} \leq  C/a_{n_2}^{(2)}$ (together with Potter's bound), we obtain for any $\gd>0$
\begin{align*}
\sum_{n=n_1-m}^{n_1+m} \bP&(\bS_n = \x) \\
&\leq  \frac{C_{\gd} \log(m/m_1)}{a_{n_2}^{(2)}\mu_{1}(a_{n_1}^{(1)})}  \times  \Big(   \Big( \frac{m}{m_2}\Big)^{- (1+\ga_2)+\gd} \big( m \mu_2(a_{n_2}^{(2)}) \big)^{\ga_2-\gamma_{2}}  +  e^{- c (m /m_2)^2 } \ind_{\{\ga_{2}=2\}} \Big)\, .
\end{align*}
Notice that $\ga_2=\gamma_2$ if $\ga_2 \in (0,2)$: we can rewrite the above as
\begin{align}
\label{finterm1}
\sum_{n=n_1-m}^{n_1+m} \bP(\bS_n = \x) \leq & \frac{C'_{\gd} \log(m/m_1)}{a_{n_2}^{(2)}\mu_{1}(a_{n_1}^{(1)})}  \times   \Big( \frac{m}{m_2}\Big)^{- (1+\ga_2)+\gd} \times R^{(2)}(m) \, ,
\end{align}
where we set $R^{(i)}(m) =1$ if $\ga_i\in (0,2)$ and $R^{(i)}(m) =m^{2-\gamma_i}  + e^{- c' (m/m_i)^2 }$ if $\ga_i=2$ (in which case $\mu_i(a_{n_i}^{(i)})$ is a constant).

For the terms with $n\leq n_1-m$, we use again Theorem~\ref{thm:locallimit2}: since $|x_1- b_n^{(1)}| \geq c m \mu_1(a_{n_1}^{(1)})$ (see~\eqref{remarquex-b}), and setting $j = n_2-n$ so that $|x_{2}-b_n^{(2)}| \geq c j \mu_2(a_{n_2}^{(2)})$, we have 
\begin{align*}
\bP&(\bS_n =\x) \\
&\leq \frac{C}{m \mu_2(a_{n_2}^{(2)})} \times \Big( n \gp_1(j \mu_1(a_{n_1}^{(1)})) \big(j \mu_1(a_{n_1}^{(1)}) \big)^{- (1+\gamma_{1})} + \frac{1}{a_{n}^{(1)}} e^{- c (j \mu_1(a_{n_1}^{(1)})/a_n^{(1)})^2}\ind_{\{\ga_1=2\}} \Big) \, .
\end{align*}
Therefore, summing over $j$ as done in \eqref{conclusionalpha1-4}-\eqref{conclusionalpha1-5} (in the case $\ga_1=1$, the case $\ga_1>1$ is identical), we obtain as in \eqref{finterm1}
\begin{align}
\sum_{n=1}^{n_1- m} \bP(\bS_n =\x)& \leq \frac{C}{m \mu_2(a_{n_2}^{(2)})}   \Big( \frac{n_1}{\mu_1(a_{n_1}^{(1)})} \gp_1(m \mu_1(a_{n_1}^{(1)})) \big( m\mu_1(a_{n_1}^{(1)}) \big)^{-\gamma_1}  +   e^{- c (m/m_1)^2} \ind_{\{\ga_1=2\}} \Big) \notag \\
& \leq  \frac{C_{\gd}}{a_{n_2}^{(2)} \mu_1(a_{n_1}^{(1)})}   \Big(\frac{m}{m_2}\Big)^{-1}  \Big( \frac{m}{m_1} \Big)^{-\ga_1+\gd}  \times R^{(1)}(m)\, .
\label{finterm2}
\end{align}

\smallskip
Similarly, we can treat the cases $n_2- m\le n \leq n_2+m $ and $n\geq n_2+m$, and we get that 
\begin{align*}
\sum_{n=n_2-m}^{n_2+m} \bP(\bS_n =\x)& \leq \frac{C'_{\gd} \log(m/m_2)}{a_{n_1}^{(1)}\mu_{2}(a_{n_2}^{(2)})}    \Big( \frac{m}{m_1}\Big)^{-(1+\ga_1) +\gd} \times R^{(1)}(m) \, ,  \\
\sum_{n=n_2+m}^{+\infty} \bP(\bS_n =\x)& \leq \frac{C'_{\gd}}{a_{n_1}^{(1)} \mu_2(a_{n_2}^{(2)})} \Big( \frac{m}{m_1}\Big)^{-1}   \Big( \frac{m}{m_2} \Big)^{-\ga_2+\gd}   \times R^{(2)}(m)\, .
\end{align*}
Notice that $a_{n_1}^{(1)} \mu_2(a_{n_2}^{(2)}) = a_{n_2}^{(2)} \mu_1(a_{n_1}^{(1)}) \times (m_1/m_2)$, so that we can re-write the above as 
\begin{align}
\label{finterm3}
\sum_{n=n_2-m}^{n_2+m} \bP(\bS_n =\x)& \leq \frac{C \log(m/m_2)}{a_{n_2}^{(2)}\mu_{1}(a_{n_1}^{(1)})} \Big( \frac{m}{m_2} \Big)^{-1}   \Big( \frac{m}{m_1}\Big)^{-\ga_1 +\gd}  R^{(1)}(m)\, ,  \\
\sum_{n=n_2+m}^{+\infty} \bP(\bS_n =\x)& \leq \frac{C}{a_{n_2}^{(2)} \mu_1(a_{n_1}^{(1)})}   \Big( \frac{m}{m_2} \Big)^{-(1+\ga_2)+\gd}  R^{(2)}(m) \, .
\label{finterm4}
\end{align}

All together, combining \eqref{finterm1}-\eqref{finterm2} and \eqref{finterm3}-\eqref{finterm4}, and bounding $\log(m/m_i)$ by a constant times $ (m/m_i)^{\gd}$, we can conclude that 
\begin{align*}
G(\x)  \leq \frac{C''_{\gd}}{a_n^{(2)} \mu_1(a_{n_1}^{(1)})} \Big(\frac{m}{m_2}\Big)^{-1+\gd} \Big( \frac{m}{m_1}\Big)^{\gd} \times \bigg\{ \Big( \frac{m}{m_1}\Big)^{-\ga_1}  R^{(1)}(m) +   \Big( \frac{m}{m_2} \Big)^{-\ga_2}  R^{(2)}(m)  \bigg\}\, ,
\end{align*}
which concludes the proof of \eqref{casII-IIIawaybi}, recalling that we chose $m= (n_2-n_1)/2$.

\subsubsection{The case $n_2\geq 2 n_1$}

Again, the case  $n_2\leq n_1/2$ would be treated symmetrically.
When $n_2\geq 2 n_1$, we choose $m=3n_2/4$ if $\gamma_1\leq \gamma_2$, and $m= n_2^{\gamma_2/\gamma_1} \vee (3 n_1/2)$ if $\gamma_1>\gamma_2$. We write
\begin{equation}
\label{allez}
G(\x) = \Big( \sum_{n=1}^{m} + \sum_{n=m+1}^{+\infty}\Big) \bP(\bS_n=\x)\, .
\end{equation}

For the first term, since $n\le m \leq 3n_2/4$, for any $\gd'$ there is a constant $c_{\gd'}$ such that $|x_2- b_n^{(2)}| \geq   c_{\gd'} n_2^{1-\gd'}$, so that Theorem~\ref{thm:locallimit2} gives that for any $\gd>0$ we have a constant $c_{\gd}$ such that, for $n\le m$,
\[
\bP(\bS_n=\x) \leq \frac{C_{\gd}}{ a_n^{(1)} \vee |x_1 - b_n^{(1)}|} \times \Big( m  \times (n_2)^{-(1+\gamma_2) +\gd}  + \frac{1}{a_{n_2}^{(2)}} e^{- c (n_2^{1-\gd'} / a_{n_2}^{(2)})^2 } \ind_{\{\ga_2=2 \}}\Big) \, .
\] 
Here, the second term decays faster than any power of $n_2$ (if the term is present, $a_{n_2}^{(2)}$ is regularly varying with exponent $1/2$), so it is negligible compared to the first term.
Then summing over $n\le m$ and using \eqref{sumlog1}, we obtain that
\begin{align}
\label{fin2term1}
\sum_{n=1}^m  \bP(\bS_n=\x) \leq   C'_{\gd} m^{1+\gd}  (n_2)^{-(1+\gamma_2) +\gd}\, .
\end{align}

For the second term in \eqref{allez}, 
we consider three parts: $m< n < 3n_2/4$; $ 3n_2/4 \leq n \leq 2n_2 $; and $n > 2n_2$.
For $m < n \leq  3n_2/4$, we use again that $|x_2 - b_n^{(2)}| \geq c_{\gd'} n_2^{1-\gd'}$ and  $|x_1- b_n^{(1)}| \geq c b_n^{(1)}$ (recall $m\geq 3n_1/2$), so that Theorem~\ref{thm:locallimit2} again gives that for any $\gd>0$
\[
\bP(\bS_n=\x) \leq C_{\gd} n_2^{-1+\gd} \times \Big( n (b_n^{(1)})^{-(1+\gamma_1)+\gd}  + \frac{1}{a_n^{(1)}} e^{- c  (b_n^{(1)} / a_{n}^{(1)} )^2} \ind_{\{\ga_1=2 \}}\Big) \leq   C_{\gd} n_2^{-1+ \bar \gd} \times n^{-\gamma_1-\gd}\, .
\]
For the last inequality, we discarded the exponential term since it decays faster than any power in $n$, and used that  $n\geq m \geq n_2^{1\wedge (\gamma_2/\gamma_1)}$ to bound $n_2^{\gd} n^{2\gd}$ by $n_2^{\bar\gd} n^{-\gd}$ (by picking $\bar \gd= \gd+ 3\gd (1\wedge (\gamma_2/\gamma_1))$) .
Summing over $n \geq m$, and since $\gamma_1+\gd>1$, we get that
\begin{align}
\label{fin2term2}
\sum_{n=m}^{3n_2/4}  \bP(\bS_n=\x)  \leq C'_{\gd} n_2^{- 1+\bar \gd} m^{1-\gamma_1} \, .
\end{align}

For $3n_2/4 \le n \leq 2n_2$, we use that $|x_1- b_n^{(1)}| \geq c n_2^{1-\gd}$ to get by Theorem~\ref{thm:locallimit2} that (discarding the exponential term as above)
\[ \bP(\bS_n=\x) \leq  \frac{C_{\gd} }{a_n^{(1)}\vee |x_2-b_n^{(2)}|} \times n_2 \, (n_2)^{- (1+\gamma_1) +\gd} \, , \]
so that, summing over $n$, \eqref{sumlog1} gives that $\sum_{n=3n_2/4}^{2 n_2} \bP(\bS_n =\x) \leq C'_{\gd} n_2^{-\gamma_1 +2\gd} $ which is smaller than a constant times $ n_2^{-1 +2\gd} m^{1-\gamma_1}$ since $m \leq  n_2$ (and $\gamma_1\geq 1$).

For $n > 2 n_2$, we use that both $|x_1-b_n^{(1)}|$ and $|x_2- b_n^{(2)}|$ are larger than $c_{\gd'} n^{1-\gd'}$: Theorem~\ref{thm:locallimit2} gives that $\bP(\bS_n =\x) \leq C n^{- (1+\gamma_2) + \gd}$. Then, summing over $n\geq 2n_2$ gives that
$\sum_{n> 2n_2} \bP(\bS_n =\x) \leq C_{\gd} n_2^{-\gamma_2 +2\gd} \leq  C'_{\gd} n_2^{-1 +2\gd} m^{1-\gamma_1}$.

To conclude the argument, we get that $G(\x)$ is bounded by a constant times
\[
 m^{1+\gd} n_2^{-(1+\gamma_2) +\gd} + n_2^{-1 +\gd}  m^{1-\gamma_1}
\leq C
\begin{cases}
n_1 \times n_2^{-(1+\gamma_2) +\gd'} &  \tif n_2^{\gamma_2/\gamma_1} \leq n_1 \, , \\
n_2^{-\gamma_2 + (\frac{\gamma_2}{\gamma_1}-1) +\gd'} &  \tif \gamma_1>\gamma_2 , n_2^{\gamma_2/\gamma_1} > n_1 \, ,\\
n_2^{-\gamma_2 +\gd'} &  \tif \gamma_1\leq \gamma_2\, ,
\end{cases}
\]
where we used the definition of $m$ in the last inequality (indeed, the optimal $m$ is $n_2^{\gamma_2/\gamma_1}$ but we have the additional conditions that $m\geq n_1$ and $n\leq n_2$).
This concludes the proof of \eqref{casII-IIIawaybii}.

\subsubsection*{About Assumption~\eqref{eq:Will}}
Let us explain briefly why assuming~\eqref{eq:Will} does not improve much the bounds \eqref{balancedawayga>1}-\eqref{balancedawayga=1}---we consider the case when $n_1\leq n_2\leq 2 n_1$. Set again $m= n_2-n_1$.
For
$n\leq n_2- m$ ($= n_1+m$) we have that 
$|b_n^{(2)} -x_2| \geq c (n- n_2) \mu_2(a_{n_2}) $ so that Theorem~\ref{thm:locallimit3} gives that  $\bP(\bS_n=\x)$ is bounded by a constant times
\[n_2 \gp\big( (n- n_2) \mu_2(a_{n_2}) \big) \big( (n- n_2) \mu_2(a_{n_2}) \big)^{- (2+\gamma)} + \frac{1}{(a_n)^2}e^{- c((n- n_2)\mu_2(a_{n_2}) /a_n)^2} \ind_{\{\ga=2\}}\, . \]

Hence, summing over $j = n_2-n $ between $m$ and $n_2$, and using that $n_2 \leq \gp(a_{n_2}) (a_{n_2})^{-\ga}$ (and $m_2=a_{n_2}/\mu_2(a_{n_2})$) we get that
\begin{align*}
\sum_{n=0}^{n_1 +m} \bP(\bS_n = \x)  &\leq \frac{C}{\mu_2(a_{n_2})}  n_2 \gp\big( m \mu_2(x_2) \big)  \big( m \mu_2(a_{n_2}) \big)^{- (1+\gamma)} + \frac{C  }{a_{n_2} \mu_2(a_{n_2})}e^{- c (m/a_{n_2})^2} \ind_{\{\ga=2\}} \, \\
& \leq \frac{C}{ a_{n_2} \mu_2(a_{n_2})}   \Big( \Big( \frac{m}{m_2}\Big)^{- (1+\ga)+\gd} (m)^{\ga-\gamma} + e^{- c (m/m_2)^2} \ind_{\{\ga=2\}} \Big)\, , 
\end{align*}
which gives the upper bound in \eqref{finterm1} up to the factor $\log (m/m_1)$. Similar bounds hold for the other terms, showing that the use of Theorem~\ref{thm:locallimit3} does not bring any real improvement.

\subsection{Case II-III (b), proof of Theorem~\ref{thm:casII-IIIa}}
\label{sec:casII-IIIawaya}

Here, we assume that  $b_n^{(i_1)}\equiv 0$  and that either $\ga_{i_1} \geq 1$ and $\mu_{i_0}\in \bbR^*$, or $\ga_{i_1}=1$  with $p_{i_0}\neq q_{i_0}$.
Recall the definitions~\eqref{def:ni} of $n_i$: in particular, $n_{i_0}\sim x_{i_0} /\mu_{i_0} (a_{n_{i_0}}^{(i_0)})$ (we work in the case where both $x_{i_0}$ and $\mu_{i_0} (a_{n_{i_0}}^{(i_0)})$ are positive), and $a_{n_{i_1}}^{(i_1)} \sim |x_{i_1}|$. We also define $m_{i_0}:= a_{n_{i_0}}^{(i_0)}/\mu_{i_0}(a_{n_{i_0}}^{(i_0)})$ as above.
The case where $n_{i_1}/n_{i_{0}}$ is bounded away from $0$ and $+\infty$ corresponds to the favorite scaling \eqref{def:favdir}.

\subsubsection{The case $n_{i_1}/n_{i_0}$ small}
\label{sec:casII-IIIawaya1}
Let us start with the case when $n_{i_1} \leq c n_{i_0}$ with $c$ a small constant.
We split $G(\x)$ into 
\begin{equation}
\label{cas1:2parts}
G(\x) = \Big( \sum_{n=1}^{ n_{i_0}/2 } + \sum_{n= n_{i_0}/2 +1}^{+\infty}
 \Big) \bP(\bS_n =\x) \, .
\end{equation}

For the second term, we can use the results of Section~\ref{sec:casII} for the case $\ga_{i_0}>1$ (combine \eqref{conclusion:main}-\eqref{conclusionlargen1}-\eqref{conclusionsmalln1}) and of Section~\ref{sec:casIII} for the case $\ga_{i_0}=1$ (combine \eqref{conclusionmainga=1}-\eqref{conclusionalpha1-3}-\eqref{lastsplitga=1}-\eqref{conclusionalpha1-4}).
We get
\begin{align}
\label{eq:casII-IIIaway1111}
\sum_{n= n_{i_0}/2 +1}^{+\infty} \bP(\bS_n =\x) \leq  \frac{C}{  \mu_{i_0}(a_{n_{i_0}}^{(i_0)}) a_{n_{i_0}}^{(i_1)} } \, .
\end{align}
Recall that if  $\mu_{i_0} \in \bbR_+^*$, then we can replace $ \mu_{i_0}(a_{n_{i_0}}^{(i_0)})$ by a constant.

For the first term in \eqref{cas1:2parts}, we use that for $n\leq n_{i_0}/2$ we have that  $x_{i_0} - b_n^{(i_0)} \geq x_{i_0}/4 \geq a_n^{(i_0)}$ so that Theorem~\ref{thm:locallimit2} gives 
\begin{align}
\bP(\bS_n = \x) & \leq \frac{C}{|x_{i_1}| \vee a_n^{(i_1)}}  \Big(  n \gp_{i_0}( x_{i_0}) x_{i_0}^{- (1+\gamma_{i_0})} +  \frac{1}{a_n^{(i_0)}} e^{- c x_{i_0} / a_n^{(i_0)}} \ind_{\{ \ga_{i_0}=2\}}\Big) \notag\\
& \leq \frac{C'}{|x_{i_1}| \vee a_n^{(i_1)}}  \times  n \gp_{i_0}( x_{i_0}) x_{i_0}^{- (1+\gamma_{i_0})} \, .
\label{bornebonne}
\end{align}
We discarded the second (exponential) term since it decays faster than any power of $n_{i_0}$.
Summing over $n$, we get that (recall $|x_{i_1}| = a_{n_{i_1}}^{(i_1)}$)
\begin{align}
\sum_{n=1}^{n_{i_0}/2} \bP(\bS_n =\x) &\leq C \gp_{i_0}( x_{i_0}) x_{i_0}^{- (1+\gamma_{i_0})} \Big( \sum_{n=1}^{n_{i_1}} \frac{n}{|x_{i_1}|} + \sum_{n=n_{i_1}+1}^{n_{i_0}} \frac{n}{a_n^{(i_1)}} \Big) \notag \\
&\leq  \gp_{i_0}( x_{i_0}) x_{i_0}^{- (1+\gamma_{i_0})}  \times 
\begin{cases}
C_{\gd}\, n_{i_1}^{2- 1/\ga_{i_1} } n_{i_0}^{\gd} & \ \tif \ga_{i_1} \leq 1/2 , \\
C\, n_{i_0}^2 /a_{n_{i_0}}^{(i_1)} & \ \tif \ga_{i_1} >  1/2 .
\end{cases}
\label{away:smalln}
 \end{align}
We used that $n/a_n^{(i_1)}$ is regularly varying with exponent $1- 1/\ga_{i_1}$, which is smaller (resp.\ larger) than $-1$ if $\ga_{i_1} >1/2$ (resp.~if $\ga_{i_1}<1/2$).

Notice that since  $n_{i_0} \sim x_{i_0}/ \mu_{i_0}(a_{n_{i_0}}^{(i_0)})$, we get that
\begin{equation}
\label{ga>12}
 \gp_{i_0}( x_{i_0}) x_{i_0}^{- (1+\gamma_{i_0})} \frac{n_{i_0}^2}{a_{n_{i_0}}^{(i_0)}} \leq  \frac{C}{a_{n_{i_0}}^{(i_1)} \mu_{i_0}(x_{i_0})}  \frac{\gp_{i_0}(x_{i_0})}{\mu_{i_0}(a_{n_{i_0}}^{(i_0)})} (x_{i_0})^{1-\gamma_{i_0}} \leq \frac{C}{ \mu_{i_0}(a_{n_{i_0}}^{(i_0)}) a_{n_{i_0}}^{(i_1)} } \, .
\end{equation}
Indeed, either $\gamma_{i_0}>1$, or $\gamma_{i_0}=1$ in which case  $\gp_{i_0}(x_{i_0})/\mu_{i_0}(x_{i_0}) \to 0$ as $x_{i_0}\to +\infty$. 
Together with~\eqref{eq:casII-IIIaway1111}, this gives~\eqref{casIaway}.

%

\subsubsection*{The case of a renewal process}

If $\bS_n$ is a renewal process (necessarily $\ga_{i_1}<1$),  we can improve the bound \eqref{casIaway}.
For a fixed $\gd>0$, we set $m:= n_{i_1} \times (n_{i_0}/n_{i_1})^{\gd}$, which is larger than $n_{i_1}$, but smaller than $n_{i_0}/2$. We write
\begin{align*}
G(\x) = \Big( \sum_{n=1}^{m} + \sum_{n=m+1}^{+\infty} \Big) \bP(\bS_n = \x) \, .
\end{align*}

The first term is treated as above, using \eqref{bornebonne} (note also that $x_{i_1} \vee a_{n}^{(i_1)} \geq c a_{n_{i_1}}^{(i_1)}$):
\begin{align*}
\sum_{n=1}^{m}\bP(\bS_n = \x) &\leq C \frac{m^2}{a_{n_{i_1}}^{(1)}} \gp_{i_0}(x_{i_0}) x_{i_0}^{-(1+\gamma_{i_0})} \leq C_{\gd}\, (n_{i_1})^{2- \frac{1}{\ga_{i_1}}} x_{i_0}^{-(1+\gamma_{i_0})+2\gd} \, ,
\end{align*}
where we used that $m:= n_{i_1} \times (n_{i_0}/n_{i_1})^{\gd}$, together with  $n_{i_0}\leq x_{i_0}^{1+\gd}$, and also $a_n^{(i_1)} \geq n_{i_1}^{1/\ga_{i_1} -\gd}$.

For the remaining term, we use that, 
\begin{align*}
\sum_{n>m}\bP(\bS_n = \x) & = \bP\big( \exists\ n>m \text{ such that } \bS_n =\x \big)  \leq \bP \big( S_{m}^{(i_1)}  \leq x_{i_1} \big) \\
& \leq \exp\big( - c (a_m^{(i_1)} / a_{n_{i_1}}^{(i_1)} )^{- \zeta_{i_1} } \big) \leq  \exp\big( - \big( n_{i_0}/n_{i_1} \big)^{\zeta_{\gd}} \big)\, .
\end{align*}
for some $\zeta_{\gd}>0$.
The second line follows from the large deviation result~\cite[Lemmas~A.3]{cf:AB} (we have $x_{i_1}\geq   a_{m}^{(i_1)}/2$), with an exponent $\zeta_{i_1}$ that depends on $\ga_{i_1}<1$.
The second inequality comes from the definition of $m= n_{i_1} \times (n_{i_0}/n_{i_1})^{\gd}$.

\subsubsection{The case $n_{i_1}/n_{i_0}$ large}
\label{sec:casII-IIIawaya2}

Let us take some  $m \in (2 n_{i_0},n_{i_1}/2)$ (the choice is optimized below), and we write
\begin{equation}
\label{eq:awayII-IIIa2}
G(\x) = \Big( \sum_{n=1}^{m} + \sum_{n= m+1}^{+\infty}  \Big) \bP(\bS_n =\x) \, .
\end{equation}
For the first term, we us that $x_{i_1}-b_n^{(i_1)} \geq c x_{i_1} \geq a_n^{(i_1)}$ ($n\leq n_{i_1}/2$), so that thanks to Theorem~\ref{thm:locallimit2} we have
\begin{align}
\bP(\bS_n = \x) \leq \frac{C}{ |x_{i_0}- b_{n}^{(i_0)}| \vee a_n^{(i_0)}} \times \frac{1}{x_{i_1}}  \Big( m \gp_{i_1}(x_{i_1}) x_{i_1}^{-\gamma_{i_1}} + e^{ - c (x_{i_1} / a_{n}^{(i_1)})^2 } \ind_{\{\ga_{i_1}=2\}} \Big)\, .
\label{eq:casII-IIIawaya2111}
\end{align}
Therefore, summing over $n\leq m$ and using \eqref{sumlog1} above, we get that
\begin{align}
\sum_{n=1}^m \bP(\bS_n =\x)
&\leq \frac{C \log m }{ \bar \mu_{i_0}(n_{i_0}, m)   } \,  \frac{1}{x_{i_1}} \Big( m \gp_{i_1}(x_{i_1}) x_{i_1}^{-\gamma_{i_1}} + e^{ - c (x_{i_1} / a_m^{(i_1)})^2 }\ind_{\{\ga_{i_1}=2\}}\Big)\, ,
\label{awayII-IIIa211}
\end{align}
with $ \bar \mu_{i_0}(n_{i_0}, m) :=  \mu_{i_0}(a_{n_{i_0}}^{(i_0)}) \wedge \mu_{i_0}(a_m^{(i_0)}) \geq n_{i_0}^{-\gd} m^{-\gd/2}$.

For the second term in \eqref{eq:awayII-IIIa2},  we use that $|x_{i_0}-b_n^{(i_0)}| \geq c b_n^{(i_0)}$ for $n > 2n_{i_0}$, so that Theorem~\ref{thm:locallimit2} gives us
\[
\bP(\bS_n = \x) \leq \frac{C}{x_{i_1} \vee a_n^{(i_1)}}  n \gp_{i_0}( b_n^{(i_0)} ) (b_n^{(i_0)})^{-(1+\gamma_{i_0})}  \, .
\]
As in \eqref{bornebonne}, we discarded the exponential term appearing in the case $\ga_{i_0}=2$, since it decays faster than any power of $n$.
Therefore, in the case $\gamma_{i_0}>1$, we have
\begin{equation}
\sum_{n=m+1}^{+\infty} \bP(\bS_n = \x)  \leq \frac{C}{a_{n_{i_1}}^{(i_1)}} \sum_{n\geq m} \gp_{i_0}( n ) n^{-\gamma_{i_0}}   \leq \frac{C}{a_{n_{i_1}}^{(i_1)}} \gp_{i_0}(m) m^{1-\gamma_{i_0}}\, . 
\label{awayII-IIIa212}
\end{equation}
For the case $\gamma_{i_0}=1$, we refer to the bound \eqref{generalbound} below: we find that the sum is bounded by a constant times $(a_{n_{i_1}}^{(i_1)} \mu_{i_0}(a_{n_{i_0}}^{(i_0)}))^{-1}$.

Combining \eqref{awayII-IIIa211} and \eqref{awayII-IIIa212}, and using $x_{i_1}=a_{n_{i_1}}^{(i_1)}$, we get that in the case $\gamma_{i_0}>1$
\begin{align}
\label{notgeneralbound}
G(\x) & \leq \frac{C_{\gd}}{a_{n_{i_1}}^{(i_1)}} \times  \Big( m^{1-\gamma_{i_0} +\gd}  +  m^{\gd} n_{i_0}^{\gd}\,   \big( m  x_{i_1}^{-\gamma_{i_1} +\gd}  +   e^{- c (x_{i_1} / a_{m}^{(i_1)})^2 } \ind_{\{ \ga_{i_1}=2 \} }  \big) \Big) \, .
\end{align}
Then, we are left with the choice of $m$: we take $m\leq n_{i_1}^{1-\gd}$ so that the last term decays faster than any power of $m$, and we only have to optimize $m^{1+\gd}  \times (m^{-\gamma_{i_0}} + x_{i_1}^{-\gamma_1} (n_{i_0} x_{i_1})^{\gd})$.
We therefore choose $m \geq  (x_{i_1})^{\gamma_{i_1}/\gamma_{i_0}} \vee n_{i_0}$. In the end, we obtain that the parenthesis in \eqref{notgeneralbound} is bounded by
$m^{1+\gd}  (n_{i_0} x_{i_1})^{\gd} x_{i_1}^{-\gamma_{i_1}} \leq m_{\gd'} x_{i_1}^{-\gamma_{i_1}}$, where we have put $m_{\gd'}= (|x_{i_1}|^{\gamma_{i_1}/\gamma_{i_0} +\gd'} \vee (n_{i_0})^{1+\gd'}) \wedge (n_{i_1})^{1-\gd'}$
for some $\gd'$ slightly larger than $\gd$
(if $(n_{i_1})^{1-\gd'} < (n_{i_0})^{1+\gd'}$ then take $m_\gd=+\infty$). 
This gives~\eqref{casIawayii}.

\smallskip
Alternatively, we also obtain from Theorem~\ref{thm:locallimit2} that for any $n$ (using  $x_{i_1} \vee a_n^{(i_1)} \geq a_{n_{i_1}}^{(i_1)}$)
\begin{align}
\bP(\bS_n = \x) &\leq \frac{C}{a_{n_{i_1}}^{(i_1)}}  \times \frac{1}{a_n^{(i_0)} \vee |x_{i_0} - b_n^{(i_0)}| }  \notag \\
&\times  \Big( n \gp_{i_0}( |x_{i_0} - b_n^{(i_0)}|)|x_{i_0} - b_n^{(i_0)}|^{-\gamma_{i_0}}  + e^{- c (|x_{i_0} - b_n^{(i_0)}| /a_n^{(i_0)})^2} \ind_{\{\ga_{i_0}=2\}}\Big) \,
\label{eq:casII-IIIawaya21112}
\end{align}
Summing over $n$, and treating the different parts of the sum  according to whether $n\leq n_{i_0} - m_{i_0}$, $n\in (n_{i_0}- m_{i_0}, n_{i_0}+m_{i_0})$ (in which range $|x_{i_0} - b_n^{(i_0)}|\leq a_n^{(i_0)}$) or $n\geq n+m_{i_0}$, we obtain as in Sections~\ref{sec:casII}-\ref{sec:casIII} that
\begin{equation}
\label{generalbound}
\sum_{n=1}^{+\infty}\bP(\bS_n = \x) \leq \frac{C}{a_{n_{i_1}}^{(i_1)}}  \times \frac{1}{\mu_{i_0}( a_{n_{i_0}}^{(i_0)})}\, . 
\end{equation}
This gives a general bound, in case the above~\eqref{notgeneralbound} does not give a satisfactory bound (for instance if $m_{\gd'}=+\infty$).

\subsubsection*{The case of a renewal process}

In that case, we have that $\bP(\bS_n =\x) = 0$ for all $n\geq x_{i_0} \sim n_{i_0} \mu(a_{n_{i_0}}^{(i_0)})$ (notice that $\mu_{i_0}(x)\geq 1$).
Hence, we have that 
\[G(\x) = \sum_{n=1}^{2 n_{i_0} \mu_{i_0}(a_{n_{i_0}}^{(i_0)})} \bP(\bS_n = \x) \leq  \frac{C \log n_{i_0}}{\mu_{i_0}(a_{n_{i_0}}^{(i_0)}) a_{n_{i_1}}^{(i_1)}} \Big( n_{i_0} \gp_{i_1}(x_{i_1}) x_{i_1}^{-\gamma_{i_1}} + e^{- c (a_{n_{i_1}}^{(i_1)}/a_{n_{i_0}}^{(i_1)} )^2}  \ind_{\{\ga_{i_1}=2\}}\Big)   \, ,\]
where we used \eqref{awayII-IIIa211} with $m= 2 n_{i_0} \mu_{i_0}(a_{n_{i_0}}^{(i_0)})$, and  also that $\mu_{i_0}(a_{n_{i_0}}^{(i_0)}) \geq \mu_{i_0}(a_m^{(i_0)})$ since $\mu$ is non-decreasing. Noting that $a_{n_{i_1}}^{(i_1)}/a_{n_{i_0}}^{(i_1)}$ is bounded below by $(n_{i_1}/n_{i_2})^{(1-\gd)/2}$ in the case $\ga_{i_1}=2$, 
this gives the bound in Theorem~\ref{thm:casII-IIIa}.

\subsection*{Acknowledgment}
The author wishes to thank Ron Doney for useful comments and for pointing out some references.


\begin{appendix}

\section{Generalized domains of attraction  and multivariate regular variation}
\label{appA}

\subsection{A few words on generalized domains of attraction}
We stress that the convergence \eqref{cf:stableconv} is  a special case of generalized domains of attractions (called operator-stable distributions): in general, the renormalization matrix $A_n$ is invertible, and does not need to be diagonal as in our case.  A few relevant and historical references are Sharpe \cite{cf:Sharpe69}, Hudson \cite{cf:Hud80}, Hahn and Klass \cite{cf:HK79, cf:HK85}, and a comprehensive overview of the subject can be found in \cite{cf:MSbook}. We stress that a local limit theorem exists in general, see~\cite{cf:Grif}.
We also mention that when $A_n$ is diagonal, all marginals $X^{(i)}$ are in the domain of attraction of an $\ga_i$-stable distribution, which is not necessarily the case in operator-stable distributions, cf.\ \cite{cf:MS98}.

Sharpe \cite{cf:Sharpe69} found that one can decompose a multivariate (operator-)stable distribution into the  product of two marginals, one normal and one strictly non-normal. In our setting, it means that if we set $d_0 =\max\{i;\ga_i=2\}$, then the stable law $\bZ$ in \eqref{cf:stableconv} has two independent components: $(Z_1,\ldots, Z_{d_0})$ normal, and $(Z_{d_0+1}, \ldots, Z_d)$ strictly non-normal, the convergence of these two marginals being enough for the joint convergence see \cite{cf:RG79, cf:Meer91}. Then, we refer to \cite{cf:MSbook} for a characterization of the convergence to an operator-stable distribution (either normal or strictly non-normal),
 in terms of regular variation in $\bbR^d$ of the distribution of $\bX_1$ (this is a generalization of Feller conditions  \cite[\S\ IX.8]{cf:Feller}, \textit{i.e.}~\eqref{hyp:XY}, to the multivariate case).

In the simpler case we are interested in, that is when the matrix $A_n$ is diagonal, Resnick and Greenwood \cite{cf:RG79} (resp.~Haan Omey and Resnick \cite{cf:HOR84}) first gave a characterization of the domains of attraction in dimension~$2$ (resp.~$d$), with the help of a (simpler) theory of regular variation in $\bbR^d$. 
We summarize it below, but we first recall the definition of regularly varying function in $\bbR^d$, as introduced in \cite{cf:RG79,cf:HOR84}.

\subsection{About regular variation in $\bbR^d$, and convergence to stable distributions}

The theory of regular variations in $\bbR$ is well established, and an exhaustive and seminal reference is \cite{cf:BGT}.
The study of regular variation in $\bbR^d$ turns out to be very rich, and has also been extensively studied, starting with \cite{cf:RG79,cf:HOR84,cf:Meer88}: we refer to \cite[Part II]{cf:MSbook} and references therein for more details. Here we give a brief (simplified) definition in the special case we are interested in.

\smallskip
First, a function $r:\bbR_+ \to \bbR_+^d$ is said to be regularly varying with exponent $\beta = (\beta_1,\ldots, \beta_d)$ if the components $r_i(\cdot)$ are (one-dimensional) regularly varying functions with respective indices $\beta_i$.
Then, we say that the function $f: \bbR^d \to \bbR$ is regularly varying at $+\infty$ (resp.~$0$), if there exists a regularly varying function $r: \bbR_+ \to \bbR_+^d$ (called auxiliary function) with index $\beta\in \bbR_+^{d}$ (resp.~$ \beta\in \bbR_-^d$), and $\gep \in\{-1,+1\}$ such that
\begin{equation}
\label{def:regvarfun}
 \lim_{t\to +\infty}  \frac{ f\big( r(t)  \x  \big)}{t^{\gep}} = \phi(\x)\,  \qquad \forall \, \x\in\bbR^d \setminus \{0\}\, .
\end{equation}
The function $\phi: \bbR^d \setminus \{0\} \to \bbR^d$ verifies $\phi(\lambda^{\beta} \x) = \lambda^{\gep}  \phi(\x)$ for all $\x\neq \mathbf{0}$ and $\lambda>0$, where we denoted $\lambda^{a} = (\lambda^{a_1}, \ldots, \lambda^{a_d})$. Then, $\rho=\gep \gb^{-1}$ is called the index of regular variation of~$f$.

Similarly, a measure $\pi$ is regularly varying at $+\infty$ (resp.~$0$) if there exists a regularly varying function $r: \bbR_+ \to \bbR_+^d$ with index $\beta\in \bbR_+^{d}$ (resp.~$ \beta\in \bbR_-^d$), and $\gep \in\{-1,+1\}$ such that, 
\begin{equation}
\label{def:regvarmeas}
 \lim_{t\to +\infty}   \frac{\pi \big( r(t) {\rm d}\x \big)}{t^{\gep}} = \varpi({\rm d}\x)\, .
\end{equation}
for some measure $\varpi$ which cannot be supported on any proper subspace of $\bbR^d$, and which verifies $\varpi(\lambda^{\beta} {\rm d}\x) = \lambda^{\gep} \varpi( {\rm d} \x)$. Then $\rho=\gep \gb^{-1}$ is called the index of regular variations of $\pi$.

\smallskip
We are now ready to give a necessary and sufficient condition for $\bS$ to be in the domain of attraction of an $\bga$-stable distribution (in the case of a diagonal $A_n$), as stated in \cite{cf:RG79,cf:HOR84}. Since the convergence of the two marginals $(Z_1,\ldots, Z_{d_0})$ (to a normal law) and $(Z_{d_0+1}, \ldots, Z_d)$ (to a strictly non-normal law) is enough, we state the results in the case where $d_0=d$ or $d_0=0$.
First, $\bS$ is in the domain of attraction of a non-degenerate normal law if and only if the truncated second moment function $\mathfrak{S}(\x):=\bE \big[ \langle \bX_1, \x\rangle^{2} \ind_{\{ |\langle \bX_1, \x \rangle | < 1 \}} \big]$ is regularly varying at $+\infty$ with index $(2,\ldots, 2)$. On the other hand, $\bS$ is in the domain of attraction of a  strictly non-normal law if and only if $\bP_{\bX_1}(\cdot)$ is regularly varying at infinity with index $(\rho_1,\ldots,\rho_d) \in (-2,0)^d$ (with $\rho_i=-\alpha_i$, the scaling sequences being $a_n^{(i)} = r_i(n)$, with $r_i(\cdot)$ defined by \eqref{def:regvarmeas}). We stress that having $\bP(\bX_1 > \x )$ regularly varying at $+\infty$ with index $-(\gamma_1,\ldots, \gamma_d)$ is a sufficient condition  sufficient condition for being in the domain of attraction of an $\bga$-stable distribution, with $\bga = (\ga_1,\ldots, \ga_d)$, $\ga_i=\gamma_i\wedge 2$.

\smallskip 
Finally, let us give two examples of regularly varying distribution of $\bX_1$ ($\bbN^d$-valued) we have in mind, that can be thought as ``fully independent'' and ``fully dependent'' cases---one can easily think about other, intermediate, cases.

\begin{example}
\label{ex:indep}
There are exponents $(\gamma_i)_{1\le i\le  d}$ and regularly varying functions $(\gp_i)_{1\le i\le d}$ such that
\begin{equation}
\bP(\bX_1 \ge \x) = \prod_{i=1}^d \gp_i(x_i) \, x_i^{- \gamma_i}\, , \qquad \text{for } \x\in\bbN^d \, .
\end{equation}
We have  $\bar F_i(x)\sim \gp_i(x_i) \, x_i^{- \gamma_i}$, and $\bP(\bX_1 > \x) $ is regularly varying at $+\infty$ with index $-(\gamma_1,\ldots, \gamma_d)$.
\end{example}
Indeed, if we set $r_i(t)$ such that $\bar F_i( r_i(t) ) \sim t^{-1}$ as $t\to+\infty$,  we have that  $r(t) = (r_1(t), \ldots, r_d(t)) $ is regularly varying with index $(\gamma_1^{-1}, \ldots, \gamma_d^{-1})$. Then, for any $\x \in \bbR^d \setminus \{0\}$, we get
\[\lim_{t\to\infty} t \bP( \bX_1 > r(t) \x ) = \phi(\x) \quad  \text{with}  \quad \phi(\x) = \sum_{i=1}^d  x_i^{-\gamma_i} \ind_{\{ x_j =0 \, \forall j\ne i\}} \, .\]
This shows that $\bbP(\bX_1 > \x)$ is regularly varying with index~$-(\gamma_1,\ldots,\gamma_d)$.

\begin{example}
\label{ex:depend}
There exist positive exponents $\gb,\, (\gb_i)_{1\le i\le d}$, and  slowly varying functions $\psi, \, (\psi_i)_{1\le i\le d}$ such that 
\begin{equation}
\bP(\bX_1 \ge  \x) =  \psi\Big(\sum_{i=1}^d \psi_i(x_i) x_i^{\beta_i} \Big)  \times  \Big(\sum_{i=1}^d \psi_i(x_i) x_i^{\beta_i} \Big)^{-\beta}  \, , \quad \text{for } \x\in\bbN^d \, .
\end{equation}
We have that $\bar F_i(x)$ is regularly varying with index $-\gamma_i := -\beta \beta_i$, and $\bP( \bX_1 \geq  \x)$ is regularly varying at $+\infty$ with index $-(\gamma_1, \ldots, \gamma_d)$.
\end{example}
Again, setting $r_i(t)$ such that $\bar F_i( r_i(t) ) \sim t^{-1}$ as $t\to+\infty$, we have that $r(t)$ is regularly varying with index $(\gamma_1^{-1}, \ldots, \gamma_d^{-1})$.  For any $\x \in \bbR^d \setminus \{0\}$, we then have that
\[ \lim_{t\to\infty} t \, \bP\big( \bX_1 > r(t) \x \big) = \phi(\x) \qquad \text{with } \phi(\x) = \Big( \sum_{i=1}^d  x_i^{\beta_i} \Big)^{-\beta} \, . \]
This shows that $\bbP(\bX_1>\x)$ is regularly varying with index~$-(\gamma_1,\ldots,\gamma_d)$.


\section{Comments on Assumption~\ref{hyp:2}}
\label{appB}

Let us  comment on conditions \eqref{cond:h} in Assumption~\ref{hyp:2}, and in particular 
on the summability of $h_{x_i}^{(i)}(|x_j|) /(1+|x_j|)$.
Indeed, in view of the discussion below Assumption~\ref{hyp:2}, a natural idea would be to simply bound $\bP(X_1^{(j)} \in [x_j,2x_j]\,  \forall j\neq i \mid X_1^{(i)} \in [x_i,2x_i])$ by~$1$, so items (ii)-(iii) in \eqref{cond:h} would not be necessary.
However, by doing so, one would not be able to derive the bound \eqref{hyp:Doney1} for each coordinate (there would be an extra factor $(\log x_i)^{d-1}$), but having \eqref{hyp:Doney1} turns out to be essential in our study.
Similarly, condition (iii) in \eqref{cond:h} may appear artificial, but it is here for technical purposes, in order to be able to bound $h_u^{(i)}(v)$ uniformly for $u$ in some interval (recall the proof of Theorem~\ref{thm:locallimit2} in Section~\ref{sec:localharder}). 
We want to stress here that the condition \eqref{cond:h} is actually very weak, and is verified in natural examples that come to mind.

\subsection*{About Example~\ref{ex2}}

We now show that Assumption~\ref{hyp:2} is verified in the case of Example~\ref{ex2}. For simplicity, we present calculations in the case $d=2$ without slowly varying function ($\psi$ is a constant):  for $\x\in \bbN^2$, 
$\bP(\bX_1= \x ) = c_0 \big( x_1^{\gb_1} + x_2^{\gb_2} \big)^{-\gb}$, with $\gb_1,\gb_2>0$ and $\gb> \gb_1^{-1} +\gb_2^{-1}$, for some constant $c_0$.
We can write, setting $\gamma_1= \gb_1 (\gb - \gb_1^{-1} -\gb_2^{-1})$,
\begin{align}
 \label{rewriteex2}
 \bP \big( \bX_1= (x_1,x_2) \big) & = c_0  x_1^{- (1+\gamma_1)}   x_1^{-\gb_1/\gb_2} \Big ( 1+ \frac{x_2^{\gb_2}}{x_1^{\gb_1}}  \Big)^{-\gb}     \\
 & =  \frac{c_0}{x_2} \  x_1^{- (1+\gamma_1)}  \Big( \frac{x_2^{\gb_2}}{x_1^{\gb_1}} \Big)^{1/\gb_2} \Big ( 1+ \frac{x_2^{\gb_2}}{x_1^{\gb_1}}  \Big)^{-\gb}
\notag
\end{align}
Hence, we have the bound \eqref{eq:hyp2}, with
\begin{equation}
\label{realh}
h^{(1)}_u(v) =\frac{v}{u^{\gb_1 /\gb_2}}  \Big ( 1+  \Big(\frac{v}{u^{\gb_1/\gb_2}} \Big)^{\gb_2}   \Big)^{-\gb} \leq \min \Big\{   \frac{ v}{u^{\gb_1/\gb_2}} ,  \Big( \frac{v}{u^{\gb_1/\gb_2}}\Big)^{1-\gb\gb_2} \Big\},
\end{equation}
the last inequality coming from considering whether $v$ is smaller or larger than $u^{\gb_1/\gb_2}$.
It remains to verify that $h_u^{(1)}(v)$ verify \eqref{cond:h}: for item (i), one easily verifies that $h_u^{(1)}(v)$ is bounded by $1$, since $\gb>1/\gb_2$; item (iii) is also trivial. For item (ii), we sum over $v$ depending on whether $v\leq u^{\gb_1/\gb_2}$ or $v>  u^{\gb_1/\gb_2}$
\begin{align*}
\sum_{v\geq 1} \frac{ h^{(1)}_u(v) }{v} \leq \sum_{v=1}^{u^{\gb_1/\gb_2} }  \frac{1}{ u^{\gb_1/\gb_2}} + \frac{1}{u^{(1 -\gb \gb_2) \gb_1/\gb_2}} \sum_{v >u^{\gb_1/\gb_2} }    v^{ -\gb \gb_2 } 
\leq 1+ c\, ,
\end{align*}
and the constant does not depend on $u$ (we have that  $\sum_{v >u^{\gb_1/\gb_2} }    v^{ -\gb \gb_2 } \leq c (u^{\gb_1/\gb_2})^{1-\gb \gb_2}$ since $\gb\gb_2>1$).
Similarly, we have $ \bP \big( \bX_1= (x_1,x_2) \big) = \frac{c_0}{x_1} x_2^{-(1+\gamma_2)} h_{x_2}^{(2)}(x_1)$, with $\gamma_2:=\gb_2 (\gb - \gb_1^{-1}-\gb_2^{-1})$ and  $h_u^{(2)}(v)$ as defined in \eqref{realh} but with $\gb_1$ and $\gb_2$ swapped.

Moreover, using a Riemann sum approximation, we get that 
$\sum_{v\geq 1} h^{(1)}_u(v)/v$ converges to $\int_{0}^{\infty} t (1+t^{\gb_2})^{-\gb}  \mathrm{d}t$ as $u\to +\infty$ (recall~\eqref{realh}).
Going back to \eqref{rewriteex2} and summing over $x_2$, we therefore get that $\bP(X_1^{(1)} =x_1) \sim c_1 x_1^{-(1+\gamma_1)}$ as $x_1\to+\infty$.


This can be generalized to the setting of Example~\ref{ex2}: we get that Assumption~\ref{hyp:2} is verified, and we find that there is a constant $c_2$ (depending only on $\gb,\gb_i$) such that
\begin{equation}
\label{eq:ex1}
\bP (X_1^{(i)} =x_i)   \sim c_2 \, \psi( x_i^{\gb_i}) \times (x_i)^{- (1+\gamma_i)} \qquad  \text{as } x_i\to+\infty\, , 
\end{equation}
with $\gamma_i :=\gb_i \big( \gb - \sum_{i=1}^d \gb_i^{-1} \big)$.
Details are left to the reader.
\end{appendix}

\end{document}